\documentclass{amsart}
\pdfoutput=1
\usepackage[utf8]{inputenc}
\usepackage[T1]{fontenc}
\usepackage{lmodern}
\usepackage{amssymb,amsxtra}
\usepackage{graphicx}
\usepackage{mathabx}
\usepackage{nicefrac,mathtools}
\usepackage[lite]{amsrefs}

\renewcommand{\PrintDOI}[1]{\href{http://dx.doi.org/\detokenize{#1}}{doi: \detokenize{#1}}}

\BibSpec{book}{%
  +{}  {\PrintPrimary}                {transition}
  +{,} { \textit}                     {title}
  +{.} { }                            {part}
  +{:} { \textit}                     {subtitle}
  +{,} { \PrintEdition}               {edition}
  +{}  { \PrintEditorsB}              {editor}
  +{,} { \PrintTranslatorsC}          {translator}
  +{,} { \PrintContributions}         {contribution}
  +{,} { }                            {series}
  +{,} { \voltext}                    {volume}
  +{,} { }                            {publisher}
  +{,} { }                            {organization}
  +{,} { }                            {address}
  +{,} { \PrintDateB}                 {date}
  +{,} { }                            {status}
  +{}  { \parenthesize}               {language}
  +{}  { \PrintTranslation}           {translation}
  +{;} { \PrintReprint}               {reprint}
  +{.} { }                            {note}
  +{.} {}                             {transition}
  +{} { \PrintDOI}                   {doi}
  +{} { available at \url}            {eprint}
  +{}  {\SentenceSpace \PrintReviews} {review}
}

\usepackage{enumitem} %<--advanced enumerate, which handles label and ref
\setlist[enumerate,1]{label=\textup{(\arabic*)}}% ensure enumerates in theorems are upright

\usepackage{tikz}
\usetikzlibrary{matrix,calc,arrows,decorations.pathmorphing}
\tikzset{node distance=2cm, auto}

\tikzset{cd/.style=matrix of math nodes,row sep=2em,column sep=2em, text height=1.5ex, text depth=0.5ex}
\tikzset{cdar/.style=->,auto}
\tikzset{mid/.style={anchor=mid}} % put labels on the arrow
\tikzset{narrowfill/.style={inner sep=1pt, fill=white}}% style for nodes with filled background
\usepackage{microtype}

\usepackage[pdftitle={On the classification of group actions on C*-algebras up to equivariant KK-equivalence},
pdfauthor={Ralf Meyer},
pdfsubject={Mathematics}
]{hyperref}
\usepackage{tikz-cd}

\numberwithin{equation}{section}
\theoremstyle{plain}
\newtheorem{theorem}[equation]{Theorem}
\newtheorem{lemma}[equation]{Lemma}

\newtheorem{proposition}[equation]{Proposition}
\newtheorem{corollary}[equation]{Corollary}
\theoremstyle{definition}
\newtheorem{definition}[equation]{Definition}
\theoremstyle{remark}
\newtheorem{remark}[equation]{Remark}
\newtheorem{example}[equation]{Example}

\DeclarePairedDelimiter{\abs}{\lvert}{\rvert}% absolute value
\DeclarePairedDelimiterX{\braket}[2]{\langle}{\rangle}{#1\,\delimsize\vert\,\mathopen{}#2}% inner product
\DeclarePairedDelimiterX{\braketop}[3]{\langle}{\rangle}{#1\,\delimsize\vert #2\delimsize\vert\,\mathopen{}#3}% inner product with operator
\DeclarePairedDelimiterX{\BRAKET}[2]{\langle}{\rangle}{\!\delimsize\langle#1\,\delimsize\vert\,\mathopen{}#2\delimsize\rangle\!}% inner product
\DeclarePairedDelimiterX{\setgiven}[2]{\{}{\}}{#1\,{:}\,\mathopen{}#2}% set given by

\DeclareMathOperator{\Ad}{Ad}
\DeclareMathOperator{\Aut}{Aut}
\DeclareMathOperator{\cone}{cone}

\newcommand{\ima}{\mathrm i}

\newcommand{\diff}{\mathrm d}
\newcommand{\C}{\mathbb C}
\newcommand{\Z}{\mathbb Z}
\newcommand{\R}{\mathbb R}
\newcommand{\Q}{\mathbb Q}
\newcommand{\N}{\mathbb N}
\newcommand{\T}{\mathbb T}% circle group

\newcommand{\Comp}{\mathbb K}
\newcommand{\Bound}{\mathbb B}

\newcommand{\Hils}{\mathcal H}%Hilbert space
\DeclareMathOperator{\Hom}{Hom}% Hom-functor
\DeclareMathOperator{\Ext}{Ext}% Ext-functor

\newcommand{\nb}{\nobreakdash}
\newcommand{\Star}{\(^*\)\nobreakdash-}
\newcommand{\blank}{{\ldots}}
\newcommand{\Cst}{\mathrm C^*}% C*-algebra
\newcommand{\Mult}{\mathcal M}% multiplier algebra
\newcommand{\K}{\mathrm K}% K-theory
\newcommand{\KK}{\mathrm{KK}}% K-theory
\newcommand{\Cont}{\mathrm C}% continuous functions
\newcommand{\Contb}{\mathrm C_\mathrm b}% bounded continuous functions
\newcommand{\Contc}{\Cont_\mathrm c} % compactly supporte continuous functions

\newcommand{\id}{\mathrm{id}}% identity map
\newcommand{\op}{{\mathrm{op}}}% opposite
\newcommand{\pt}{\mathrm{pt}}% one-point space

\newcommand{\BS}{\mathrm{BS}}% Baaj-Skandalis duality
\newcommand{\defeq}{\mathrel{\vcentcolon=}}% used for definitions
\newcommand{\congto}{\xrightarrow\sim}
\newcommand{\into}{\rightarrowtail}% inflation
\newcommand{\prto}{\twoheadrightarrow}% deflation
\DeclareMathOperator{\im}{im}% image
\DeclareMathOperator{\coker}{coker}% cokernel

\newcommand{\Ga}[2]{\alpha_{#1#2}}% generator alpha
\newcommand{\Gd}[2]{\dot\alpha_{#1#2}}% generator dotted alpha
\newcommand{\Gt}[1]{t_{#1}}% generator t
\newcommand{\Gs}[1]{s_{#1}}% generator s
\newcommand{\Gu}[1]{1_{#1}}% generator 1

\newcommand{\littlering}{\mathfrak{T}}% ring that acts on little invariant
\newcommand{\Kring}{\mathfrak{K}}% Köhler's ring
\newcommand{\Sring}{\mathfrak{S}}% group ring and variant

\begin{document}

\title[Classification of group actions on C*-algebras]{On the classification of group actions on C*-algebras up to equivariant KK-equivalence}
\author{Ralf Meyer}
\email{rmeyer2@uni-goettingen.de}
\address{Mathematisches Institut\\
  Georg-August Universität Göttingen\\
  Bunsenstraße 3--5\\
  37073 Göttingen\\
  Germany}


\begin{abstract}
  We study the classification of group actions on \(\Cst\)\nb-algebras
  up to equivariant KK-equivalence.  We show that any group action is
  equivariantly KK-equivalent to an action on a simple, purely
  infinite C*-algebra.  We show that a conjecture of Izumi is
  equivalent to an equivalence between cocycle conjugacy and
  equivariant KK-equivalence for actions of torsion-free amenable
  groups on Kirchberg algebras.  Let~\(G\)
  be a cyclic group of prime order.  We describe its actions up to
  equivariant KK-equivalence, based on previous work by Manuel
  K\"ohler.  In particular, we classify actions of~\(G\)
  on stabilised Cuntz algebras in the equivariant bootstrap class up
  to equivariant KK-equivalence.
\end{abstract}
\maketitle

\section{Introduction}
\label{sec:introduction}

The Kirchberg--Phillips Classification Theorem says that two purely
infinite, simple, separable, nuclear \(\Cst\)\nb-algebras
-- briefly called Kirchberg algebras -- are isomorphic if and only if
they are KK-equivalent (see~\cite{Phillips:Classification}).
Kirchberg~\cite{Kirchberg:Michael} has generalised this theorem to
non-simple \(\Cst\)\nb-algebras
that are separable, nuclear and ``strongly'' purely infinite.  To be
isomorphic, they must have homeomorphic primitive ideal spaces.
Assume this and fix such a homeomorphism.  This lifts to an
isomorphism of \(\Cst\)\nb-algebras
if and only if the \(\Cst\)\nb-algebras
are equivalent in a suitable equivariant KK-theory, which takes the
underlying primitive ideal space into account.  To apply these
theorems, we still need a criterion when two \(\Cst\)\nb-algebras
(with the same primitive ideal space) are equivariantly KK-equivalent.

The best criterion in the simple case follows from the Universal
Coefficient Theorem: if both \(\Cst\)\nb-algebras belong to the
bootstrap class -- the class of \(\Cst\)\nb-algebras for which the
Universal Coefficient Theorem holds -- then any isomorphism of
\(\Z/2\)\nb-graded Abelian groups between their \(\K\)\nb-theories
lifts to a \(\KK\)-equivalence.  It can also be shown that any
\(\Z/2\)\nb-graded countable Abelian groups may be realised by some
Kirchberg algebra.  So isomorphism classes of stable Kirchberg
algebras in the bootstrap class are in bijection with isomorphism
classes of \(\Z/2\)\nb-graded countable Abelian groups.  There are
generalisations of this result to the non-simple case as well (see
\cites{Meyer-Nest:Filtrated_K, Bentmann:Thesis,
  Bentmann-Koehler:UCT, Bentmann-Meyer:More_general,
  Meyer:Compute_obstruction}).  But the results are not as complete,
and there is no evidence for a general classification result for
non-simple Kirchberg algebras.

There has been a lot of work aiming at a classification of suitable
group actions on \(\Cst\)\nb-algebras.  This theory differs from the
corresponding classification of group actions on von Neumann factors
because of several \(\K\)\nb-theoretic obstructions.  These are
particularly complex for finite groups.  In this article, we shall
focus on these \(\K\)\nb-theoretic issues by aiming only at a
classification up to equivariant \(\KK\)-equivalence.  Already in
the classification of non-simple Kirchberg algebras, the problem may
be decomposed into two aspects.  On the one hand, there is the
problem of lifting an equivalence in a suitable KK-theory to an
isomorphism.  On the other hand, there is the problem of detecting a
KK-equivalence by simpler invariants.  Whereas the first problem is
analytic in nature and depends on assumptions such as pure
infiniteness and nuclearity, the second problem is algebraic
topology and may be addressed using homological algebra in
triangulated categories.  Separating these two aspects of the
classification problem seems particularly helpful in the more
complicated case of group actions.

Let~\(G\) be a locally compact group.  The analytic question here is
under which conditions a \(\KK^G\)-equivalence between two
\(G\)\nb-actions \((A,\alpha)\) and~\((B,\beta)\) on separable
\(\Cst\)\nb-algebras lifts to an equivariant Morita--Rieffel
equivalence; this becomes a cocycle conjugacy between the group
actions if \(A\) and~\(B\) are stable.  Such a result is only
plausible if \(A\) and~\(B\) are Kirchberg algebras and the actions
are pointwise outer.  It also seems necessary to assume the
group~\(G\) to be amenable.  A recent result by
Szabó~\cite{Szabo:Equivariant_Absorption} shows that two pointwise
outer actions of an amenable group on unital Kirchberg algebras are
cocycle conjugate once they are \(G\)\nb-equivariantly homotopy
equivalent.  Equivariant homotopy is a much more flexible notion
than cocycle conjugacy.  So it seems conceivable that
\(\KK^G\)-equivalent outer actions on Kirchberg algebras are
cocycle conjugate if~\(G\) is amenable.

The existing classification theorems for group actions on Kirchberg
algebras use direct methods and mix the analytic and topological
aspects of the problem.  The starting point here is Nakamura's result
for single automorphisms in~\cite{Nakamura:Aperiodic}.  Similar
results have been shown for certain actions of \(\Z^2\)
and~\(\Z^N\)
(see \cites{Izumi-Matui:Z2-actions, Matui:Outer_ZN_O2}).
Izumi~\cite{Izumi:Group_actions} has conjectured a classification
result for actions of torsion-free groups, and recent preprints by
Izumi and Matui prove this conjecture for poly-\(\Z\)-groups (see
\cites{Izumi-Matui:Poly-Z,Izumi-Matui:Poly-Z_II}).
Several results in this direction are
still limited by assumptions that reduce the complexity of the
\(\K\)\nb-theoretic
invariants.  One of our results here shows that Izumi's conjecture
in~\cite{Izumi:Group_actions} is equivalent to an analogue of
Kirchberg's Classification Theorem for non-simple
\(\Cst\)\nb-algebras.
More precisely, let~\(A\)
be a Kirchberg algebra and~\(G\)
a torsion-free, amenable, second countable, locally compact group.  We
prove a bijection between \(\KK^G\)-equivalence
classes of actions of~\(G\)
on~\(A\)
and isomorphism classes of principal \(\Aut(A)\)-bundles
over the classifying space~\(B G\).
If~\(G\)
is discrete, Izumi conjectures that the latter classify outer
\(G\)\nb-actions
on~\(A\)
up to cocycle conjugacy, not just up to \(\KK^G\)-equivalence.
So Izumi's conjecture is equivalent to the conjecture that two
\(\KK^G\)-equivalent
outer \(G\)\nb-actions on~\(A\) are already cocycle conjugate.

For finite groups or, more generally, groups with torsion, the
existing classification results need extra assumptions on the group
actions such as the Rokhlin property (see
\cites{Izumi:Finite_group_I, Izumi:Finite_group}).  The Rokhlin
property is a severe restriction that drastically simplifies the
\(\K\)\nb-theory groups related to the group action.  Here we study
actions without the Rokhlin property, but only for the group
\(G=\Z/p\) for a
prime number~\(p\).  Manuel Köhler~\cite{Koehler:Thesis} has proven
a Universal Coefficient Theorem for~\(\KK^G\) for these
groups~\(G\).  It implies a classification of objects in the
equivariant bootstrap class in~\(\KK^G\) up to
\(\KK^G\)-equivalence, as in the classical Kirchberg--Phillips
Theorem.

The equivariant bootstrap class~\(\mathfrak{B}^G\)
in~\(\KK^G\)
is defined as the localising subcategory generated by \(\Cont(G)\)
with the free translation action and by~\(\C\).
That is, it is the smallest class of separable
\(G\)\nb-\(\Cst\)-algebras
that contains \(\C\)
and \(\Cont(G)\)
and is closed under \(\KK^G\)-equivalence,
direct sum and suspensions and has the two-out-of-three property for
extensions of \(G\)\nb-\(\Cst\)-algebras
with a \(G\)\nb-equivariant
completely positive contractive section.  The equivariant bootstrap
class~\(\mathfrak{B}^G\)
is equal to the class of all \(G\)\nb-actions
that are \(\KK^G\)-equivalent
to an action on a Type~I \(\Cst\)\nb-algebra
(see \cite{dellAmbrogio-Emerson-Meyer:Equivariant_Lefschetz}*{§3.1}).
It is unclear whether we can replace ``Type~I'' by ``commutative''
here.  For the group \(\T=\R/\Z\),
however, Emerson has found an action in~\(\mathfrak{B}^\T\)
that cannot be \(\KK^\T\)-equivalent
to an action on a commutative \(\Cst\)\nb-algebra.
If \(A\in\mathfrak{B}^G\),
then~\(A\)
and \(A\rtimes G\)
belong to the usual bootstrap class in~\(\KK\).
We do not know whether the converse of this is true.

For an action of \(G=\Z/p\)
on a general separable \(\Cst\)\nb-algebra~\(A\),
Köhler's invariant is the direct sum
\[
F(A)\defeq \K_*(A\rtimes G) \oplus \K_*(A) \oplus \KK^G_*(D,A),
\]
where~\(D\) is the mapping cone of the unit map \(\C \hookrightarrow
\Cont(G)\).
The natural transformations \(F\Rightarrow F\) form a
\(\Z/2\)\nb-graded ring~\(\Kring\), and \(F(A)\) is a countable
\(\Z/2\)\nb-graded \(\Kring\)\nb-module.  This module structure is
part of Köhler's invariant.  Let \(A\) and~\(B\) be objects
of~\(\mathfrak{B}^G\).  Köhler shows that \(A\) and~\(B\) are
\(\KK^G\)-equivalent if and only if \(F(A) \cong F(B)\) as
\(\Z/2\)\nb-graded \(\Kring\)\nb-modules.  His proof also shows
which countable \(\Z/2\)\nb-graded \(\Kring\)\nb-modules belong to
the range of~\(F\).  Namely, the condition is that two chain
complexes must be exact.  One is the Puppe six-term exact sequence
for the unit map \(\C \hookrightarrow \Cont(G)\).  The other is the
Puppe sequence for the induced map
\(\Cont(G) \cong \C\rtimes G \hookrightarrow \Cont(G)\rtimes G \cong
\Comp(\ell^2(G))\).  An important point here is that the latter cone
is \(\KK^G\)-equivalent to the suspension of~\(D\).  A
\(\Z/2\)\nb-graded \(\Kring\)\nb-module is called exact if these two
Puppe sequences are exact.

So isomorphism classes of objects in the bootstrap class in~\(\KK^G\)
are in bijection with isomorphism classes of countable, exact
\(\Z/2\)\nb-graded
\(\Kring\)\nb-modules.
These are pretty complicated objects.  Let~\(M\)
be a \(\Z/2\)\nb-graded
\(\Kring\)\nb-module.
Then \(M = M_0 \oplus M_1 \oplus M_2\),
where \(M_0\), \(M_1\)
and~\(M_2\)
are the images of the idempotent natural transformations that
project~\(F(A)\)
to the three summands
\(\K_*(A\rtimes G)\),
\(\K_*(A)\),
and \(\KK^G_*(D,A)\).
The summands \(M_0\)
and~\(M_1\)
are modules over
\(\Z[G] \cong \operatorname{Rep}(G) \cong \Z[x]/(x^p-1)\),
where the actions are induced by the given action of~\(G\)
on~\(A\)
and the dual action of \(G\cong \widehat{G}\)
on \(A\rtimes G\).
The summand~\(M_2\)
is a module over the more complicated ring
\[
\Sring_2 \defeq \Z[s,t] \bigm/ \bigl((1-s)\cdot(1-t),N(s)+N(t)-p\bigr).
\]
Then there are further natural transformations that link
\(M_0\), \(M_1\) and~\(M_2\).

Since Köhler's invariant is complicated, we are going to consider
several classes of exact \(\Kring\)\nb-modules that may be
understood through simpler objects.  My goal has been to classify
those exact \(\Kring\)\nb-modules where~\(M_1\) is a cyclic group.
Given Köhler's classification theorem, these exact
\(\Kring\)\nb-modules correspond to \(G\)\nb-actions in the
bootstrap class \(A\in\mathfrak{B}^G\) where \(\K_*(A)\) is cyclic.
In fact, any such~\(A\) is \(\KK^G\)-equivalent to an action on a
stabilised Cuntz algebra.  The main issue here is to make the
underlying \(\Cst\)\nb-algebra of the action simple.  That this is
possible is proven in complete generality for actions of locally
compact groups and is also used in the torsion-free case to
reformulate Izumi's conjecture.  Since \(A\in\mathfrak{B}^G\) is
assumed, the results in this article only cover actions
in~\(\mathfrak{B}^G\), so that \(A\rtimes G\) must be in the
non-equivariant bootstrap class.  Thus we have nothing to say about
the important question whether \(A\rtimes G\) belongs to the
bootstrap class for any \(G\)\nb-action, say, on~\(\mathcal{O}_2\)
(see, for instance, \cite{Barlak-Li:Cartan_UCT}).

The end result of our classification of the actions of~\(G\) on
Cuntz algebras is remarkably complicated.  In
Section~\ref{sec:representative_examples}, we list nine examples of
such exact modules; some of these involve parameters, so that they
rather form a family of examples.  There is a tenth family of
examples only for the case \(p=2\), which is about a general method
to modify a given example to produce another one.  Our main theorem
says that any exact module where~\(M_1\) is cyclic contains one of
these ten examples as an exact submodule~\(M'\), such that the
quotient \(M'' \defeq M/M'\) has \(M''_1=0\).  We also show that any
such~\(M''\) is uniquely \(p\)\nb-divisible and may be described
through a \(\Z/2\)\nb-graded module over \(\Z[\vartheta_p,1/p]\).
The examples in Section~\ref{sec:representative_examples} are all
quite different, and some of them are not obvious.  I have found
these examples by a case-by-case study of the possible exact modules
over Köhler's ring with cyclic~\(M_1\).  The results of Köhler also
imply that the module extension \(M' \into M \prto M''\) in our main
theorem lifts to an exact triangle in the bootstrap class
in~\(\KK^G\).  This explains why the piece~\(M''\) has to be there.
It corresponds to adding~\(\mathcal{O}_2\) with a suitable action to
a given \(G\)\nb-action and then passing to a \(\KK^G\)-equivalent
action on a Kirchberg algebra.

In order to classify the exact \(\Kring\)\nb-modules with
cyclic~\(M_1\), we prove several more general results.  First, we
describe Köhler's ring through generators and relations.  This
simplifies writing down a \(\Kring\)\nb-module.  As it turns out, a
module over Köhler's ring consists of three \(\Z/2\)\nb-graded
Abelian groups~\(M_j\), \(j=0,1,2\), with
homomorphisms~\(\Ga{j}{k}^M\colon M_k \to M_j\) for \(j\neq k\) such
that the two triangle-shaped diagrams
\[
\begin{tikzcd}[row sep=large, column sep=large]
  &M_1 \ar[dr, "\Ga{2}{1}^M"]&\\
  M_0 \ar[ur, "\Ga{1}{0}^M"]&&
  M_2 \ar[ll, "\Ga{0}{2}^M"]
\end{tikzcd}\qquad
\begin{tikzcd}[row sep=large, column sep=large]
  &M_1 \ar[dl, "\Ga{0}{1}^M"']&\\
  M_0 \ar[rr, "\Ga{2}{0}^M"']&&
  M_2 \ar[ul, "\Ga{1}{2}^M"']
\end{tikzcd}
\]
are chain complexes.  Exactness means that these two chain complexes
are exact.  There are also some additional relations involving the
products \(\Ga{j}{k}^M \circ \Ga{k}{j}^M\) (see
Theorem~\ref{the:generators_and_relations}).  Köhler's invariant for
a \(G\)\nb-\(\Cst\)-algebra~\(A\) has \(M_0 = \K_*(A\rtimes G)\) and
\(M_1 = \K_*(A)\), and the maps \(\Ga{0}{1}^M\) and \(\Ga{1}{0}^M\)
are induced by the inclusion maps \(A\subseteq A\rtimes G\) and
\(A\rtimes G \subseteq A\otimes \Comp(\ell^2 G)\).  We show by an
example that \(M_0\) and~\(M_1\) alone with their natural extra
structure do not yet classify up to \(\KK^G\)-equivalence, by
exhibiting two non-isomorphic exact modules over K\"ohler's ring for
which the \(M_0\)- and \(M_1\)-parts are isomorphic (see
Example~\ref{exa:actions_on_Cuntz_4}).

We describe several classes of exact modules over Köhler's ring using
simpler ingredients.  First, we describe modules where one of the
groups~\(M_j\)
is uniquely \(p\)\nb-divisible
(see Example~\ref{exa:divisible}); in particular, this covers the case
when one of them vanishes, say, \(M_1=0\).
In the latter case, \(M_0 \cong M_2\)
is a \(\Z/2\)\nb-graded
module over \(\Z[\zeta_p,1/p]\)
with the primitive \(p\)th
root of unity \(\zeta_p \defeq \exp(2\pi\ima /p)\),
and any such \(\Z/2\)\nb-graded
module corresponds to a unique exact \(\Kring\)\nb-module
with \(M_1=0\).
Then we describe modules where one of the maps~\(\Ga{j}{k}^M\)
is zero.  By exactness, this covers all cases where \(\Ga{0}{1}^M\)
or \(\Ga{1}{0}^M\)
is zero, injective, or surjective.  Third, we describe modules with
\(M_0 = \im \Ga{0}{1}^M \oplus \im \Ga{0}{2}^M\)
or similar.  As it turns out, these different cases suffice to
describe all exact \(\Kring\)\nb-modules
where~\(M_1\)
is a cyclic group.

The starting point for the study of exact modules over Köhler's ring
here is the Bachelor's thesis of Vincent
Grande~\cite{Grande:Bachelor}, which contains several examples of
such modules and partial results on the different cases of
submodules considered here.

It is a great pleasure to thank the referee for a very careful
reading of the manuscript and a number of helpful suggestions.

\section{Realising actions on Kirchberg algebras}
\label{sec:actions_Kirchberg}

The goal of this section is the following theorem:

\begin{theorem}
  \label{the:make_actions_simple}
  Let~\(G\)
  be a second countable, locally compact group.  Let~\(A\)
  be a separable \(\Cst\)\nb-algebra
  and let~\(\alpha\)
  be a continuous action of~\(G\)
  on~\(A\).
  Then there is a non-zero, simple, purely infinite, stable
  \(\Cst\)\nb-algebra~\(B\)
  with a pointwise outer, continuous action~\(\beta\)
  of~\(G\)
  such that \((A,\alpha)\)
  is \(\KK^G\)-equivalent
  to \((B,\beta)\).  If~\(A\) is nuclear or exact, then so is~\(B\).
\end{theorem}

\begin{proof}
  The
  proof of~\cite{Benson-Kumjian-Phillips:Symmetries_Kirchberg}*{Proposition~4.2}
  shows this, but the theorem makes only a weaker statement.  So we go
  through the main construction again.  The main tool is the Toeplitz
  algebra for \(\Cst\)\nb-correspondences
  defined by Pimsner~\cite{Pimsner:Generalizing_Cuntz-Krieger}, in the
  special case studied by Kumjian~\cite{Kumjian:Cuntz-Pimsner}.  A
  similar construction appeared recently
  in~\cite{Deaconu:Pimsner_representation}.

  Before we really start, we need an extra step for the trivial case
  \(A=0\).  Then we replace~\(A\) by~\(\mathcal{O}_2\) with the
  trivial action of~\(G\).  This is non-zero, nuclear and
  \(\KK^G\)\nb-equivalent to~\(0\).  So we may assume without loss
  of generality that \(A\neq0\).
  
  First we build a regular covariant representation of
  \((A,\alpha,G)\)
  as in \cite{Pedersen:Cstar_automorphisms}*{§7.7.1}.  There is a
  faithful representation \(\pi\colon A\to\Bound(\Hils)\)
  of~\(A\)
  on a separable Hilbert space~\(\Hils\).
  Taking an infinite direct sum of copies of~\(\pi\),
  we arrange that \(\pi(A)\)
  contains no compact operators.  Now represent~\(G\)
  and the \(\Cst\)\nb-algebra~\(A\)
  on \(L^2(G,\Hils)\)
  by the left regular representation
  \((\lambda_g f)(x) \defeq f(g^{-1} x)\)
  and by \((\tilde\pi(a) f)(x) \defeq \pi(\alpha_{x^{-1}}(a)) f(x)\)
  for all \(g,x\in G\),
  \(a\in A\),
  \(f\in L^2(G,\Hils)\).
  This is a covariant representation, that is,
  \(\lambda_g \tilde\pi(a) \lambda_{g^{-1}} = \tilde\pi(\alpha_g(a))\)
  for all \(g\in G\),
  \(a\in A\).

  Let~\(E\)
  be the Hilbert \(A\)\nb-module
  \(L^2(G,\Hils)\otimes A\),
  equipped with the left action \(\varphi\colon A\to \Bound(E)\),
  \(a\mapsto \tilde\pi(a)\otimes1\),
  and the diagonal \(G\)\nb-action
  \(\gamma_g (f\otimes a) \defeq \lambda_g(f) \otimes \alpha_g(a)\).
  This makes~\(E\)
  a \(G\)\nb-equivariant
  \(\Cst\)\nb-correspondence
  from~\(A\)
  to itself, that is, the left action is by adjointable operators and
  the \(G\)\nb-action satisfies
  \[
  \gamma_g(a\cdot \xi) = \alpha_g(a) \cdot \gamma_g(\xi),\qquad
  \gamma_g(\xi\cdot a) = \gamma_g(\xi) \cdot \alpha_g(a),\qquad
  \alpha_g(\braket{\xi}{\eta}) = \braket{\gamma_g(\xi)}{\gamma_g(\eta)}
  \]
  for all \(g\in G\),
  \(a\in A\),
  \(\xi,\eta\in E\).
  If \(\varphi(a)\)
  is compact, then~\(\pi(a)\)
  must be compact and then \(a=0\)
  by construction.  In particular, the left action of~\(A\)
  on~\(E\) is injective.

  Let~\(\mathcal{T}_E\)
  be the Toeplitz \(\Cst\)\nb-algebra
  of~\(E\)
  as in~\cite{Pimsner:Generalizing_Cuntz-Krieger}.  The
  \(G\)\nb-actions
  on \(A\)
  and~\(E\)
  satisfy the right compatibility assumptions to induce a continuous
  action of~\(G\)
  on~\(\mathcal{T}_E\)
  (see \cite{Pimsner:Generalizing_Cuntz-Krieger}*{Remark~4.10.(2)}).
  The natural inclusion of~\(A\)
  into~\(\mathcal{T}_E\)
  is a KK-equivalence by
  \cite{Pimsner:Generalizing_Cuntz-Krieger}*{Theorem~4.4}, and this
  remains true in~\(\KK^G\)
  because all Kasparov cycles and homotopies between them used in the
  proof of that theorem are defined naturally and hence are exactly
  \(G\)\nb-equivariant.
  This is already observed by Pimsner in
  \cite{Pimsner:Generalizing_Cuntz-Krieger}*{Remark~4.10.(2)}.

  Since no non-zero element of~\(A\)
  acts on~\(E\)
  by a compact operator, \cite{Kumjian:Cuntz-Pimsner}*{Theorem~2.8}
  shows that \(\mathcal{T}_E = \mathcal{O}_E\)
  is purely infinite and simple.  To make it stable, we
  replace~\(\mathcal{T}_E\)
  by its \(\Cst\)\nb-stabilisation
  \(\mathcal{T}_E \otimes \Comp(\ell^2\N)\),
  which remains purely infinite and simple and \(\KK^G\)-equivalent
  to~\(A\).
  The structural analysis of the Toeplitz algebra in
  \cite{Pimsner:Generalizing_Cuntz-Krieger} or
  \cite{Kumjian:Cuntz-Pimsner} shows that it inherits both nuclearity
  and exactness from~\(A\).

  Finally, we show that the induced action~\(\beta\)
  of~\(G\)
  on~\(\mathcal{T}_E\)
  is pointwise outer (the \(\Cst\)\nb-stabilisation
  does not affect this).  Assume the contrary.  So there are
  \(g\in G\setminus\{1\}\)
  and a unitary multiplier~\(U\)
  of~\(B\)
  with \(\beta_g = \Ad_U\).
  Let \(\gamma\colon \T\to\Aut(\mathcal{T}_E)\)
  be the gauge action on the Toeplitz algebra~\(\mathcal{T}_E\).
  It commutes with the action of~\(G\).  Therefore,
  \[
  \Ad_{\gamma_z(U)}
  = \gamma_z \beta_g \gamma_z^{-1}
  = \beta_g
  = \Ad_U.
  \]
  This implies that conjugation by the unitary~\(\gamma_z(U) U^*\)
  is the identity automorphism.  That is, this unitary is central.
  Since~\(\mathcal{T}_E\)
  is simple, \(\gamma_z(U) = c_z\cdot U\)
  for some scalar \(c_z\in \C\)
  with \(\abs{c_z}=1\).
  The map \(z\mapsto c_z\)
  is a character on~\(\T\).
  Since the dual group of~\(\T\)
  is~\(\Z\),
  there is \(n\in\Z\)
  with \(\gamma_z(U) = z^n\cdot U\)
  for all \(z\in\T\).
  Recall that the Fock representation of~\(\mathcal{T}_E\)
  is a faithful and non-degenerate representation
  on the Hilbert \(A\)\nb-module
  \(\mathcal{F} \defeq \bigoplus_{k=0}^\infty E^{\otimes_A k}\).
  The unitary~\(U\)
  gives a unitary operator on~\(\mathcal{F}\).
  We describe it by a block matrix \((U_{k,l})_{k,l\in\N}\)
  in the direct sum decomposition of~\(\mathcal{F}\).
  The gauge action~\(\gamma_z\)
  on~\(\mathcal{T}_E\)
  acts on the matrix coefficients by
  \(\gamma_z(U_{k,l}) = z^{k-l} U_{k,l}\)
  for all \(k,l\in\N\).
  So the homogeneity condition \(\gamma_z(U) = z^n\cdot U\)
  says that \(U_{k,l}=0\)
  for \(k-l \neq n\).
  If \(n>0\),
  then this implies that \(k>0\),
  so that the range of~\(U\)
  is orthogonal to the first summand~\(A\)
  in~\(\mathcal{F}\).
  This is impossible for a unitary operator.
  Similarly, \(U^*\)
  cannot be unitary if \(n<0\).
  So we must have \(n=0\).

  So~\(U\)
  is a multiplier of the fixed-point subalgebra
  \(\mathcal{T}_E^\T \subseteq \mathcal{T}_E\)
  of the gauge action.  This subalgebra acts on~\(\mathcal{F}\)
  by diagonal operators.
  Projecting to the first summand~\(A\)
  in~\(\mathcal{F}\)
  is a non-degenerate representation of~\(\mathcal{T}_E^\T\).
  It maps~\(U\)
  to a unitary multiplier~\(V\)
  of~\(A\)
  with \(\Ad_V = \alpha_g\).
  Projecting to the second summand~\(E\)
  in~\(\mathcal{F}\)
  is another non-degenerate representation of~\(\mathcal{T}_E^\T\),
  which maps~\(U\)
  to a unitary operator~\(W\)
  on~\(E\).
  Conjugation by~\(W\)
  must give the restriction of~\(\beta_g\)
  on \(\Comp(E) \cong \Comp(L^2(G,\Hils))\otimes A\).
  Conjugation by \(\lambda_g \otimes V\)
  has the same effect.  So \(W = W_0\cdot (\lambda_g \otimes V)\)
  for some unitary~\(W_0\)
  in the centre of the multiplier algebra of \(\Comp(E)\).
  Such operators must be of the form
  \(W_0 = 1_{L^2(G,\Hils)}\otimes W_1\)
  for some central unitary multiplier~\(W_1\) of~\(A\).
  So \(W = \lambda_g \otimes (W_1\cdot V)\).
  The image of~\(\mathcal{T}_E^\T\)
  in~\(\Bound(E)\)
  is the linear span of \(\Comp(E)\)
  and \(\varphi(A)\),
  where \(\varphi\colon A\to\Bound(E)\)
  is the left action chosen above.
  We are going to prove that
  \begin{equation}
    \label{eq:proof_detail_1}
    W\cdot \varphi(a) \notin \varphi(A) + \Comp(E)
  \end{equation}
  for \(a\in A\), \(a\neq0\).  Then~\(W\) is not a multiplier of
  \(\varphi(A) + \Comp(E)\).  But it must be one, coming from a
  multiplier of~\(\mathcal{T}_E^\T\).  This contradiction shows that
  the automorphism~\(\beta_g\) on~\(\mathcal{T}_E\) cannot be inner.

  It remains to prove~\eqref{eq:proof_detail_1}.  Let \(a\in A\) be
  such that \(W\cdot \varphi(a) \in \varphi(A) + \Comp(E)\).  We
  are going to prove that \(a=0\).  Since \(g\neq1\), there is
  \(h\in \Contc(G)\) so that \(h\) and~\(\lambda_g h\) have disjoint
  support.  We use the creation operator
  \(T_h\colon \Hils \otimes A \to E\), \(\xi\mapsto h\otimes \xi\).
  Then \(T_{\lambda_g h}^* \varphi(A) T_h = 0\)
  because~\(\varphi(A)\) does not change the support of functions
  on~\(G\).  So
  \(T_{\lambda_g h}^* W \varphi(a) T_h \in \Comp(\Hils \otimes A)\).
  Let \(\xi\in \Hils \otimes A\).  Then
  \begin{multline*}
    T_{\lambda_g h}^* W \varphi(a) T_h(\xi) =
    \int_G \overline{h(x)} W_1 V (\pi(\alpha_x^{-1}(a)) \otimes 1_E) h(x)\xi \,\diff x
    \\= W_1 V \pi\biggl( \int_G \abs{h(x)}^2 \alpha_x^{-1}(a) \,\diff
    x\biggr) \otimes 1.
  \end{multline*}
  Since~\(W_1 V\) is unitary and \(\pi(A) \cap \Comp(\Hils) = 0\),
  this is only compact if
  \[
    \int_G \abs{h(x)}^2 \alpha_x^{-1}(a) \,\diff x = 0.
  \]
  Then \(a=0\) because the action on~\(A\) is continuous.  This
  proves the claim.
\end{proof}

There are not yet any definite classification results for general
group actions on \(\Cst\)\nb-algebras.
So it is unclear which conditions are sufficient for classification.
If a generalisation of Kirchberg's classification theorem exists, it
should probably assume a \(G\)\nb-action
on a simple, nuclear, separable \(\Cst\)\nb-algebra
that tensorially absorbs a suitable model \(G\)\nb-action
on~\(\mathcal{O}_\infty\),
which is \(\KK^G\)-equivalent
to~\(\C\)
with the trivial action.  This is analogous to the strong pure
infiniteness requirement in Kirchberg's classification theorem for
non-simple \(\Cst\)\nb-algebras,
which says that the \(\Cst\)\nb-algebra
tensorially absorbs~\(\mathcal{O}_\infty\).
The model action on~\(\mathcal{O}_\infty\)
should tensorially absorb itself, so that tensoring with it produces a
\(\KK^G\)-equivalent
object that tensorially absorbs the model action.
Such model actions are constructed in
\cite{Goldstein-Izumi:Quasi-free}*{Theorem~5.1},
\cite{Izumi-Matui:Poly-Z}*{Theorem~4.13} and
\cite{Szabo:Equivariant_Absorption}*{Corollary~3.7}.  I thank the
referee for pointing out these references.
If a classification
theorem for group actions will be proven, then one may use
Theorem~\ref{the:make_actions_simple} to show that any object
of~\(\KK^G\)
whose underlying \(\Cst\)\nb-algebra
is nuclear is \(\KK^G\)-equivalent
to a classifiable object, by tensoring the \(G\)\nb-action
from Theorem~\ref{the:make_actions_simple} with the model action.  It
may well be that the actions produced by
Theorem~\ref{the:make_actions_simple} already absorb the model action.
We do not discuss this further here because, anyway, our
classification results are only up to \(\KK^G\)-equivalence.

\section{Actions of torsion-free, amenable groups and a conjecture of Izumi}
\label{sec:torsion-free_Izumi}

Throughout this section, we assume that~\(G\)
is an amenable, second countable, Hausdorff, locally compact group.
Let~\(\mathcal{E}G\)
denote a second countable, locally compact, Hausdorff space with a
proper, continuous \(G\)\nb-action
that is the universal proper action in the following sense:
\begin{enumerate}
\item for any other locally compact, proper \(G\)\nb-space~\(Z\)
  there is a continuous \(G\)\nb-equivariant
  map \(Z \to \mathcal{E}G\), and
\item any two continuous \(G\)\nb-equivariant
  maps \(Z \rightrightarrows \mathcal{E}G\)
  are homotopic through a continuous \(G\)\nb-equivariant
  homotopy \(Z\times [0,1] \to \mathcal{E}G\).
\end{enumerate}
Such a space exists by \cite{Kasparov-Skandalis:Bolic}*{Lemma~4.1}
(take \(X=G\)).  But smaller models are preferable for actual
computations.

We will soon also assume~\(G\)
to be torsion-free, but the first step in our argument does not yet
need this.  For an amenable, torsion-free group, we will show that
\(\KK^G\)-equivalence
classes of actions of~\(G\)
on a stable Kirchberg algebra~\(A\)
are in bijection with isomorphism classes of locally trivial
\(\Aut(A)\)-bundles
over the classifying space~\(B G\).
This relates a conjecture by Izumi~\cite{Izumi:Group_actions} to
Kirchberg's Classification Theorem.  First we recall a key step in the
proof of the Baum--Connes conjecture for amenable groups by Higson and
Kasparov~\cite{Higson-Kasparov:E_and_KK}.

\begin{definition}
  \label{def:proper_G-Cstar}
  A \(G\)\nb-\(\Cst\)-algebra~\(\mathcal{P}\)
  is \emph{proper} if there are a locally compact, proper
  \(G\)\nb-space~\(X\)
  and a \(G\)\nb-equivariant,
  nondegenerate \Star{}homomorphism from \(\Cont_0(X)\)
  to the centre of the multiplier algebra of~\(\mathcal{P}\).
  In other words, \(\mathcal{P}\)
  carries an action of the transformation groupoid \(G\ltimes X\).
  Since the latter groupoid is proper, proper group actions on
  \(\Cst\)\nb-algebras behave much better than general actions.
\end{definition}

\begin{theorem}[\cite{Higson-Kasparov:E_and_KK}]
  \label{the:Higson-Kasparov}
  Let~\(G\)
  be an amenable group or, more generally, a group with the Haagerup
  approximation property.  Then there is a proper
  \(G\)\nb-\(\Cst\)-algebra~\(\mathcal{P}\)
  that is \(\KK^G\)-equivalent
  to~\(\C\).  In addition, \(\mathcal{P}\) is nuclear.
\end{theorem}

The \(\KK^G\)-equivalence
between~\(\mathcal{P}\)
and~\(\C\)
consists of two elements \(D\in \KK^G_0(\mathcal{P},\C)\),
\(\eta\in \KK^G_0(\C,\mathcal{P})\),
called Dirac and dual Dirac elements.  Their composites in both
direction are~\(1\).
This data is built in~\cite{Higson-Kasparov:E_and_KK} and then used to
prove the Baum--Connes conjecture for~\(G\).
If~\(X\)
is as in Definition~\ref{def:proper_G-Cstar}, then there is a map
\(X \to \mathcal{E}G\),
which induces a \(G\)\nb-equivariant,
nondegenerate \Star{}homomorphism
\(\Cont_0(\mathcal{E}G)\to \Mult(\Cont_0(X))\).
So we may replace~\(X\)
by~\(\mathcal{E}G\)
in Definition~\ref{def:proper_G-Cstar}.  Then~\(\mathcal{P}\)
carries an action of the transformation groupoid
\(G\ltimes \mathcal{E}G\)
as in~\cite{LeGall:KK_groupoid}.  Given a
\(G\)\nb-\(\Cst\)-algebra~\(A\),
the \(\Cst\)\nb-algebra
\(\Cont_0(\mathcal{E}G)\otimes A \cong \Cont_0(\mathcal{E}G,A)\)
carries a canonical action of \(G\ltimes \mathcal{E}G\),
using the diagonal \(G\)\nb-action
and the equivariant nondegenerate \Star{}homomorphism
\(\Cont_0(\mathcal{E}G) \to Z\Mult(\Cont_0(\mathcal{E}G)\otimes A)\),
\(f\mapsto f\otimes 1\).

The construction \(A\mapsto \Cont_0(\mathcal{E} G,A)\)
above is part of a functor
\(p_{\mathcal{E}G}^*\colon \KK^G \to \KK^{G\ltimes\mathcal{E}G}\).
This may be proven directly by writing down
what~\(p_{\mathcal{E}G}^*\)
does with Kasparov cycles.  An elegant proof using the universal
property of~\(\KK^G\)
is explained in~\cite{Meyer-Nest:BC}*{Section~3.2}.  More generally,
the same proof gives pull-back functors
\[
f^*\colon \KK^{G\ltimes X} \to \KK^{G\ltimes Y}
\]
for two locally compact \(G\)\nb-spaces
\(X\)
and~\(Y\)
and a \(G\)\nb-equivariant
continuous map \(f\colon Y\to X\);
the functor \(p_{\mathcal{E}G}^*\)
is the special case where~\(f\)
is the constant map from~\(\mathcal{E}G\) to the one-point space.

\begin{definition}
  \label{def:symmetric_algebra_over_EG}
  Let
  \(p_1,p_2\colon \mathcal{E}G \times \mathcal{E}G \rightrightarrows \mathcal{E}G\)
  be the two coordinate projections.  An object of
  \(\KK^{G\ltimes\mathcal{E}G}\)
  is called \emph{symmetric} if \(p_1^*(\mathcal{B})\)
  and \(p_2^*(\mathcal{B})\)
  are \(\KK^{G\ltimes(\mathcal{E}G\times\mathcal{E}G)}\)-equivalent.
\end{definition}

\begin{proposition}
  \label{pro:KKG_vs_KK_over_EG}
  Let~\(G\)
  be a locally compact group and \(\mathcal{E}G\)
  a universal proper \(G\)\nb-space.
  Assume that there is a nuclear, proper
  \(G\)\nb-\(\Cst\)-algebra~\(\mathcal{P}\)
  that is \(\KK^G\)-equivalent
  to~\(\C\).
  The functor \(\KK^G\to \KK^{G\ltimes\mathcal{E}G}\)
  defined above is an equivalence from~\(\KK^G\)
  onto the full subcategory of symmetric objects in
  \(\KK^{G\ltimes\mathcal{E}G}\).
  This functor restricts to an equivalence from the
  subcategory of \(G\)\nb-actions
  on nuclear \(\Cst\)\nb-algebras
  to the full subcategory of symmetric objects
  in \(\KK^{G\ltimes\mathcal{E}G}\)
  whose underlying \(\Cst\)\nb-algebra
  is nuclear.
\end{proposition}

\begin{proof}
  The proper \(G\)\nb-\(\Cst\)-algebra~\(\mathcal{P}\)
  gives rise to functor
  \[
  \KK^{G\ltimes\mathcal{E}G} \to \KK^G,\qquad
  A\mapsto A\otimes_{\mathcal{E}G} \mathcal{P}.
  \]
  On the object level this functor first forms the tensor product
  \(A\otimes \mathcal{P}\),
  which is a \(\Cst\)\nb-algebra
  over \(\mathcal{E}G\times\mathcal{E}G\),
  and then restricts this to the diagonal
  \(\setgiven{(x,x)}{x\in\mathcal{E}G}\).
  Then the \(\Cont_0(\mathcal{E}G)\)-\(\Cst\)\nb-algebra
  structure is forgotten.  The group action of~\(G\)
  goes along in these constructions.  The universal property
  of~\(\KK^{G\ltimes X}\)
  shows that it defines a functor on Kasparov theory
  (see~\cite{Meyer-Nest:BC}*{Section~3.2}).  Both functors
  \(\KK^G \leftrightarrow \KK^{G\ltimes\mathcal{E}G}\)
  preserve nuclearity of \(\Cst\)\nb-algebras.
  The composite functor \(\KK^G \to \KK^G\)
  maps a \(G\)-\(\Cst\)-algebra~\(A\)
  to \(\Cont_0(\mathcal{E}G,A) \otimes_{\mathcal{E}G}\mathcal{P}\),
  which is canonically isomorphic to \(A\otimes \mathcal{P}\).
  And this is \(\KK^G\)-equivalent
  to~\(A\)
  because of the \(\KK^G\)-equivalence
  between \(\mathcal{P}\)
  and~\(\C\).
  So the composite \(\KK^G \to \KK^G\)
  is equivalent to the identity functor.  Hence the other composite
  functor
  \(S\colon \KK^{G\ltimes\mathcal{E}G}\to\KK^{G\ltimes\mathcal{E}G}\)
  is a projection onto a subcategory
  of~\(\KK^{G\ltimes\mathcal{E}G}\).  Since
  \[
  p_1^*(\Cont_0(\mathcal{E}G,A)) \cong
  \Cont_0(\mathcal{E}G\times\mathcal{E}G,A)
  \cong p_2^*(\Cont_0(\mathcal{E}G,A)),
  \]
  the objects of this subcategory must be symmetric.  Conversely,
  assume that~\(\mathcal{B}\)
  is a symmetric object of~\(\KK^{G\ltimes\mathcal{E}G}\).
  The composite functor~\(S\) above maps~\(\mathcal{B}\) to
  \begin{multline}
    \label{eq:tensor_with_symmetric}
    \Cont_0(\mathcal{E}G) \otimes
    (\mathcal{B}\otimes_{\mathcal{E}G} \mathcal{P})
    \cong \Cont_0(\mathcal{E}G,\mathcal{P})
    \otimes_{\mathcal{E}G\times\mathcal{E}G}
    (\Cont_0(\mathcal{E}G) \otimes \mathcal{B})
    \\\sim_{\KK^{G\ltimes \mathcal{E}G}}
      \Cont_0(\mathcal{E}G,\mathcal{P})
    \otimes_{\mathcal{E}G\times\mathcal{E}G}
    (\mathcal{B} \otimes \Cont_0(\mathcal{E}G))
    \cong \mathcal{B} \otimes \mathcal{P}
    \sim_{\KK^{G\ltimes \mathcal{E}G}} \mathcal{B} \otimes \C
    \cong \mathcal{B}.
  \end{multline}
  All objects in~\eqref{eq:tensor_with_symmetric} carry the diagonal
  action of~\(G\)
  and the \(\Cont_0(\mathcal{E}G)\)-\(\Cst\)-algebra
  structure that acts only on the first leg, coming from the
  projection
  \(p_1\colon \mathcal{E}G\times\mathcal{E}G \to \mathcal{E}G\).
  The first \(\KK^{G\ltimes\mathcal{E}G}\)-equivalence
  in~\eqref{eq:tensor_with_symmetric} uses that~\(\mathcal{B}\)
  is symmetric.  The second \(\KK^{G\ltimes\mathcal{E}G}\)-equivalence
  uses once again that~\(\mathcal{P}\)
  is \(\KK^G\)-equivalent
  to~\(\C\).
  So all objects in the range of~\(S\)
  are symmetric, and~\(S\)
  is equivalent to the identity functor on the full subcategory of
  symmetric objects.  This finishes the proof.
\end{proof}

Theorem~\ref{the:Higson-Kasparov} shows that
Proposition~\ref{pro:KKG_vs_KK_over_EG} applies if~\(G\)
is amenable or has the Haagerup approximation property.

Now assume also that~\(G\)
is torsion-free, that is, the only compact subgroup in~\(G\)
is the trivial subgroup~\(\{1\}\).
Then any proper \(G\)\nb-action
is free.  So~\(\mathcal{E}G\)
is also a universal free and proper action.  Since~\(G\)
has only the trivial compact subgroup, Abels' Slice Theorem for proper
actions in \cite{Abels:Universal}*{Theorem~3.3} says here that any
\(x\in\mathcal{E}G\)
has a \(G\)\nb-invariant
open neighbourhood \(U\subseteq\mathcal{E}G\)
that is \(G\)\nb-equivariantly
homeomorphic to \(G\times Y\)
with the translation action on~\(G\times Y\).
In other words, the projection \(\mathcal{E}G \to \mathcal{E}G/G\)
is locally trivial.  So~\(\mathcal{E}G/G\)
is a model for the classifying space of~\(G\),
and it is legitimate to define \(BG\defeq \mathcal{E}G/G\).
We are going to build a functor
\[
M\colon \KK^{BG} \to \KK^{G\ltimes\mathcal{E}G},\qquad
A\mapsto \Cont_0(\mathcal{E}G) \otimes_{BG} A,
\]
and show that it is an equivalence of categories.  An action of~\(BG\)
on a \(\Cst\)\nb-algebra
is simply a \(\Cont_0(BG)\)-\(\Cst\)-algebra
structure.  The bundle projection \(q\colon \mathcal{E}G \to BG\)
is \(G\)\nb-equivariant
if we let~\(G\)
act trivially on~\(BG\).
Equipping a \(\Cont_0(BG)\)-\(\Cst\)-algebra
with the trivial action of~\(G\)
gives a functor \(\KK^{BG} \to KK^{BG\rtimes G}\),
which we may compose with the pull-back functor~\(q^*\)
to get the functor~\(M\)
above.  The explanation why~\(M\)
is an equivalence of categories is that the \(G\)\nb-action
on~\(\mathcal{E}G\)
is free and proper, so that the groupoid \(G\ltimes\mathcal{E}G\)
is equivalent to the space~\(BG\).
  
\begin{proposition}
  \label{pro:equivalence_KK_G-EG_BG}
  Let~\(G\)
  be a torsion-free, locally compact group.  The functor~\(M\)
  above is an equivalence of categories
  \(\KK^{G\ltimes\mathcal{E}G} \simeq \KK^{BG}\).
  The equivalence restricts to the subcategories of objects whose
  underlying \(\Cst\)\nb-algebra is nuclear.
\end{proposition}

\begin{proof}
  Kasparov~\cite{Kasparov:Novikov} builds a functor from
  \(\KK^{G\ltimes\mathcal{E}G}\)
  to \(\KK^{BG}\)
  by taking a fixed-point algebra.  Let~\(\mathcal{B}\)
  be a \(G\ltimes\mathcal{E}G\)-\(\Cst\)-algebra.  Let
  \[
  \Mult(\mathcal{B})^G \defeq
  \setgiven{b\in \Mult(\mathcal{B})}
  {\beta_g(b) = b\text{ for all }g\in G \text{ and }
    \Cont_0(\mathcal{E}G)\cdot b\subseteq \mathcal{B}}.
  \]
  We map \(\Cont_0(BG)\)
  to \(\Contb (\mathcal{E}G) = \Mult(\Cont_0(\mathcal{E}G))\)
  using~\(q^*\)
  and further to \(\Mult(\mathcal{B})\).
  The image of \(\Cont_0(BG)\)
  belongs to the centre of \(\Mult(\mathcal{B})^G\).  So
  \[
  \mathcal{B}^G \defeq \Cont_0(BG) \cdot \Mult(\mathcal{B})^G
  \]
  is a well defined ideal in~\(\Mult(\mathcal{B})^G\).
  This is the \emph{generalised fixed-point algebra} for the
  \(G\)\nb-action
  on~\(\mathcal{B}\).
  It is naturally a \(\Cont_0(BG)\)-\(\Cst\)-algebra.
  This construction is equivalent to the definition
  in~\cite{Kasparov:Novikov} and to the construction of generalised
  fixed-point algebras for generalised proper actions by
  Rieffel~\cite{Rieffel:Proper}.  The fibre of the
  \(\Cont_0(BG)\)-\(\Cst\)-algebra~\(\mathcal{B}^G\)
  at \(x\in BG\)
  is isomorphic to the fibre of~\(\mathcal{B}\)
  at any \(y\in\mathcal{E}G\)
  with \(q(y)=x\);
  these fibres in~\(\mathcal{B}\)
  are canonically isomorphic using the \(G\)\nb-action.
  So in terms of upper semicontinuous fields of \(\Cst\)\nb-algebras,
  our construction is simply viewing a \(G\)\nb-equivariant
  upper semicontinuous field of \(\Cst\)\nb-algebras
  over~\(\mathcal{E}G\)
  as an upper semicontinuous field of \(\Cst\)\nb-algebras
  over~\(BG\).
  The construction of generalised fixed-point algebras above is
  clearly functorial for nondegenerate \(G\)\nb-equivariant,
  \(\Cont_0(\mathcal{E}G)\)-linear
  \Star{}homomorphisms.  And a \(G\)\nb-invariant
  ideal in~\(\mathcal{B}\)
  gives an ideal in~\(\mathcal{B}^G\)
  in an obvious way.  This allows to extend the functoriality of the
  generalised fixed-point algebra construction to degenerate
  \(G\)\nb-equivariant,
  \(\Cont_0(\mathcal{E}G)\)-linear
  \Star{}homomorphisms.  The functor so obtained is split-exact and
  \(\Cst\)\nb-stable
  in the appropriate sense that it defines a functor
  \(\KK^{G\ltimes\mathcal{E}G} \to \KK^{BG}\)
  by the universal property of equivariant \(\KK\)-theory.
  This functor is built by Kasparov~\cite{Kasparov:Novikov} in terms
  of Kasparov cycles, and shown in
  \cite{Kasparov:Novikov}*{Theorem~3.4} to be fully faithful.

  We have built two functors
  \(\KK^{G\ltimes\mathcal{E}G} \leftrightarrow \KK^{BG}\) in
  opposite directions.  We claim that they are inverse to each other
  up to natural equivalence.  Indeed, when we start with a
  \(\Cont_0(BG)\)-\(\Cst\)-algebra~\(A\), we first map it to
  \(\Cont_0(\mathcal{E}G)\otimes_{BG} A\) and then take the
  generalised fixed-point algebra.  In terms of upper semicontinuous
  fields, the field over~\(\mathcal{E}G\) has the same fibres
  as~\(A\) over~\(BG\), and so we get back~\(A\).  Conversely,
  let~\(\mathcal{B}\) be a \(G\ltimes\mathcal{E}G\)\nb-action.  Then
  the structural homomorphism
  \(\Cont_0(\mathcal{E}G) \otimes_{\Cont_0(B G)} \mathcal{B}^G \to
  \mathcal{B}\), \(f\otimes b \mapsto f\cdot b\), is an isomorphism
  of \(\Cont_0(\mathcal{E}G)\)-algebras because it is an isomorphism
  fibrewise.  Therefore, \(\mathcal{B}\) is in the essential range
  of the functor~\(M\) above.  This equivalence of categories is
  checked in different notation
  in~\cite{Huef-Raeburn-Williams:Brauer_semigroup}.
\end{proof}

\begin{theorem}
  \label{the:KKG_torsion-free}
  Let~\(G\)
  be a torsion-free, amenable, second countable group.  Let~\(BG\)
  be a second countable, locally compact model for the classifying
  space of~\(G\).
  The category~\(\KK^G\)
  is equivalent to the subcategory of~\(\KK^{BG}\)
  whose objects are the locally trivial bundles of
  \(\Cst\)\nb-algebras
  over~\(BG\).
  This equivalence restricts to an equivalence between the
  subcategories of objects whose underlying \(\Cst\)\nb-algebra
  is nuclear.
\end{theorem}

\begin{proof}
  Any model of~\(BG\) carries a universal principal \(G\)\nb-bundle,
  whose total space is a model for~\(\mathcal{E}G\).  Therefore, the
  choice of~\(BG\) does not matter.

  Proposition~\ref{pro:KKG_vs_KK_over_EG} applies here because of
  Theorem~\ref{the:Higson-Kasparov}.
  Proposition~\ref{pro:equivalence_KK_G-EG_BG} applies because~\(G\)
  is torsion-free.  Both propositions together give an equivalence
  of categories from~\(\KK^G\) onto a full subcategory
  of~\(\KK^{B G}\).  Let
  \(BG^{(2)} \defeq (\mathcal{E}G \times \mathcal{E}G)/G\).  The
  coordinate projections
  \(p_1,p_2\colon \mathcal{E}G\times\mathcal{E}G \rightrightarrows
  \mathcal{E}G\) induce continuous maps
  \(\check{p}_1, \check{p}_2\colon BG^{(2)} \rightrightarrows BG\).
  The description of the range of the equivalence from~\(\KK^G\) to
  a full subcategory of \(\KK^{G\ltimes \mathcal{E}G}\) implies that a
  \(\Cont_0(B G)\)-\(\Cst\)-algebra~\(\mathcal{B}\) belongs to the
  essential image of~\(\KK^G\) if and only if
  \(\check{p}_1^*(\mathcal{B})\) and~\(\check{p}_2^*(\mathcal{B})\)
  are \(\KK^{B G^{(2)}}\)-equivalent.  We must show that this holds
  if and only if~\(\mathcal{B}\) is \(\KK^{B G}\)-equivalent to a
  locally trivial bundle of \(\Cst\)\nb-algebras.

  In one direction, we may directly compute the functor
  \(\KK^G \to \KK^{B G}\)
  on objects.  Let~\(A\)
  be a separable \(G\)\nb-\(\Cst\)-algebra,
  viewed as an object of~\(\KK^G\).
  The functor to~\(\KK^{BG}\)
  maps~\(A\)
  to the generalised fixed-point algebra
  \(\Cont_0(\mathcal{E}G,A)^G\).
  This bundle of \(\Cst\)\nb-algebras
  over~\(B G\)
  is locally trivial because the projection \(\mathcal{E} G \to B G\)
  is locally trivial.  So any object in the essential range of our
  functor is \(\KK^{B G}\)-equivalent
  to a locally trivial bundle of \(\Cst\)\nb-algebras.
  Conversely, assume~\(\mathcal{B}\)
  to be a locally trivial bundle of \(\Cst\)\nb-algebras
  over~\(B G\).
  The maps
  \(\check{p}_1, \check{p}_2\colon BG^{(2)} \rightrightarrows BG\)
  are homotopic to each other because \(p_1,p_2\)
  are \(G\)\nb-equivariantly
  homotopic.  The quotient spaces \(BG\)
  and~\(BG^{(2)}\)
  are second countable and locally compact, Hausdorff because
  \(\mathcal{E}G\)
  and~\(\mathcal{E}G\times\mathcal{E}G\)
  are so and~\(G\)
  acts properly on them.  So \(BG\)
  and~\(BG^{(2)}\)
  are paracompact.  Hence the pull-backs of~\(\mathcal{B}\)
  along the maps
  \(\check{p}_1, \check{p}_2\colon BG^{(2)} \rightrightarrows BG\)
  are isomorphic as locally trivial bundles over~\(BG^{(2)}\)
  (see \cite{Husemoller:Fibre_bundles}*{Theorem 2.9.9 on p.~51}).
  Hence they are \(\KK^{B G^{(2)}}\)-equivalent as needed.

  Finally, the constructions above all preserve nuclearity.  So the
  subcategory of all \(G\)\nb-actions
  on separable nuclear \(\Cst\)\nb-algebras
  is equivalent to the subcategory of all locally trivial bundles of
  nuclear \(\Cst\)\nb-algebras over~\(B G\).
\end{proof}

\begin{remark}
  The equivalence from~\(\KK^G\) to a full subcategory
  of~\(\KK^{BG}\) maps a \(\Cst\)\nb-algebra~\(A\) with
  \(G\)\nb-action~\(\alpha\) to a locally trivial bundle of
  \(\Cst\)\nb-algebras with fibre~\(A\).  This is equivalent to a
  locally trivial \(\Aut(A)\)-principal bundle.  The latter is equal
  to the principal bundle~\(\mathcal{P}_\alpha\) that appears in the
  conjectures of Izumi~\cite{Izumi:Group_actions}.
\end{remark}

\begin{definition}
  An isomorphism or \(\KK^G\)-equivalence \(A\to A\) is called
  \emph{\(\KK\)-trivial} if the forgetful functor maps it to the
  identity in \(\KK(A,A)\).  An isomorphism of \(\Aut(A)\)-principal
  bundles is \(\KK\)-trivial if the induced isomorphism of locally
  trivial \(\Cst\)\nb-algebra bundles with fibre~\(A\) restricts to
  the identity on~\(A\) in each fibre.
\end{definition}

The notion of a \(\KK\)-trivial isomorphism of \(\Aut(A)\)-principal
bundles is well defined, that is, it does not depend on the
trivialisation of the locally trivial bundle.  This is because the
identity in \(\KK(A,A)\) is invariant under conjugation by
automorphisms.

\begin{theorem}
  \label{the:actions_Kirchberg_algebra}
  Let~\(A\) be a purely infinite, simple, nuclear,
  \(\Cst\)\nb-stable, separable \(\Cst\)\nb-algebra and let~\(G\) be
  a torsion-free, amenable, second countable, locally compact group.
  Let~\(B G\) be a second countable, locally compact classifying
  space for~\(G\).  Then \textup{(}\(\KK\)-trivial\textup{)}
  \(\KK^G\)-equivalence classes of \(G\)\nb-actions on~\(A\) are in
  bijection with \textup{(}\(\KK\)-trivial\textup{)} isomorphism
  classes of \(\Aut(A)\)-principal bundles over~\(B G\).
\end{theorem}

\begin{proof}
  Let \(\alpha,\beta\)
  be two actions of~\(G\)
  on~\(A\).
  Let \(\mathcal{A}\)
  and~\(\mathcal{B}\)
  be the corresponding locally trivial bundles of \(\Cst\)\nb-algebras
  over~\(B G\).
  Theorem~\ref{the:KKG_torsion-free} shows that
  \[
  \KK^G\bigl((A,\alpha),(B,\beta)\bigr)
  \cong \KK^{B G}(\mathcal{A},\mathcal{B}).
  \]
  In particular, our functor maps \(\KK^G\)-equivalences
  between \((A,\alpha)\)
  and \((B,\beta)\)
  bijectively to \(\KK^{B G}\)-equivalences
  between \(\mathcal{A}\)
  and~\(\mathcal{B}\).
  Since the fibres of the bundles \(\mathcal{A}\)
  and~\(\mathcal{B}\)
  are isomorphic to the Kirchberg algebra~\(A\),
  these bundles are purely infinite.  We claim that they are even
  strongly purely infinite.  If~\(B G\)
  is finite-dimensional, this follows from
  \cite{Blanchard-Kirchberg:Non-simple_infinite}*{Theorem~5.8}.
  If~\(B G\)
  is infinite-dimensional, the result is still true for a different
  reason.  It is easy to see that~\(\mathcal{A}\)
  is (``weakly'') purely infinite.  Since \(\Cont_0(B G)\)
  is in the centre of the multiplier algebra of~\(\mathcal{A}\)
  and~\(A\)
  has an approximate unit of projections, it follows
  that~\(\mathcal{A}\)
  has the ``locally central decomposition property'' (see
  \cite{Kirchberg-Rordam:Infinite_absorbing}*{Definition~6.1}).  And a
  purely infinite \(\Cst\)\nb-algebra
  with the the locally central decomposition property is strongly
  purely infinite by
  \cite{Kirchberg-Rordam:Infinite_absorbing}*{Theorem~6.8}.  Then
  \(\mathcal{A} \cong \mathcal{A}\otimes \mathcal{O}_\infty\)
  by \cite{Kirchberg-Rordam:Infinite_absorbing}*{Theorem~8.6}.
  In addition, \(\mathcal{A}\) is \(\Cst\)\nb-stable by
  \cite{Hirshberg-Rordam-Winter:Stability}*{Proposition~3.9}; I
  thank the referee for clarifying the need for this reference.

  Hence \(\mathcal{A}\)
  and~\(\mathcal{B}\)
  are covered by Kirchberg's Classification Theorem (see
  \cite{Kirchberg:Michael}*{Folgerung~4.3}).  Namely, any
  \(\KK^{B G}\)-equivalence
  \(\mathcal{A}\cong \mathcal{B}\)
  is induced by an isomorphism \(\mathcal{A}\cong \mathcal{B}\)
  as \(\Cont_0(B G)\)-\(\Cst\)\nb-algebras.
  We are also interested in \(\KK^{B G}\)-equivalences
  between two actions \(\alpha,\beta\) on the same stable Kirchberg
  algebra~\(A\) that restrict to the identity in \(\KK(A,A)\)
  in each fibre.  This happens if and only if the corresponding
  equivalence in \(\KK^G((A,\alpha),A,\beta))\)
  is \(\KK\)-trivial.
  We see that any \(\KK\)-trivial equivalence
  in \(\KK^G((A,\alpha),A,\beta))\)
  comes from a \(\KK\)-trivial
  isomorphism \(\mathcal{A}\cong \mathcal{B}\)
  as \(\Cont_0(B G)\)-\(\Cst\)\nb-algebras.
  And the latter are the same as \(\KK\)-trivial
  isomorphisms of \(\Aut(A)\)-principal bundles.

  To prove the desired bijection, we must show that any locally
  trivial bundle of \(\Cst\)\nb-algebras~\(\mathcal{A}\)
  over~\(B G\)
  with fibre~\(A\)
  comes from some action~\(\alpha\)
  on~\(A\).
  To begin with, Theorem~\ref{the:KKG_torsion-free} shows that there
  is some object \((B,\beta)\)
  of \(\KK^G\)
  so that~\(B\)
  is nuclear and the locally trivial \(\Cst\)\nb-algebra
  bundle over~\(B G\)
  associated to \((B, \beta)\)
  is \(\KK^{B G}\)-equivalent
  to~\(\mathcal{A}\).  Theorem~\ref{the:make_actions_simple} gives a
  \(\KK^G\)-equivalence
  from~\((B,\beta)\)
  to some action \((\tilde{B},\tilde{\delta})\)
  where~\(\tilde{B}\)
  is purely infinite, simple, nuclear, \(\Cst\)\nb-stable
  and separable.  By the Kirchberg--Phillips Theorem,
  \(\tilde{B} \cong A\).
  So the action~\(\tilde{\delta}\)
  transfers to an action~\(\alpha\)
  on~\(A\)
  so that the associated locally trivial \(\Cst\)\nb-algebra
  bundle over~\(B G\)
  is \(\KK^{B G}\)-equivalent
  to our given~\(\mathcal{A}\).
  As above, this \(\KK^{B G}\)-equivalence
  may be improved to an isomorphism.  And if the isomorphism is not
  yet \(\KK\)-trivial,
  we simply transfer the action along this isomorphism to get another
  action on~\(A\)
  such that the resulting bundle over~\(B G\)
  is even \(\KK\)-trivially isomorphic to the given bundle.
\end{proof}

\begin{example}
  Let \(G=\R^n\).  This group is torsion-free and the translation
  action on~\(\R^n\) is a model for~\(\mathcal{E}G\) with
  \(BG=\pt\).  So Theorem~\ref{the:actions_Kirchberg_algebra}
  implies that all \(\R^n\)\nb-actions on~\(A\) are
  \(\KK\)-trivially \(\KK^G\)-equivalent to the trivial action.
\end{example}

\begin{example}
  Let \(G=\Z\).  The translation action of~\(\Z\) on~\(\R\) is a
  universal proper action.  So \(BG=\R/\Z\) is a circle.
  Isomorphism classes of \(\Aut(A)\)-principal bundles over the
  circle are naturally in bijection with the path connected
  components of \(\Aut(A)\).  Any automorphism of~\(A\) determines
  an action of~\(\Z\).  Thus any \(\Aut(A)\)-principal bundle over
  the circle comes from an action of~\(\Z\) on~\(A\).
  Theorem~\ref{the:actions_Kirchberg_algebra} says, in addition,
  that two \(\Z\)\nb-actions are \(\KK\)-trivially
  \(\KK^G\)-equivalent if and only if the generating automorphisms
  of~\(A\) are homotopic.
\end{example}

\begin{example}
  Let \(G=\Z^2\).  Then the translation action of~\(\Z^2\)
  on~\(\R^2\) is a universal proper action.  So \(BG=\R^2/\Z^2\) is
  a \(2\)\nb-torus.  Theorem~\ref{the:actions_Kirchberg_algebra}
  says that any \(\Aut(A)\)-principal bundle over the \(2\)\nb-torus
  is associated to an action of~\(\Z^2\) on~\(A\).  Using the
  standard CW-complex decomposition of the \(2\)\nb-torus and some
  obstruction theory, we may describe an \(\Aut(A)\)-principal
  bundle by two elements \(\alpha_1,\alpha_2\) of \(\Aut(A)\)
  together with a homotopy~\(H\) between \(\alpha_1 \alpha_2\) and
  \(\alpha_2 \alpha_1\) in \(\Aut(A)\).  The existence part of
  Theorem~\ref{the:actions_Kirchberg_algebra} says that there are
  automorphisms \(\beta_1,\beta_2\) of~\(A\) with
  \(\beta_1\beta_2 = \beta_2\beta_1\) and homotopies~\(K_j\)
  from~\(\alpha_j\) to~\(\beta_j\) for \(j=1,2\), such that the
  concatenation of the homotopies \(K_2\cdot K_1\) reversed, \(H\)
  and \(K_1\cdot K_2\) is homotopic to the constant homotopy on
  \(\beta_1 \beta_2 = \beta_2 \beta_1\).  The pair
  \(\beta_1,\beta_2\) with \(\beta_1 \beta_2 = \beta_2 \beta_1\) is
  equivalent to an action of~\(\Z^2\).

  Take two actions of~\(\Z^2\) on~\(A\) and describe the
  corresponding principal \(\Aut(A)\)-bundles through triples
  \((\alpha_1,\alpha_2,H)\) and \((\alpha_1',\alpha_2',H')\) as
  above.  Since we start with \(\Z^2\)\nb-actions, we may arrange
  that \(\alpha_1 \alpha_2 = \alpha_2 \alpha_1\),
  \(\alpha'_1 \alpha'_2 = \alpha'_2 \alpha'_1\), and \(H\)
  and~\(H'\) are the constant homotopies.  The principal bundles are
  isomorphic if and only if~\(\alpha_j\) is homotopic
  to~\(\alpha_j'\) for \(j=1,2\) in such a way that composing the
  resulting homotopy from
  \(\alpha_1 \alpha_2 \alpha_1^{-1} \alpha_2^{-1}= \id_A\) to
  \(\alpha_1' \alpha_2' (\alpha_1')^{-1} (\alpha_2')^{-1} = \id_A\)
  is homotopic to the constant homotopy.  The uniqueness part of
  Theorem~\ref{the:actions_Kirchberg_algebra} says that this happens
  if and only if the two \(\Z^2\)\nb-actions on~\(A\) are
  \(\KK\)\nb-trivially \(\KK^{\Z^2}\)-equivalent.  Both the
  existence and uniqueness part of
  Theorem~\ref{the:actions_Kirchberg_algebra} are non-trivial in
  this case and depend on~\(A\) being a Kirchberg algebra.
\end{example}

These examples end the discussion of actions of torsion-free groups.
We now turn to the opposite end: actions of cyclic groups of prime
order.

\section{The equivariant bootstrap class}
\label{sec:equivariant_bootstrap}

\begin{definition}[\cite{dellAmbrogio-Emerson-Meyer:Equivariant_Lefschetz}*{§3.1}]
  Let~\(G\)
  be a compact group.  A \(G\)\nb-action
  \((A,\alpha)\)
  on a separable \(\Cst\)\nb-algebra~\(A\)
  belongs to the \emph{\(G\)\nb-equivariant
    bootstrap class~\(\mathfrak{B}^G\)}
  if it is \(\KK^G\)-equivalent
  to an action~\((B,\beta)\) where~\(B\) is a \(\Cst\)\nb-algebra of Type~I.
\end{definition}

The equivariant bootstrap class is also described
in~\cite{dellAmbrogio-Emerson-Meyer:Equivariant_Lefschetz} as the
localising subcategory generated by certain ``elementary'' building
blocks.  These building blocks are actions of~\(G\)
on \(\Cst\)\nb-algebras
that are contained in \(\Comp(\Hils)\)
for a separable Hilbert space~\(\Hils\).
Up to taking direct sums, we may get them all by induction applied to
an action of a closed subgroup of~\(G\)
on a matrix algebra.  Actions on matrix algebras, in turn, correspond
to projective representations.  Two projective representations with
the same cocycle give equivariantly Morita equivalent actions on
matrix algebras.  So it suffices to take one projective representation
from each cohomology class.

The equivariant bootstrap class~\(\mathfrak{B}^G\)
is closed under \(\KK^G\)-equivalence
by definition.  So \((B,\beta)\)
in Theorem~\ref{the:make_actions_simple} belongs to the equivariant
bootstrap class if \((A,\alpha)\) does.

From now on, we specialise to a finite cyclic group of prime order,
that is, \(G=\Z/p\)
for a prime~\(p\).
The only proper subgroup of~\(G\)
is the trivial subgroup, and neither~\(G\)
nor the trivial group have non-trivial projective representations.
Thus the equivariant bootstrap class~\(\mathfrak{B}^G\)
for~\(G\)
is generated by two objects, namely, \(\C = \Cont(G/G)\)
with the trivial action of~\(G\)
and \(\Cont(G)\)
with the translation action of~\(G\).
Any \(G\)\nb-action
on a finite-dimensional \(\Cst\)\nb-algebra
is \(G\)\nb-equivariantly
Morita--Rieffel equivalent to a direct sum of copies of these two
actions.  Given a separable \(G\)-\(\Cst\)\nb-algebra~\(A\),
its \emph{little invariant} is defined by
\[
L_*(A) \defeq \KK_*^G(\C,A) \oplus \KK_*^G(\Cont(G),A)
\cong \KK_*^G(\C\oplus \Cont(G),A),
\]
where the isomorphism uses the well-known additivity of~\(\KK^G\).

The little invariant does not classify actions up to
\(\KK^G\)-equivalence; that is, there are isomorphisms
\(L_*(A) \cong L_*(B)\) that do not lift to isomorphisms
\(A \cong B\) (see Example~\ref{exa:actions_on_Cuntz_4}).  It has
enough information, however, to detect whether an object in the
equivariant bootstrap class is non-zero:

\begin{lemma}
  \label{lem:little_invariant_detects_zero}
  Let \(A,B\in\mathfrak{B}^G\).
  Then \(L_*(A) \cong 0\)
  if and only if \(A\)
  is \(\KK^G\)-equivalent
  to~\(0\).
  And \(f\in\KK^G_0(A,B)\)
  is invertible if and only if \(L_*(f) \colon L_*(A) \to L_*(B)\)
  is invertible.  
\end{lemma}

\begin{proof}
  This uses basic triangulated category theory.  Any arrow~\(f\)
  is part of an exact triangle
  \[
  \Sigma B \to C \to A \xrightarrow{f} B.
  \]
  Here the object~\(C\)
  of~\(\mathfrak{B}^G\)
  is unique up to isomorphism.  It is called the \emph{cone} of~\(f\).
  Versions of the Five Lemma show that~\(f\)
  is invertible if and only if \(C\cong 0\)
  and that \(L_*(f)\)
  is invertible if and only if \(L_*(C)\cong 0\).
  So the second statement follows from the first statement
  for~\(C\).
  It is clear that \(C\cong 0\)
  implies \(L_*(C) \cong 0\).
  Conversely, assume \(L_*(C)\cong0\).
  Let \(Z\subseteq \mathfrak{B}^G\)
  be the class of all objects~\(D\)
  with \(\KK^G_*(D,C)=0\).
  This is localising and contains \(\C\)
  and \(\Cont(G)\)
  by assumption.  Hence it contains all of~\(\mathfrak{B}^G\).
  So \(\KK^G_*(C,C)=0\).
  This is equivalent to \(C\cong 0\).
\end{proof}

Next we relate \(L_*(A)\)
to \(\K\)\nb-theory.
We use the Green--Julg isomorphism
\[
\KK_*^G(\C,A) \cong \KK_*(\C,A\rtimes G)
\]
(see \cite{Julg:K_equivariante} or
\cite{Blackadar:K-theory}*{Theorem 20.2.7.(a)}), the
induction--restriction adjunction
\[
\KK_*^G(\Cont(G),A) \cong \KK_*(\C,A)
\]
for the inclusion of the trivial subgroup in~\(G\)
(see \cite{Meyer-Nest:BC}*{Section~3.2}) and the isomorphism
\[
\KK_*(\C,A\rtimes G) \cong \K_*(A\rtimes G)
\]
(see \cite{Kasparov:Operator_K} or
\cite{Blackadar:K-theory}*{Proposition~17.5.5}).  These give a
natural isomorphism
\[
L_*(A) \cong \K_*(A\rtimes G) \oplus \K_*(A).
\]
In particular, \(L_*(A)\)
is countable for any separable \(G\)\nb-\(\Cst\)-algebra~\(A\).

The invariant \(L_*(A)\)
is a \(\Z/2\)\nb-graded
\emph{right} module over the \(\Z/2\)\nb-graded
ring \(\KK^G_*(\C\oplus \Cont(G),\C\oplus\Cont(G))\)
by the Kasparov composition product.  We will later describe modules
as representations of the underlying ring.  Then it becomes confusing
to use right modules.  Therefore, we use the opposite
\(\Z/2\)\nb-graded ring
\[
\littlering \defeq\KK^G_*(\C\oplus \Cont(G),\C\oplus\Cont(G))^\op
\]
and treat \(L_*(A) \cong \KK^G_*(\C\oplus \Cont(G),A)\)
as a \emph{left} \(\littlering\)\nb-module.
We denote reverse order composition products in~\(\KK^G\)
by the symbol~\(\otimes\);
this symbol is also used by Kasparov, but with a subscript denoting
the \(\Cst\)\nb-algebra
over which tensor products of \(\KK\)-cycles
are balanced.  We will never use exterior products here,
so~\(\otimes\)
for \(\KK^G\)-cycles
will always mean the reverse-order composition product.

We may compute the underlying \(\Z/2\)\nb-graded
Abelian group of~\(\littlering\)
using the additivity of~\(\KK^G\)
and the isomorphisms that relate~\(L_*(A)\) to \(\K\)\nb-theory:
\begin{multline}
  \label{eq:R01_through_K-theory}
  \littlering = L_*(\C\oplus \Cont(G))
  \cong L_*(\C)\oplus L_*(\Cont(G))
  \\\cong \K_*(\C\rtimes G) \oplus \K_*(\C) \oplus \K_*(\Cont(G)\rtimes G) \oplus \K_*(\Cont(G))
  \\\cong \Z[G] \oplus \Z \oplus \Z \oplus \Z[G]
  \cong \Z^{2 p + 2}.
\end{multline}
In particular, the odd part of~\(\littlering\)
vanishes.  The ring structure on~\(\littlering\)
is also transparent from this computation.  First, the direct sum
decomposition of~\(\littlering\)
used in~\eqref{eq:R01_through_K-theory} gives~\(\littlering\)
a \(2\times 2\)-matrix structure with elements written as
\[
\begin{pmatrix}
  x_{00}&x_{01}\\x_{10}&x_{11}
\end{pmatrix},\qquad
\begin{alignedat}{2}
  x_{00} &\in \KK^G_0(\C,\C) \cong \Z[G],&\ x_{01} &\in \KK^G_0(\C,\Cont(G))\cong \Z,\\
  x_{10} &\in \KK^G_0(\Cont(G),\C) \cong \Z,&\ x_{11} &\in \KK^G_0(\Cont(G),\Cont(G))\cong \Z[G].
\end{alignedat}
\]
These are multiplied like \(2\times2\)\nb-matrices.  The diagonal
entry \(\KK^G_0(\C,\C)\) is isomorphic to the representation ring
of~\(G\).  It acts on \(\K_*(A\rtimes G) \cong \K_*^G(A)\) by
exterior tensor product of \(G\)\nb-equivariant projective modules
with representations because the Kasparov composition product with
elements of \(\KK^G_0(\C,\C)\) is equal to the exterior product.
The diagonal entry \(\KK^G_0(\Cont(G),\Cont(G))\) is isomorphic to
the group ring of~\(G\), realised by the \(\KK^G\)-classes of
translations \(G\to G\), \(x\mapsto g\cdot x\), for \(g\in G\).
Since~\(G\) is Abelian, the map on \(\KK^G_*(\Cont(G),A)\) induced
by a translation in \(\Cont(G)\) is the same as the map induced by
applying the corresponding automorphism~\(\alpha_g\) of~\(A\).
Therefore, the induced action of~\(G\) on
\(\KK^G_*(\Cont(G),A) \cong \K_*(A)\) is the canonical one induced
by the \(G\)\nb-action on~\(A\).

Since~\(G\) is the cyclic group of order~\(p\), its representation
ring and its group ring are canonically isomorphic to
\begin{equation}
  \label{eq:S_is_group_ring}
  \Sring \defeq \Z[t]/(t^p-1).
\end{equation}
This ring will play a crucial role in our study.  Next we explain the
isomorphism between the representation ring and the group ring
of~\(G\)
using Baaj--Skandalis duality (see~\cite{Baaj-Skandalis:Hopf_KK}).
This duality will generate an important involutive automorphism of
Manuel Köhler's ring as well.

Baaj--Skandalis duality for a locally compact group~\(G\)
says that taking the crossed product and giving it
the dual coaction defines an equivalence of \(\Z/2\)\nb-graded
categories
\[
\BS\colon \KK^G \congto \KK^{\widehat{G}},\qquad
A\mapsto A\rtimes G,\quad
\KK^G_*(A,B) \congto \KK^{\widehat{G}}_*(A\rtimes G,B\rtimes G).
\]
For our group~\(G\),
a coaction of~\(G\)
is equivalent to an action of the dual group~\(\widehat{G}\),
and \(\widehat{G} \cong G\).
So we get an automorphism \(\KK^G \to \KK^G\),
mapping a separable \(G\)\nb-\(\Cst\)\nb-algebra~\(A\)
to~\(A\rtimes G\)
with the dual action of~\(G\).
Taking this isomorphism twice gives
\((A\rtimes G)\rtimes \widehat{G} \cong A\otimes \Comp(\ell^2 G)\),
which is equivariantly Morita--Rieffel equivalent to~\(A\).
Now \(\Cont(G) \rtimes G \cong \Comp(\ell^2 G)\)
is equivariantly Morita--Rieffel equivalent to~\(\C\)
and \(\C\rtimes G \cong \Cst(G) \cong \Cont(G)\)
equivariantly.  So
\((\C\oplus \Cont(G))\rtimes G \sim_{\KK^G} \C\oplus \Cont(G)\),
but this \(\KK^G\)-equivalence
exchanges the two summands.  And Baaj--Skandalis duality contains an
isomorphism
\[
\KK^G_0(\C,\C) \cong \KK^G_0(\Cont(G),\Cont(G)).
\]

Actually, the Baaj--Skandalis duality isomorphism depends on the
choice of the isomorphism \(\widehat{G} \cong G\)
that is used to turn \(G\)\nb-coactions
into \(G\)\nb-actions.
Different choices differ by an automorphism~\(\varphi\)
of~\(G\),
and the resulting automorphisms~\(\BS\)
of~\(\KK^G\)
differ by the automorphism of~\(\KK^G\)
that replaces an action \(\alpha\colon G\to\Aut(A)\)
by \(\alpha\circ\varphi\).
We choose the isomorphism \(\widehat{G} \cong G\)
that is given by the symmetric pairing
\[
\Z/p \times \Z/p \to \mathrm{U}(1),\qquad
(g,h) \mapsto \exp(2\pi \ima\cdot g\cdot h/p).
\]
For this pairing, the Baaj--Skandalis duality automorphism is
involutive, that is, \(\BS\circ \BS\)
is naturally equivalent to the identity functor on~\(\KK^G\).

Baaj--Skandalis duality induces a ring automorphism
of~\(\littlering\),
which exchanges the two copies of~\(\Sring\)
on the diagonal and the two copies of~\(\Z\)
off the diagonal.  One copy of~\(\Z\) is
\begin{equation}
  \label{eq:KK_G_C_G}
  \KK^G_0(\C,\Cont(G)) \cong \K_0(\Cont(G)\rtimes G) \cong
  \K_0(\Comp(\ell^2 G)) \cong \K_0(\C) \cong \Z.
\end{equation}
Both the representation ring and the group ring of~\(\Z\)
act on this by the augmentation character
\begin{equation}
  \label{eq:augmentation_character}
  \tau\colon \Sring =\Z[t]/(t^p-1) \to \Z,\qquad
  t\mapsto 1.
\end{equation}
This is because the regular representation \(\lambda\)
of~\(G\)
satisfies \(\chi \otimes \lambda \cong \lambda\)
for all characters~\(\chi\)
of~\(G\)
and because translations in~\(G\)
become inner automorphisms of \(\Cont(G)\rtimes G\)
and thus act trivially on \(\K\)\nb-theory.
This gives the multiplication of the diagonal entries with the
\(01\)-entries
of \(2\times2\)-matrices
in~\(\littlering\).
By Baaj--Skandalis duality, the multiplication of the \(10\)-entries
with diagonal entries is given by the same formula.

To compute the multiplication of two off-diagonal entries, we
describe the generators of these entries more concretely.  Let
\(\Ga{0}{1}\colon \C \to \Cont(G)\) be the unit map.  We also denote
its class in \(\KK^G_0(\C,\Cont(G))\) by~\(\Ga{0}{1}\).  The first
isomorphism in~\eqref{eq:KK_G_C_G} maps~\(\Ga{0}{1}\) to the
projection onto the trivial representation of~\(G\)
in~\(\ell^2(G)\).  Since this has dimension~\(1\), it follows
that~\(\Ga{0}{1}\) is a generator of
\(\KK^G_0(\C,\Cont(G)) \cong \Z\).

Baaj--Skandalis duality maps~\(\Ga{0}{1}\) to a generator
\begin{equation}
  \label{eq:alpha_10}
  \Ga{1}{0}\defeq \BS(\Ga{0}{1}) \in \KK^G_0(\Cont(G),\C)
\end{equation}
for the \(10\)-summand
\(\KK^G_0(\Cont(G),\C) \cong \Z\)
in~\(\littlering\).
Concretely, this is represented by the \(G\)\nb-equivariant
\Star{}homomorphism
\[
\Cont(G) \cong \C\rtimes G \xrightarrow{\Ga{0}{1}\rtimes G}
\Cont(G)\rtimes G \cong \Comp(\ell^2 G),
\]
combined with the canonical equivariant Morita--Rieffel equivalence
\(\Comp(\ell^2 G) \sim \C\).
A short computation identifies this \(G\)\nb-equivariant
\Star{}homomorphism with the canonical inclusion
\(\Cont(G) \to\Comp(\ell^2 G)\)
by pointwise multiplication.

\begin{lemma}
  \label{lem:natural_trafo_alpha_10_01}
  The natural transformations
  \(\K_*(A\rtimes G) \leftrightarrows \K_*(A)\) induced by
  \(\Ga{1}{0}\) and \(\Ga{0}{1}\) are equal to the natural
  transformations induced by the inclusions
  \(A \hookrightarrow A\rtimes G \hookrightarrow A\otimes
  \Comp(\ell^2 G)\).
\end{lemma}

\begin{proof}
  The inclusions \(A \hookrightarrow A\rtimes G\) are natural and
  hence form a natural transformation
  \(\eta_A\colon \K_*(A) \to \K_*(A\rtimes G)\).  By the Yoneda
  Lemma, this must come from a class in \(\KK^G_0(\C,\Cont(G))\).
  Since \(\KK^G_0(\C,\Cont(G)) \cong \Z\) with
  generator~\(\Ga{0}{1}\), the natural transformation~\(\eta\) must
  be a multiple of the natural transformation induced
  by~\(\Ga{0}{1}\).  Now~\(\eta_\C\) maps the class of the unit in
  \(\K_0(\C)\) onto the class of the unit in
  \(\K_0(\Cst(G)) \cong \Z[G]\).  So does the natural transformation
  induced by~\(\Ga{0}{1}\).  This proves the claim for the inclusion
  \(A \hookrightarrow A\rtimes G\).  The other inclusion is treated
  similarly.
\end{proof}

The computation of \(\Ga{0}{1}\) and \(\Ga{1}{0}\) implies that
\(\Ga{0}{1}\otimes \Ga{1}{0}\in \KK^G_0(\C,\C)\)
is represented by the unit map \(\C \to \Comp(\ell^2 G)\).
Its image in the representation ring is the class of the regular
representation.  The isomorphism \(\KK^G_0(\C,\C) \cong \Sring\)
maps this to the \emph{norm element}
\begin{equation}
  \label{eq:define_Norm_element}
  N(t) \defeq 1 + t + t^2 + \dotsb + t^{p-1} \in \Sring
\end{equation}
because each character appears in the regular representation with
multiplicity one.  We compute the other composition using that
Baaj--Skandalis duality is an involutive ring automorphism:
\[
\Ga{1}{0}\otimes \Ga{0}{1}
= \BS(\Ga{0}{1}) \otimes \BS^2(\Ga{0}{1})
= \BS(\Ga{0}{1}\otimes \Ga{1}{0})
= \BS(N(t)).
\]
This completely describes the multiplication in the
ring~\(\littlering\).  The following lemma summarises our results:

\begin{lemma}
  \label{lem:little_invariant_ring}
  The elements in the ring~\(\littlering\)
  are equivalent to \(2\times2\)-matrices
  \[
  \begin{pmatrix}
    x_{00}&x_{01}\\x_{10}&x_{11}
  \end{pmatrix},\qquad
  x_{00},x_{11} \in \Sring,\ x_{01},x_{10}\in \Z,
  \]
  and these are multiplied by
  \[
  \begin{pmatrix}
    x_{00}&x_{01}\\x_{10}&x_{11}
  \end{pmatrix}\otimes
  \begin{pmatrix}
    y_{00}&y_{01}\\y_{10}&y_{11}
  \end{pmatrix} =
  \begin{pmatrix}
    x_{00} y_{00} + x_{01}  y_{10}  N(t)&
    \tau(x_{00})  y_{01} + x_{01}  \tau(y_{11})\\
    x_{10} \tau(y_{00}) + \tau(x_{11}) y_{10}&
    x_{10} y_{01} N(t) + x_{11} y_{11}
  \end{pmatrix}.
  \]
  The augmentation character~\(\tau\)
  and the norm element~\(N(t)\)
  are defined in \eqref{eq:augmentation_character}
  and~\eqref{eq:define_Norm_element}.  Baaj--Skandalis duality acts on
  this ring by the involutive automorphism
  \[
  \BS \colon \littlering \congto \littlering,\qquad
  \begin{pmatrix}
    x_{00}&x_{01}\\x_{10}&x_{11}
  \end{pmatrix}\mapsto
  \begin{pmatrix}
    x_{11}&x_{10}\\x_{01}&x_{00}
  \end{pmatrix}.
  \]
\end{lemma}

The ring~\(\littlering\)
described in Lemma~\ref{lem:little_invariant_ring} is generated by the
following elements:
\begin{alignat*}{2}
  \Gu{0}&\defeq
  \begin{pmatrix}
    1&0\\0&0
  \end{pmatrix},&\qquad
  \Gu{1}&\defeq
  \begin{pmatrix}
    0&0\\0&1
  \end{pmatrix},\\
  \Gt{0}&\defeq
  \begin{pmatrix}
    t&0\\0&0
  \end{pmatrix},&\qquad
  \Gt{1}&\defeq
  \begin{pmatrix}
    0&0\\0&t
  \end{pmatrix},\\
  \Ga{0}{1}&\defeq
  \begin{pmatrix}
    0&1\\0&0
  \end{pmatrix},&\qquad
  \Ga{1}{0}&\defeq
  \begin{pmatrix}
    0&0\\1&0
  \end{pmatrix}.
\end{alignat*}
These are subject to the following relations:
\begin{enumerate}
\item \(\Gu{0}\)
  and \(\Gu{1}\)
  are idempotent and \(\Gu{0}+\Gu{1}\)
  is a unit element in~\(\littlering\);
\item \(\Gu{j}\otimes \Gt{j} = \Gt{j} = \Gt{j}\otimes \Gu{j}\)
  for \(j=0,1\)
  and
  \(\Gu{0}\otimes \Ga{0}{1}= \Ga{0}{1} = \Ga{0}{1}\otimes \Gu{1}\),
  \(\Gu{1}\otimes \Ga{1}{0} = \Ga{1}{0} = \Ga{1}{0}\otimes \Gu{0}\);
\item
  \(\Ga{0}{1}\otimes \Gt{1}= \Ga{0}{1}= \Gt{0}\otimes \Ga{0}{1}\)
  and
  \(\Ga{1}{0}\otimes \Gt{0}= \Ga{1}{0}= \Gt{1}\otimes \Ga{1}{0}\);
\item \(\Ga{0}{1}\otimes \Ga{1}{0}= N(\Gt{0})\) and
 \(\Ga{1}{0}\otimes \Ga{0}{1}= N(\Gt{1})\).
\end{enumerate}
The relations \(\Gt{j}^p = \Gu{j}\) for \(j=0,1\) follow from this by writing
\[
\Gt{j}^p-\Gu{j}
= (\Gt{j}-\Gu{j}) \otimes N(\Gt{j})
= (\Gt{j}-\Gu{j}) \otimes (\Ga{j}{k}\otimes \Ga{k}{j}) = 0
\]
with \(k=1-j\).  So the generators~\(\Gu{j}\) are redundant.

\section{Manuel Köhler's invariant}
\label{sec:Koehler_invariant}

Now we extend the invariant
\(L_*(A) \cong \K_*(A\rtimes G) \oplus \K_*(A)\)
studied above as in the thesis of Manuel Köhler~\cite{Koehler:Thesis}.
Let
\[
D\defeq \cone(\Ga{0}{1})
\cong \setgiven{f\in \Cont_0(G\times [0,\infty), \C)}{f(g,0) = f(h,0)
  \text{ for all }g,h\in G}
\]
be the mapping cone of the embedding \(\C\to\Cont(G)\)
as constant functions.  This \(\Cst\)\nb-algebra
is commutative.  We identify it with \(\Cont_0(X)\) for
\begin{equation}
  \label{eq:cone_D_explicit}
  X\defeq \setgiven{(r\cos(\varphi), r\sin(\varphi))\in\R^2}
  {r\ge0,\ \varphi - \varphi_0 \in 2 \pi \Z/p}
\end{equation}
for any \(\varphi_0\in\R\),
with~\(G\)
acting by rotations around the origin, cyclically permuting the rays
in~\(X\).
This particular description will be used in the proof of
Proposition~\ref{pro:BS_self-dual}.

Adding~\(D\)
as an extra generator does not enlarge the equivariant bootstrap
class~\(\mathfrak{B}^G\).
So~\(\mathfrak{B}^G\)
is also the localising subcategory of~\(\KK^G\)
generated by \(\C\),
\(\Cont(G)\)
and~\(D\).
And the same localising subcategory is generated by the direct sum
\[
B \defeq \C \oplus \Cont(G) \oplus D.
\]
Here we use that idempotents in localising subcategories of~\(\KK^G\)
split because they are triangulated categories with countable direct
sums (see \cite{Neeman:Triangulated}*{§1.3}).

Let \(\Kring \defeq \KK^G_*(B,B)^\op\)
be the opposite of the \(\Z/2\)\nb-graded
endomorphism ring of~\(B\)
in~\(\KK^G\),
that is, \(\Kring_0 \defeq \KK^G_0(B,B)^\op\)
and \(\Kring_1 \defeq \KK^G_1(B,B)^\op\),
with the Kasparov composition product in reverse order as composition.
Let~\(A\) be an object of~\(\KK^G\).  Then
\[
F_*(A) \defeq \KK_*^G(B,A) =  \KK_0^G(B,A) \oplus  \KK_1^G(B,A)
\]
is a \(\Z/2\)\nb-graded
left \(\Kring\)\nb-module by the Kasparov product in reverse order.

Let \(\Gu{0},\Gu{1},\Gu{2}\)
be the projections to the summands \(\C\),
\(\Cont(G)\)
and~\(D\)
in~\(B\),
respectively.  These are even elements of~\(\Kring\),
and they are orthogonal idempotents whose sum is the unit element
in~\(\Kring\), that is,
\begin{align*}
  \Gu{j}\otimes \Gu{k}&= \delta_{j,k} \; \Gu{j}
  \qquad \text{for }j,k\in\{0,1,2\},\\
  \Gu{0}+\Gu{1}+\Gu{2}&= 1.
\end{align*}
Using these idempotent elements, we may describe~\(\Kring\)
as a ring of \(3\times3\)-matrices with entries
\begin{alignat*}{3}
  x_{00}&\in \KK^G_*(\C,\C),&\qquad
  x_{01}&\in \KK^G_*(\C,\Cont(G)),&\qquad
  x_{02}&\in \KK^G_*(\C,D),\\
  x_{10}&\in \KK^G_*(\Cont(G),\C),&\qquad
  x_{11}&\in \KK^G_*(\Cont(G),\Cont(G)),&\qquad
  x_{12}&\in \KK^G_*(\Cont(G),D),\\
  x_{20}&\in \KK^G_*(D,\C),&\qquad
  x_{21}&\in \KK^G_*(D,\Cont(G)),&\qquad
  x_{22}&\in \KK^G_*(D,D).
\end{alignat*}
The multiplication follows the matrix multiplication pattern as for
the ring~\(\littlering\)
that acts on the little invariant.  The idempotents~\(\Gu{j}\)
also decompose any \(\Z/2\)\nb-graded
\(\Kring\)\nb-module~\(M\)
as a direct sum \(M = M_0 \oplus M_1 \oplus M_2\)
with \(\Z/2\)\nb-graded
subgroups \(M_j \defeq \Gu{j}\cdot M\).
The action of~\(\Kring\)
on~\(M\)
is given by matrix-vector multiplication when we write elements
of~\(M\)
as column vectors with entries in~\(M_j\)
for \(j=0,1,2\).
We shall tacitly identify \(\Kring\)\nb-modules
with such triples \((M_j)_{j=0,1,2}\)
in the following.
Let~\(A\) be an object of~\(\KK^G\).  Then
\[
F_*(A)_0 = \K_*(A\rtimes G),\qquad
F_*(A)_1 = \K_*(A),\qquad
F_*(A)_2 = \KK_*^G(D,A),
\]
where~\(\K_*\) denotes ordinary K\nb-theory, because
\begin{align*}
  \KK_*^G(B,A) &\cong \KK_*^G(\C,A) \oplus \KK_*^G(\Cont(G),A) \oplus \KK_*^G(D,A)\\
  &\cong \K_*(A\rtimes G) \oplus \K_*(A) \oplus \KK_*^G(D,A).
\end{align*}
  
There are canonical equivariant \Star{}homomorphisms
\begin{equation}
  \label{eq:first_cone_sequence}
  \Sigma \Cont(G) \xrightarrow{\Ga{1}{2}}
  D \xrightarrow{\Ga{2}{0}}
  \C \xrightarrow{\Ga{0}{1}}
  \Cont(G),
\end{equation}
namely,
\(\Ga{1}{2}\colon \Sigma \Cont(G) = \Cont_0(G\times\R) \to
D=\Cont_0(X)\)
is the ideal inclusion corresponding to the subset
\(X\setminus\{0\} \subseteq X\),
and \(\Ga{2}{0}\colon \Cont_0(X) = D \to \C\)
is evaluation at~\(0\).
The maps in~\eqref{eq:first_cone_sequence} form an exact triangle
in~\(\KK^G\)
(see~\cite{Meyer-Nest:BC}).  This exact triangle induces a long
exact sequence of \(\KK\)\nb-groups,
the Puppe sequence for \(\C\to\Cont(G)\):
\begin{equation}
  \label{eq:K-theory_long_exact_sequence}
  \begin{tikzcd}[baseline=(current bounding box.west)]
    \K_0(A) \arrow[r, "\Ga{0}{1}^*"] &
    \K_0(A\rtimes G) \arrow[r, "\Ga{2}{0}^*"] &
    \KK_0^G(D,A) \arrow[d, "\Ga{1}{2}^*"] \\
    \KK_1^G(D,A) \arrow[u, "\Ga{1}{2}^*"] &
    \K_1(A\rtimes G) \arrow[l, "\Ga{2}{0}^*"] &
    \K_1(A).  \arrow[l, "\Ga{0}{1}^*"]
  \end{tikzcd}
\end{equation}
This long exact sequence implies that \(\KK_*^G(D,A)\)
is countable because \(\K_*(A)\)
and \(\K_*(A\rtimes G)\)
are countable for any separable \(\Cst\)\nb-algebra~\(A\).
Thus \(F_*(A)\)
is a countable \(\Z/2\)\nb-graded
\(\Kring\)\nb-module.
The functor~\(F\)
from \(\KK^G\)
to the category of countable \(\Z/2\)\nb-graded
\(\Kring\)\nb-modules
is the invariant studied by Manuel Köhler.  This functor is a stable
homological functor as in
\cites{Meyer-Nest:Homology_in_KK,Meyer:Homology_in_KK_II}, so that the
homological algebra machinery developed in these articles applies to
it.  Even more, \(F\)
is the universal stable homological functor defined by its kernel on
morphisms.  Roughly speaking, this says that the homological algebra
in~\(\KK^G\)
that is produced by~\(F\)
is equivalent to the homological algebra in the category of
\(\Z/2\)\nb-graded
countable \(\Kring\)\nb-modules.
While this is an important point for the proof of the main theorem
of~\cite{Koehler:Thesis}, there is no need to discuss this here.
The following is the main result of~\cite{Koehler:Thesis}:

\begin{theorem}[\cite{Koehler:Thesis}]
  \label{the:Koehler_invariant_UCT_classifies}
  Let \(A,C\)
  be separable \(G\)\nb-\(\Cst\)-algebras.
  Assume \(A\in\mathfrak{B}^G\).  Then there is a natural short exact sequence
  \[
  \Ext^1_\Kring\bigl(F_{1+*}(A),F_*(C)\bigr) \into \KK^G_*(A,C) \prto \Hom_\Kring\bigl(F_*(A),F_*(C)\bigr).
  \]
  If \(A,C\in\mathfrak{B}^G\),
  then an isomorphism \(\varphi\colon F_*(A) \to F_*(C)\)
  of \(\Z/2\)\nb-graded
  \(\Kring\)\nb-modules
  lifts to a \(\KK^G\)-equivalence
  in \(\KK^G_0(A,C)\).
  In particular, \(A,C\)
  are \(\KK^G\)-equivalent
  if and only if \(F_*(A) \cong F_*(C)\)
  as \(\Z/2\)\nb-graded \(\Kring\)\nb-modules.
\end{theorem}

The main part in the proof of the Universal Coefficient Theorem exact
sequence in Theorem~\ref{the:Koehler_invariant_UCT_classifies} is to
show that \(F_*(A)\)
for~\(A\) in~\(\KK^G\)
has a projective resolution of length~\(1\)
in the category of \(\Z/2\)\nb-graded
\(\Kring\)\nb-modules.
The naturality of the UCT exact sequence means that composing the image
of \(\Ext^1_\Kring\)
with an element \(\varphi\in \KK^G_*(A,C)\)
is equivalent to composing with \(F_*(\varphi)\)
in the Ext-theory over~\(\Kring\).
In particular, the \(\Ext^1_\Kring\)-part
is nilpotent.  Therefore, any lifting of an isomorphism in
\(\Hom_\Kring(F_*(A),F_*(C))\)
is invertible in~\(\KK^G\),
provided the UCT applies to both \(A\)
and~\(C\).
The general machinery of homological algebra in triangulated
categories in~\cite{Meyer-Nest:Homology_in_KK} also shows that a
separable \(G\)\nb-\(\Cst\)-algebra~\(A\)
belongs to~\(\mathfrak{B}^G\)
if and only if \(\KK^G_*(A,C)\)
is computed by the UCT for all~\(C\),
if and only if \(\KK^G_*(A,C)=0\)
for all~\(C\)
with \(F_*(C)=0\).

In addition to Theorem~\ref{the:Koehler_invariant_UCT_classifies},
Köhler also describes the range of the invariant~\(F_*\).
This result is particularly important for us.  Put in a nutshell, a
countable \(\Kring\)\nb-module
belongs to the essential range of~\(F_*\)
if and only if two sequences built from the \(\Kring\)\nb-module
structure are exact.  One of these exact sequences is the Puppe
sequence in~\eqref{eq:K-theory_long_exact_sequence}.  The most elegant
way to get the second exact sequence uses that the object~\(D\)
above is self-dual with respect to Baaj--Skandalis duality.  This is a
special feature of working with \(\Cst\)\nb-algebras
and equivariant~\(\KK^G\).
It seems to have no analogue in the homological algebra of
\(\Sring\)\nb-modules or in the algebraic topology of \(G\)\nb-spaces.

\begin{proposition}
  \label{pro:BS_self-dual}
  \(\BS(D) \cong \Sigma D\),
  where~\(\Sigma\)
  denotes suspension.  Hence Baaj--Skandalis duality induces an
  involutive automorphism of the ring~\(\Kring\)
  and a compatible isomorphism \(F_*(A) \cong F_*(\BS(A))\)
  -- but these do not preserve the \(\Z/2\)\nb-gradings.
\end{proposition}

\begin{proof}
  This is shown in~\cite{Koehler:Thesis}.  We sketch a proof that uses
  the geometric description of equivariant bivariant \(\K\)\nb-theory
  developed in~\cite{Emerson-Meyer:Correspondences}.  Define
  \(\Ga{1}{0}\in \KK^G_0(\Cont(G),\C)\)
  as in~\eqref{eq:alpha_10}.  Then \(\BS(D)\)
  is a mapping cone for~\(\Ga{1}{0}\).
  We shall describe this cone in a different way to show that it is
  \(\KK^G\)\nb-equivalent
  to~\(\Sigma D\).
  The groups \(\KK^G_*(\Cont_0(X),\Cont_0(Y))\)
  for two \(G\)\nb-manifolds
  \(X,Y\)
  are described in~\cite{Emerson-Meyer:Correspondences} through
  equivariant correspondences, following Connes and
  Skandalis~\cite{Connes-Skandalis:Longitudinal}.  In particular,
  some geometric \(\KK^G\)-cycle
  must give~\(\Ga{1}{0}\).
  Actually, the geometric cycles that we need are only wrong-way
  elements of equivariantly \(\K\)\nb-oriented
  smooth maps, which are treated in~\cite{Emerson-Meyer:Normal_maps}.
  We claim that \(\Ga{1}{0}\)
  is the wrong-way element~\(p!\)
  associated to the constant map \(p\colon G\to \pt\)
  with the obvious \(\K\)\nb-orientation,
  where~\(\pt\)
  denotes the one-point space.  The Kasparov cycle associated
  to~\(p!\)
  for a surjective submersion is, in general, represented by the
  Kasparov cycle of the family of Dirac operators along the fibres of
  the submersion; this is a form of the analytic index map.  In our
  \(0\)\nb-dimensional
  case, this boils down to the class in \(\KK^G_0(\Cont(G),\Cont(\pt))\)
  of \(\ell^2(G)\)
  with~\(\Cont(G)\)
  acting by poinwise multiplication and with zero Fredholm operator.
  This is exactly our~\(\Ga{1}{0}\).
  The analytic index is always equal to the topological index, which
  is defined using only Thom isomorphisms and open embeddings.  In
  general, the topological index map for a smooth \(\K\)\nb-oriented
  map \(X\to Y\)
  uses an equivariant embedding of~\(X\)
  into~\(\R^n\)
  with some linear action of~\(G\).
  In the case at hand, let~\(\R^2_\varrho\)
  denote~\(\R^2\)
  with the action of \(G=\Z/p\)
  by rotations by angles~\(2\pi k/p\);
  this is \(\K\)\nb-oriented in a canonical way.  Then
  \[
  G\to \R^2_\varrho,\qquad
  k\mapsto(\cos(2\pi k/p),\sin(2\pi k/p)),
  \]
  is a \(G\)\nb-equivariant
  \(\K\)\nb-oriented
  embedding.  Its normal bundle is \(G\times \R^2_\varrho \to G\).
  We choose a diffeomorphism from \(G\times \R^2_\varrho\)
  onto an open neighbourhood of~\(G\)
  in~\(\R^2_\varrho\).
  The topological index element
  \(p!_\mathrm{top} \in \KK^G_0(\Cont(G),\C)\)
  is the composite of the Thom isomorphism class in
  \(\KK^G_0(\Cont(G),\Cont_0(G\times \R^2_\varrho))\),
  the inclusion
  \(\Cont_0(G\times\R^2_\varrho) \hookrightarrow
  \Cont_0(\R^2_\varrho)\)
  as an ideal from the open inclusion
  \(G\times \R^2_\varrho \hookrightarrow \R^2_\varrho\),
  and the class of the Bott isomorphism in
  \(\KK^G_0(\Cont_0(\R^2_\varrho),\C)\).
  The general index theorem in
  \cite{Emerson-Meyer:Normal_maps}*{Theorem~6.1} contains the
  statement that \(p! = p!_\mathrm{top}\)
  in \(\KK^G_0(\Cont(G),\C)\).
  So \(\BS(D)\)
  is also a cone for \(p!_\mathrm{top}\).
  Here the ``cone'' is defined as the third entry of an exact triangle
  \[
  \Sigma \Cont(G) \to \BS(D) \to\Cont(G) \xrightarrow{p!_\mathrm{top}} \C,
  \]
  using the triangulated category structure on~\(\KK^G\).
  This cone is determined uniquely up to \(\KK^G\)-equivalence.
  The \(\KK^G\)-cycle
  above does not yet come from an equivariant \Star{}homomorphism.
  But since the Thom isomorphism and Bott isomorphism classes are
  invertible in \(\KK^G\),
  a cone for the ideal inclusion
  \(\Cont_0(G\times\R^2_\varrho) \hookrightarrow
  \Cont_0(\R^2_\varrho)\)
  is also a cone for~\(p!_\mathrm{top}\).
  The extension
  \(\Cont_0(G\times\R^2_\varrho) \into \Cont_0(\R^2_\varrho)
  \prto \Cont_0(\R^2_\varrho \setminus (G\times\R^2_\varrho))\)
  has long exact sequences in \(\KK^G\).
  So there is an exact triangle
  \[
  \Cont_0(G\times\R^2_\varrho) \to \Cont_0(\R^2_\varrho)
  \to \Cont_0(\R^2 \setminus (G\times\R^2_\varrho)) \to
  \Sigma \Cont_0(G\times\R^2_\varrho)
  \]
  in~\(\KK^G\).
  Hence the suspension of
  \(\Cont_0(\R^2_\varrho \setminus (G\times\R^2_\varrho))\)
  is also a cone of~\(p!_\mathrm{top}\).
  Now we choose the tubular neighbourhood diffeomorphism
  \(G\times \R^2_\varrho \to \R^2_\varrho\)
  so that each \(\{k\} \times \R^2\)
  is mapped diffeomorphically onto an entire open sector
  \[
    \setgiven{(r\cos \varphi,r\sin \varphi)}
    {0<r,\ \varphi \in ((2k-1)\cdot\pi/p,(2k+1)\cdot\pi/p)}.
  \]
  Its complement in~\(\R^2\) is the union of the rays
  \((r\cos \varphi,r\sin \varphi)\)
  with \(0\le r <\infty\)
  and \(\varphi \in (2k-1)\cdot \pi/p\),
  \(k\in\Z/p\).
  This is diffeomorphic to the space~\(X\)
  in~\eqref{eq:cone_D_explicit}.  So Baaj--Skandalis duality
  maps~\(D\) to something \(\KK^G\)-equivalent to~\(\Sigma D\).

  Now we replace the generator~\(B\)
  by \(B\oplus \Sigma B\).
  This is more precise because we are, anyway, taking the
  \(\Z/2\)\nb-graded
  group \(\KK^G_*(B,A) = \KK^G_0(B,A) \oplus \KK^G_0(\Sigma B,A)\).
  The isomorphisms \(\BS(\C) \cong \Cont(G)\),
  \(\BS(\Cont(G)) \cong \C\),
  \(\BS(D) \cong \Sigma D\)
  in \(\KK^G\)
  and their suspensions combine to an isomorphism
  \(\BS(B\oplus \Sigma B) \cong B\oplus \Sigma B\).
  This induces an involutive automorphism~\(\BS\)
  of the ring
  \(\Kring = \KK^G_0(B\oplus\Sigma B, B \oplus \Sigma B)^\op\),
  and it induces an isomorphism
  \[
  \BS\colon F_*(A) \defeq \KK^G_0(B\oplus \Sigma B,A) \to
  \KK^G_0(\BS(B\oplus \Sigma B),\BS(A)) \cong F_*(\BS(A)),
  \]
  which is compatible with the \(\Kring\)\nb-module
  structure in the sense that
  \(\BS(x)\otimes \BS(y) = \BS(x\otimes y)\)
  for all \(x\in \Kring\),
  \(y\in F_*(A)\).
  We must be aware, however, that \(\BS\)
  does not preserve the \(\Z/2\)\nb-gradings
  on \(\Kring\)
  and~\(F_*(A)\).
  When we decompose them according to the direct sum decomposition
  \(B = \C \oplus \Cont(G) \oplus D\),
  then the grading is preserved on \(\C\)
  and \(\Cont(G)\) and reversed on~\(D\).
\end{proof}

Baaj--Skandalis duality applied to the exact
triangle~\eqref{eq:first_cone_sequence} gives another
exact triangle in~\(\KK^G\):
\[
\Sigma \BS(\Cont(G))  \xrightarrow{\BS(\Ga{1}{2})}
\BS(D) \xrightarrow{\BS(\Ga{2}{0})}
\BS(\C) \xrightarrow{\BS(\Ga{0}{1})}
\BS(\Cont(G)).
\]
The \(\KK^G\)-equivalences \(\BS(\Cont(G)) \sim \C\),
\(\BS(\C) \sim \Cont(G)\) and \(\BS(D) \sim \Sigma D\) show that it
is isomorphic to an exact triangle
\begin{equation}
  \label{eq:second_cone_sequence}
  \Sigma \C \xrightarrow{\Sigma \Ga{0}{2}}
  \Sigma D \xrightarrow{\Ga{2}{1}}
  \Cont(G) \xrightarrow{\Ga{1}{0}}
  \C.
\end{equation}
Implicitly, this defines \(\Ga{0}{2}\), \(\Ga{2}{1}\)
and~\(\Ga{1}{0}\) as the images of \(\Ga{1}{2}\), \(\Ga{2}{0}\)
and~\(\Ga{0}{1}\) under~\(\BS\), composed with the chosen
\(\KK^G\)-equivalences.  In other words, the automorphism
\(\BS\colon \Kring\to \Kring\) satisfies
\[
\BS(\Ga{1}{2}) = \Sigma \Ga{0}{2},\qquad
\BS(\Ga{2}{0}) = \Ga{2}{1},\qquad
\BS(\Ga{0}{1}) = \Ga{1}{0}.
\]
Here we have viewed the \(\KK^G\)-classes
of~\(\Ga{j}{k}\)
as elements of~\(\Kring\).  Since~\(\BS\) is an involution, we also get
\[
\BS(\Ga{0}{2}) = \Sigma \Ga{1}{2},\qquad
\BS(\Ga{2}{1}) = \Ga{2}{0},\qquad
\BS(\Ga{1}{0}) = \Ga{0}{1}.
\]
The elements \(\Ga{j}{k}\in \Kring\)
are even or odd depending on the suspensions in
\eqref{eq:first_cone_sequence} and~\eqref{eq:second_cone_sequence}.
Namely, \(\Ga{1}{2}\)
and \(\Ga{2}{1}\)
are odd and \(\Ga{2}{0}\),
\(\Ga{0}{1}\),
\(\Ga{0}{2}\),
\(\Ga{1}{0}\)
are even.  Since~\(F_*\)
is a homological functor, the exact
triangle~\eqref{eq:second_cone_sequence} induces a second exact
sequence, called the \emph{dual Puppe sequence},
\begin{equation}
  \label{eq:dual_Puppe}
  \begin{tikzcd}[baseline=(current bounding box.west)]
    \KK_0^G(D,A) \arrow[r, "\Ga{0}{2}^*"] &
    \K_0(A\rtimes G) \arrow[r, "\Ga{1}{0}^*"] &
    \K_0(A) \arrow[d, "\Ga{2}{1}^*"] \\
    \K_1(A) \arrow[u, "\Ga{2}{1}^*"] &
    \K_1(A\rtimes G) \arrow[l, "\Ga{1}{0}^*"] &
    \KK_1^G(D,A). \arrow[l, "\Ga{0}{2}^*"]
  \end{tikzcd}
\end{equation}

\begin{definition}
  \label{def:module_notation}
  Let~\(M\) be a \(\Kring\)\nb-module, \(j,k\in\{0,1,2\}\) and
  \(x\in \Kring\) with \(x = \Gu{j}\otimes x \otimes \Gu{k}\).
  Write \(M = M_0 \oplus M_1 \oplus M_2\) as above.  We
  write~\(x^M\) for the map \(M_k \to M_j\),
  \(y\mapsto y\otimes x\).  In particular, this
  defines~\(\Ga{j}{k}^M\) for \(j\neq k\).  Later, we will also get
  elements~\(\Gt{j}^M\) for \(j=0,2\) and~\(\Gs{j}^M\) for
  \(j=1,2\).
\end{definition}

We can now characterise the range of the invariant~\(F_*\).

\begin{definition}
  \label{def:exact_R-module}
  A \(\Kring\)\nb-module~\(M\)
  is \emph{exact} if the two sequences of
  Abelian groups
  \[
  \begin{tikzcd}[row sep=large, column sep=large]
    &M_1 \ar[dr, shift right, "\Ga{2}{1}^M"'] \ar[dl, shift right, "\Ga{0}{1}^M"']&\\
    M_0 \ar[rr, shift right, "\Ga{2}{0}^M"'] \ar[ur, shift right, "\Ga{1}{0}^M"']&&
    M_2 \ar[ul, shift right, "\Ga{1}{2}^M"'] \ar[ll, shift right, "\Ga{0}{2}^M"']
  \end{tikzcd}
  \]
  are exact.
\end{definition}

\begin{theorem}[\cite{Koehler:Thesis}]
  \label{the:range_F}
  Let~\(M\)
  be a countable \(\Z/2\)\nb-graded
  \(\Kring\)\nb-module.  The following are equivalent:
  \begin{enumerate}
  \item \label{en:range_F_BG}%
    \(M = F_*(A)\) for some~\(A\) in~\(\mathfrak{B}^G\);
  \item \label{en:range_F_KKG}%
    \(M = F_*(A)\) for some~\(A\) in~\(\KK^G\);
  \item \label{en:range_F_exact}%
    \(M\) is exact;
  \item \label{en:range_F_proj}%
    \(M\) has a projective \(\Kring\)\nb-module resolution of length~\(1\).
  \end{enumerate}
\end{theorem}

\begin{proof}[Straightforward parts of the proof]
  It is trivial that \ref{en:range_F_BG}
  implies~\ref{en:range_F_KKG}.  Since the Puppe
  sequence~\eqref{eq:K-theory_long_exact_sequence} and the dual
  Puppe sequence~\eqref{eq:dual_Puppe} are exact,
  \ref{en:range_F_KKG} implies~\ref{en:range_F_exact}.
  Condition~\ref{en:range_F_proj} implies~\ref{en:range_F_BG} by
  lifting projective resolutions in the Abelian approximation to
  projective resolutions in the triangulated category (see
  \cite{Meyer-Nest:Homology_in_KK}*{Theorem~59}).  The highly
  non-trivial point is that~\ref{en:range_F_exact}
  implies~\ref{en:range_F_proj}, that is, exact modules have
  projective resolutions of length~\(1\).  We refer
  to~\cite{Koehler:Thesis} for this.  This together with the
  abstract UCT in \cite{Meyer-Nest:Homology_in_KK}*{Theorem~66} also
  implies Theorem~\ref{the:Koehler_invariant_UCT_classifies}.
\end{proof}

\subsection{Generators and relations for Köhler's ring}
\label{sec:generators_relations}

We are going to describe the ring~\(\Kring\) by generators and
relations.  This makes it easier to apply Theorems
\ref{the:Koehler_invariant_UCT_classifies} and~\ref{the:range_F} to
classify objects of~\(\mathfrak{B}^G\) up to \(\KK^G\)-equivalence.
Recall that
\[
  N(t) = 1 + t + \dotsb + t^{p-1}.
\]

\begin{theorem}
  \label{the:generators_and_relations}
  The ring~\(\Kring\)
  is the universal ring generated by the elements \(\Gu{j}\)
  for \(j=0,1,2\)
  and~\(\Ga{j}{k}\)
  for \(0\le j,k \le 2\)
  with \(j\neq k\) with the following relations:
  \begin{align}
    \label{eq:1_vs_1}
    \Gu{j}\otimes\Gu{k}&= \delta_{j,k} \Gu{j}
    \qquad \text{for }j,k\in\{0,1,2\},\\
    \label{eq:1_sum_is_unit}
    \Gu{0}+\Gu{1}+\Gu{2}&= 1.\\
    \label{eq:alpha_vs_1}
    \Gu{j}\otimes \Ga{j}{k}\otimes \Gu{k}&= \Ga{j}{k},\\
    \label{eq:alpha_exact}
    \Ga{j}{k}\otimes \Ga{k}{m}&= 0\qquad \text{if } \{j,k,m\}=\{0,1,2\}, \\
    \label{eq:alpha_010}
    \Ga{0}{1}\otimes\Ga{1}{0}&= N(\Gu{0}- \Ga{0}{2}\otimes\Ga{2}{0}),\\
    \label{eq:alpha_101}
    \Ga{1}{0}\otimes\Ga{0}{1}&= N(\Gu{1}- \Ga{1}{2}\otimes\Ga{2}{1}),\\
    \label{eq:alpha_2}
    p\cdot \Gu{2} &= N(\Gu{2}- \Ga{2}{0}\otimes\Ga{0}{2})
        +  N(\Gu{2}- \Ga{2}{1}\otimes\Ga{1}{2}).
  \end{align}
  The \(\Z/2\)\nb-grading on~\(\Kring\) is such that \(\Ga{1}{2}\)
  and \(\Ga{2}{1}\) are odd and all other generators are even.
\end{theorem}

This theorem follows from the computations in~\cite{Koehler:Thesis},
but is not stated there as such.  It will take a while to prove it.
The relations \eqref{eq:1_vs_1} and~\eqref{eq:1_sum_is_unit}
say that the elements \(\Gu{j}\)
are orthogonal idempotents that sum up to~\(1\).
This is equivalent to the \(3\times3\)-matrix
decomposition of the ring~\(\Kring\).
Equation~\eqref{eq:alpha_vs_1} says that~\(\Ga{j}{k}\)
belongs to the \(j,k\)-th
entry in the matrix description of~\(\Kring\).
Condition~\eqref{eq:alpha_exact} consists of six equations and says
that the composite of any two consecutive maps in the two
exact triangles \eqref{eq:first_cone_sequence}
and~\eqref{eq:second_cone_sequence} vanishes -- as needed for an exact
triangle.  Our next goal is to understand the relations
\eqref{eq:alpha_010} and~\eqref{eq:alpha_101}.  First we add some
convenient definitions to the relations above:
\begin{equation}
  \label{eq:def_st}
  \begin{aligned}
    \Gt{0}&\defeq \Gu{0}- \Ga{0}{2}\otimes \Ga{2}{0},&\qquad
    \Gs{1}&\defeq \Gu{1}- \Ga{1}{2}\otimes \Ga{2}{1},\\
    \Gt{2}&\defeq \Gu{2}- \Ga{2}{0}\otimes \Ga{0}{2},&\qquad
    \Gs{2}&\defeq \Gu{2}- \Ga{2}{1}\otimes \Ga{1}{2}.
  \end{aligned}
\end{equation}
Since~\(\Ga{j}{k}\)
belongs to the \(j,k\)-component
of~\(\Kring\),
\(\Gt{j}\)
for \(j=0,2\)
and \(\Gs{j}\)
for \(j=1,2\)
belong to the \(j,j\)-component
of~\(\Kring\),
respectively.  The elements \(\Gt{j}\)
and~\(\Gs{j}\)
are even.  The relations \eqref{eq:alpha_010},
\eqref{eq:alpha_101} and~\eqref{eq:alpha_2} in
Theorem~\ref{the:generators_and_relations} say that
\begin{align}
  \label{eq:alpha_010_N}
  \Ga{0}{1}\otimes \Ga{1}{0}&= N(\Gt{0}),\\
  \label{eq:alpha_101_N}
  \Ga{1}{0}\otimes \Ga{0}{1}&= N(\Gs{1}),\\
  \label{eq:alpha_2_N}
  p\cdot \Gu{2} &= N(\Gt{2}) +  N(\Gs{2}).
\end{align}
The definitions of \(\Gt{j}\) and~\(\Gs{j}\) say that
\begin{align}
  \label{eq:alpha_020}
  \Ga{0}{2}\otimes \Ga{2}{0}&=  \Gu{0}- \Gt{0},\\
  \label{eq:alpha_121}
  \Ga{1}{2}\otimes \Ga{2}{1}&=  \Gu{1}- \Gs{1},\\
  \label{eq:alpha_202}
  \Ga{2}{0}\otimes \Ga{0}{2}&=  \Gu{2}- \Gt{2},\\
  \label{eq:alpha_212}
  \Ga{2}{1}\otimes \Ga{1}{2}&=  \Gu{2}- \Gs{2}.
\end{align}
The equations \eqref{eq:alpha_010_N}, \eqref{eq:alpha_101_N} and
\eqref{eq:alpha_020}--\eqref{eq:alpha_212} express the products
\(\Ga{j}{k}\otimes \Ga{k}{j}\)
for all \(j,k\in\{0,1,2\}\)
with \(j\neq k\)
as a polynomial in \(\Gt{j}\)
or~\(\Gs{j}\).
All other products of two \(\alpha\)\nb-generators
vanish by~\eqref{eq:alpha_exact}.  Before we prove
Theorem~\ref{the:generators_and_relations}, we list
further relations among the elements \(\Gt{j}\),
\(\Gs{j}\)
and~\(\Ga{j}{k}\)
that follow from the relations in
Theorem~\ref{the:generators_and_relations}.

\begin{proposition}
  \label{pro:further_relations}
  The relations in Theorem~\textup{\ref{the:generators_and_relations}}
  imply the following relations:
  \begin{alignat}{2}
    \label{eq:tp}
    \Gt{j}^p &=1&\qquad &\text{for }j=0,2,\\
    \label{eq:sp}
    \Gs{j}^p &=1&\qquad &\text{for }j=1,2,\\
    \label{eq:ta_01}%
    \Gt{0}\otimes \Ga{0}{1}&= \Ga{0}{1},\\
    \label{eq:ta_02}%
    \Gt{0}\otimes \Ga{0}{2}&= \Ga{0}{2}\otimes \Gt{2},\\
    \label{eq:ta_21}%
    \Gt{2}\otimes \Ga{2}{1}&= \Ga{2}{1},\\
    \label{eq:ta_20}%
    \Gt{2}\otimes \Ga{2}{0}&= \Ga{2}{0}\otimes \Gt{0},\\
    \label{eq:at_10}%
    \Ga{1}{0}\otimes \Gt{0}&= \Ga{1}{0},\\
    \label{eq:at_12}%
    \Ga{1}{2}\otimes \Gt{2}&= \Ga{1}{2},\\
    \label{eq:sa_10}%
    \Gs{1}\otimes \Ga{1}{0}&= \Ga{1}{0},\\
    \label{eq:sa_12}%
    \Gs{1}\otimes \Ga{1}{2}&= \Ga{1}{2}\otimes \Gs{2},\\
    \label{eq:sa_20}%
    \Gs{2}\otimes \Ga{2}{0}&= \Ga{2}{0},\\
    \label{eq:sa_21}%
    \Gs{2}\otimes \Ga{2}{1}&= \Ga{2}{1}\otimes \Gs{1},\\
    \label{eq:as_01}%
    \Ga{0}{1}\otimes \Gs{1}&= \Ga{0}{1},\\
    \label{eq:as_02}%
    \Ga{0}{2}\otimes \Gs{2}&= \Ga{0}{2},\\
    \label{eq:st_2}%
    (\Gu{2}- \Gs{2})\otimes (\Gu{2}- \Gt{2})&= 0,\\
    \label{eq:ts_2}%
    (\Gu{2}- \Gt{2})\otimes (\Gu{2}- \Gs{2})&= 0,\\
    \label{eq:Na_20}%
    N(\Gt{2})\otimes \Ga{2}{0}&= 0,\\
    \label{eq:Na_02}%
    N(\Gt{0})\otimes \Ga{0}{2}&= 0,\\
    \label{eq:aN_20}%
    \Ga{2}{0}\otimes N(\Gt{0}) &= 0,\\
    \label{eq:aN_02}%
    \Ga{0}{2}\otimes N(\Gt{2}) &= 0,\\
    \label{eq:Na_21}%
    N(\Gs{2})\otimes \Ga{2}{1}&= 0,\\
    \label{eq:Na_12}%
    N(\Gs{1})\otimes \Ga{1}{2}&= 0,\\
    \label{eq:aN_21}%
    \Ga{2}{1}\otimes N(\Gs{1}) &= 0,\\
    \label{eq:aN_12}%
    \Ga{1}{2}\otimes N(\Gs{2}) &= 0.
  \end{alignat}
\end{proposition}

\begin{proof}
  Equations \eqref{eq:ta_02}, \eqref{eq:ta_20}, \eqref{eq:sa_12}
  and~\eqref{eq:sa_21} follow from the definitions of the elements
  \(\Gt{j}\) and~\(\Gs{j}\).  One of these computations is
  \[
  \Gt{0}\otimes \Ga{0}{2}= (\Gu{0}- \Ga{0}{2}\otimes \Ga{2}{0})\otimes \Ga{0}{2}
  = \Ga{0}{2}- \Ga{0}{2}\otimes \Ga{2}{0}\otimes \Ga{0}{2}
  = \Ga{0}{2}\otimes \Gt{2}.
  \]
  The others are the same with different indices.
  The definitions of the elements \(\Gt{j}\)
  and~\(\Gs{j}\)
  and Equation~\eqref{eq:alpha_exact} imply \eqref{eq:ta_01},
  \eqref{eq:ta_21}, \eqref{eq:at_10}, \eqref{eq:at_12},
  \eqref{eq:sa_10}, \eqref{eq:sa_20}, \eqref{eq:as_01}
  and~\eqref{eq:as_02}.  One of these computations is
  \[
  \Gt{0}\otimes \Ga{0}{1}= (\Gu{0}- \Ga{0}{2}\otimes \Ga{2}{0})\otimes \Ga{0}{1}
  = \Ga{0}{1}- \Ga{0}{2}\otimes \Ga{2}{0}\otimes \Ga{0}{1}= \Ga{0}{1}.
  \]
  The others are the same with different indices.  A similar
  computation shows \eqref{eq:st_2} and~\eqref{eq:ts_2}.  One
  computation suffices:
  \[
  (\Gu{2}- \Gs{2})\otimes (\Gu{2}- \Gt{2})
  = \Ga{2}{1}\otimes \Ga{1}{2}\otimes \Ga{2}{0}\otimes \Ga{0}{2}
  = 0.
  \]
  Equations \eqref{eq:alpha_2_N} and~\eqref{eq:sa_20}
  imply~\eqref{eq:Na_20}:
  \[
  N(\Gt{2})\otimes \Ga{2}{0}
  = (p\cdot \Gu{2}- N(\Gs{2}))\otimes \Ga{2}{0}
  = (p- p)\cdot \Ga{2}{0}
  = 0.
  \]
  Similar computations give \eqref{eq:aN_02}, \eqref{eq:Na_21}
  and~\eqref{eq:aN_12}.  Equation~\eqref{eq:Na_02} holds because
  \[
  N(\Gt{0})\otimes \Ga{0}{2}= \Ga{0}{1}\otimes \Ga{1}{0}\otimes \Ga{0}{2}=0.
  \]
  Similar computations give \eqref{eq:aN_20}, \eqref{eq:Na_12},
  \eqref{eq:aN_21}.  We use \(t^p-1 = -N(t)(1-t)\)
  and~\eqref{eq:alpha_exact} to check \(\Gt{0}^p = \Gu{0}\):
  \[
  \Gu{0}- \Gt{0}^p = \Ga{0}{1}\otimes \Ga{1}{0}\otimes \Ga{0}{2}\otimes \Ga{2}{0}= 0,
  \]
  A similar computation gives \(\Gs{1}^p - \Gu{1}= 0\).
  Finally,
  \[
  \Gu{2}- \Gt{2}^p
  = (\Gu{2}- \Gt{2})\otimes N(\Gt{2})
  = (\Gu{2}- \Gt{2})\otimes (p\cdot \Gu{2} - N(\Gs{2})).
  \]
  This vanishes by~\eqref{eq:ts_2} because the polynomial \(p-N\)
  vanishes at~\(1\)
  and therefore may be written as \((1-s)f(s)\)
  for some polynomial~\(f\).
  A similar computations gives \(\Gu{2}- \Gs{2}^p = 0\).
\end{proof}

As a first step in the proof of
Theorem~\ref{the:generators_and_relations}, we justify the definitions
of \(\Gt{0}\)
and~\(\Gs{1}\)
above.  Namely, we show that these elements are the canonical
generators of \(\Sring \defeq \Z[t]/(t^p-1)\)
in the description of~\(\littlering\)
in Lemma~\ref{lem:little_invariant_ring}.

\begin{lemma}
  \label{lem:alpha_2_compose}
  Describe \(\KK^G_0(\C,\C)\)
  and \(\KK^G_0(\Cont(G),\Cont(G))\)
  as \(\Sring\defeq \Z[t]/(t^p-1)\)
  as in Lemma~\textup{\ref{lem:little_invariant_ring}}.
  Then \(\Ga{0}{2}\otimes \Ga{2}{0}\)
  and \(\Ga{1}{2}\otimes \Ga{2}{1}\) become \(1 - t \in \Sring\).
\end{lemma}

\begin{proof}
  The proof of Proposition~\ref{pro:BS_self-dual} describes~\(\Ga{0}{2}\)
  as the composite of the Thom isomorphism in
  \(\KK^G_0(\C,\Cont_0(\R^2_\varrho))\)
  and the restriction map
  \(\Cont_0(\R^2_\varrho) \to \Cont_0(X) = D\).
  And \(\Ga{2}{0}\)
  is represented by the \Star{}homomorphism \(\Cont_0(X) \to \C\),
  \(f\mapsto f(0)\).
  Thus \(\Ga{0}{2}\otimes \Ga{2}{0}\)
  is represented by the composite of the Thom isomorphism in
  \(\KK^G_0(\C,\Cont_0(\R^2_\varrho))\)
  and the evaluation map \(\Cont_0(\R^2_\varrho) \to \C\),
  \(f\mapsto f(0)\).
  The Thom isomorphism uses the \(\K\)\nb-orientation
  on~\(\R^2_\varrho\) that is given by the obvious complex structure
  on the tangent bundle.  Namely,
  \(T(\R^2_\varrho)\cong \C_\varrho\),
  where the action of~\(G\)
  on~\(\C_\varrho\)
  is given by the character \(k\bmod p\mapsto \exp(2\pi\ima k/p)\).
  The Thom isomorphism involves a Fredholm operator on the
  \(\Z/2\)\nb-graded
  vector bundle
  \(\Lambda^* T(\R^2_\varrho) = \C_1 \ominus \C_\varrho\),
  where~\(\C_1\)
  denotes the trivial vector bundle with the trivial action of~\(G\)
  in even parity, and~\(\ominus\)
  signifies that~\(\C_\varrho\)
  has odd parity.  Restricting to~\(0\),
  the Fredholm operator becomes redundant and we get the class in the
  representation ring \(\KK^G_0(\C,\C)\)
  of \(1\ominus\varrho\),
  the trivial representation minus the character~\(\varrho\).
  This corresponds to the element \(1-t\)
  in \(\KK^G_0(\C,\C) \cong \Sring\).

  The ring automorphism~\(\BS\)
  maps \(\Ga{0}{2}\otimes \Ga{2}{0}\)
  to \(\Ga{1}{2}\otimes \Ga{2}{1}\)
  and it induces the identity map on~\(\Sring\).
  So the claim about \(\Ga{1}{2}\otimes \Ga{2}{1}\) follows.
\end{proof}

Lemma~\ref{lem:alpha_2_compose} shows that we do not need the
generators~\(\Gt{0}\)
and~\(\Gs{1}\)
of~\(\littlering\)
as generators of~\(\Kring\).
Together with the relations of~\(\littlering\)
listed after Lemma~\ref{lem:little_invariant_ring},
Lemma~\ref{lem:alpha_2_compose} implies the relations
\eqref{eq:alpha_010} and~\eqref{eq:alpha_101} in
Theorem~\ref{the:generators_and_relations}.

The ring~\(\littlering\)
described in Lemma~\ref{lem:little_invariant_ring} gives four of the
nine matrix entries of the ring~\(\Kring\).
Now we describe the remaining five.  First, we describe
\(\KK^G_*(\C,D)\).
The exact triangle~\eqref{eq:first_cone_sequence} induces an exact
sequence
\begin{multline*}
  0 \to \KK^G_0(\C,D)
  \xrightarrow{(\Ga{2}{0})_*} \KK^G_0(\C,\C)
  \\\xrightarrow{(\Ga{0}{1})_*} \KK^G_0(\C,\Cont(G))
  \xrightarrow{(\Ga{1}{2})_*}
  \KK^G_1(\C,D) \to 0.
\end{multline*}
Lemma~\ref{lem:little_invariant_ring} shows that the map
\((\Ga{0}{1})_*\)
here is equivalent to the augmentation character
\(\tau\colon \Sring \to \Z\),
which is surjective.  So \(\KK^G_1(\C,D)=0\)
and \((\Ga{2}{0})_*\)
is an isomorphism \(\KK^G_0(\C,D)\congto\ker \tau\).
Now \(\ker \tau\)
is the principal ideal
generated by
\(t-1 \in \Sring\),
and it is isomorphic to \(\Z[t]/(N(t))\)
because \(f\in \Z[t]\)
satisfies \(t^p-1\mid (t-1)f\)
if and only if \(N(t) \mid f\).
Lemma~\ref{lem:alpha_2_compose} shows that \(1-t \in \Sring\)
is the composite \(\Ga{0}{2}\otimes \Ga{2}{0}\).
Hence there is an isomorphism
\[
\Z[t]/(N(t)) \to \KK^G_0(\C,D),\qquad
f\mapsto f(\Gt{0})\otimes \Ga{0}{2}.
\]

Then Baaj--Skandalis duality implies
\(\KK^G_0(\Cont(G),D)=0\)
and that the following map is an isomorphism:
\[
\Z[t]/(N(t)) \to \KK^G_1(\Cont(G),D),\qquad
f\mapsto f(\Gs{1})\otimes \Ga{1}{2}.
\]

The exact triangle~\eqref{eq:first_cone_sequence}
also implies an exact sequence for \(\KK^G_*(\blank,\C)\).
This has the following form:
\begin{multline*}
  0 \leftarrow \KK^G_0(D,\C)
  \xleftarrow{(\Ga{2}{0})^*} \KK^G_0(\C,\C)
  \\\xleftarrow{(\Ga{0}{1})^*} \KK^G_0(\Cont(G),\C)
  \xleftarrow{(\Ga{1}{2})^*} \KK^G_1(D,\C) \leftarrow 0.
\end{multline*}
Lemma~\ref{lem:little_invariant_ring} shows that the map
\((\Ga{0}{1})^*\)
here is equivalent to the map
\(\Z\to \Sring = \Z[t]/(t^p-1)\),
\(k\mapsto k\cdot N(t)\).
This map is injective.  As above, it follows that
\(\KK^G_1(D,\C)=0\) and that there is an isomorphism
\[
\Z[t]/(N(t)) \to \KK^G_0(D,\C),\qquad
f\mapsto  \Ga{2}{0} \otimes f(\Gt{0}).
\]
Baaj--Skandalis duality now gives an analogous
isomorphism
\[
\Z[t]/(N(t)) \to \KK^G_*(D,\Cont(G)),\qquad
f\mapsto \Ga{2}{1} \otimes f(\Gs{1}).
\]
Since \(\Ga{2}{1}\)
is odd, this component of~\(\Kring\)
is only odd.  Finally, it remains to describe the diagonal entry
\(\KK^G_*(D,D)\).  Here we use the cyclic long exact sequence
\begin{equation}
  \label{eq:exact_sequence_KKDD}
  \begin{tikzcd}[baseline=(current bounding box.west),column sep=large]
    \KK^G_1(D,\Cont(G)) \ar[r, "(\Ga{1}{2})_*"] &
    \KK^G_0(D,D) \arrow[r, "(\Ga{2}{0})_*"] &
    \KK^G_0(D,\C) \arrow[d] \\
    0 \arrow[u] &
    \KK_1^G(D,D) \arrow[l] &
    0.  \arrow[l]
  \end{tikzcd}
\end{equation}
Here we have used that \(\KK^G_*(D,\C)\)
lives in even and \(\KK^G_*(D,\Cont(G))\)
in odd parity.  So \(\KK_1^G(D,D)=0\),
the map
\((\Ga{1}{2})_*\colon \KK^G_1(D,\Cont(G)) \to \KK^G_0(D,D)\)
is injective, and its cokernel is identified by~\((\Ga{2}{0})_*\)
with \(\KK^G_0(D,\C)\).
Define \(\Gt{2}\)
and~\(\Gs{2}\)
by~\eqref{eq:def_st}.  Then \eqref{eq:alpha_202}
and~\eqref{eq:alpha_212} say that
\[
\Ga{2}{0}\otimes \Ga{0}{2}= \Gu{2}- \Gt{2},\qquad
\Ga{2}{1}\otimes \Ga{1}{2}= \Gu{2}- \Gs{2}.
\]
It follows from the definition of~\(\Gs{1}\) above that
\[
\Ga{2}{1}\otimes (\Gu{1}- \Gs{1})^j \otimes \Ga{1}{2}
= \Ga{2}{1}\otimes (\Ga{1}{2}\otimes \Ga{2}{1})^j\otimes \Ga{1}{2}
= (\Gu{2}- \Gs{2})^{j+1}.
\]
The computations above imply that the elements of the form
\(\Ga{2}{1} \otimes (\Gu{1}- \Gs{1})^j \)
for \(j=0,\dotsc,p-2\)
form a \(\Z\)\nb-module
basis of \(\KK^G_1(D,\Cont(G))\).
Their images in \(\KK^G_0(D,D)\)
are \((\Gu{2}- \Gs{2})^j\)
for \(j=1,\dotsc,p-1\).
These form a \(\Z\)\nb-module
basis of the image of \(\KK^G_1(D,\Cont(G))\)
in \(\KK^G_0(D,D)\).
The definitions of \(\Gt{0}\)
and~\(\Gt{2}\) in~\eqref{eq:def_st} imply
\[
\Ga{0}{2}\otimes (\Gu{2}- \Gt{2})
= \Ga{0}{2}\otimes \Ga{2}{0}\otimes \Ga{0}{2}
= (\Gu{0}- \Gt{0}) \otimes \Ga{0}{2}.
\]
So
\((\Gu{2}- \Gt{2})^j \otimes \Ga{2}{0} = \Ga{2}{0} \otimes
(\Gu{0}- \Gt{0})^j\).
These elements for \(j=0,\dotsc,p-2\)
form a \(\Z\)\nb-module
basis of \(\KK^G_0(D,\C)\)
by the computation above.  Putting things together, we conclude that
the elements \((\Gu{2}- \Gt{2})^j\)
for \(j=0,\dotsc,p-2\)
and \((\Gu{2}- \Gs{2})^j\)
for \(j=1,\dotsc,p-1\)
form a \(\Z\)\nb-module basis of~\(\KK^G_0(D,D)\).

\begin{lemma}
  \label{lem:norm_relation_ts2}
  \(N(\Gt{2}) + N(\Gs{2}) = p\).
\end{lemma}

\begin{proof}
  The definitions of \(\Gt{0}\) and~\(\Gt{2}\) in~\eqref{eq:def_st}
  imply \(f(\Gt{2}) \Ga{2}{0} = \Ga{2}{0} f(\Gt{0})\) for all
  polynomials~\(f\).  Thus
  \(N(\Gt{2}) \Ga{2}{0} = \Ga{2}{0} N(\Gt{0})\).  Using Lemmas
  \ref{lem:little_invariant_ring} and~\ref{lem:alpha_2_compose}, we
  rewrite this further as
  \(\Ga{2}{0} N(\Gt{0}) = \Ga{2}{0} \otimes \Ga{0}{1}\otimes
  \Ga{1}{0} = 0\).  Since the
  sequence~\eqref{eq:exact_sequence_KKDD} is exact, \(N(\Gt{2})\)
  belongs to the image of \(\KK^G_1(D,\Cont(G))\) in
  \(\KK^G_0(D,\C)\).  And \((1- \Gs{2})^j\) for \(j=1,\dotsc,p-1\)
  is a \(\Z\)\nb-module basis for this image.  This is also a basis
  for the principal ideal in \(\Z[\Gs{2}]/(\Gs{2}^p-1)\) generated
  by \(1- \Gs{2}\).  So there is a unique element \(f(\Gs{2})\) in
  this principal ideal with \(N(\Gt{2}) = f(\Gs{2})\).  Since
  \(N(1)=p\), \(p-N\) belongs to this principal ideal, and we must
  show that \(f = p-N\).  The Baaj--Skandalis duality
  involution~\(\BS\) on the ring~\(\Kring\) acts on the component
  \(\KK^G_0(D,D)\) by exchanging \(\Gs{2}\leftrightarrow \Gt{2}\).
  Therefore, the relation \(N(\Gt{2}) = f(\Gs{2})\) also implies
  \(f(\Gt{2}) = N(\Gs{2})\).  We rewrite this as
  \((p-f)(\Gt{2}) = (p-N)(\Gs{2})\) in order to get another
  \(\Z\)\nb-relation among the elements \((1- \Gs{2})^j\) for
  \(j=1,\dotsc,p-1\) and \((1- \Gt{2})^j\) for \(j=0,\dotsc,p-1\).
  Our computation shows that~\(\KK^G_0(D,D)\) is a free Abelian
  group of rank \(2(p-1)\).  So any two \(\Z\)\nb-linear relations
  among the elements above must be multiples of each other.  Hence
  there is \(c\in\Z\) with \(p-f = c\cdot N\) and
  \(p-N = c\cdot f\).  So \(f = p - c\cdot N\) and
  \(p-N = c\cdot p - c^2\cdot N\).  Plugging in a primitive
  \(p\)\nb-th root of unity, this implies \(p = c\cdot p\).  So
  \(f = p -N\) as desired.
\end{proof}

Our computations so far show that the ring~\(\Kring\)
is generated by the elements \(\Gu{j}\)
and~\(\Ga{j}{k}\)
and that all relations in Theorem~\ref{the:generators_and_relations}
hold in it.  It remains to show that these relations already imply all
relations in~\(\Kring\).
To begin with, any element of~\(\Kring\)
is a linear combination of \(\Gu{j}\)
and of products in the generators~\(\Ga{j}{k}\).
Here we may restrict attention to products of the form
\((\Ga{j}{k}\otimes \Ga{k}{j})^\ell\)
or \((\Ga{j}{k}\otimes \Ga{k}{j})^\ell\otimes \Ga{j}{k}\)
for \(j\neq k\)
because all other products vanish by~\eqref{eq:alpha_exact}.  Now we
have to go through the nine matrix entries and check that the
relations in Proposition~\ref{pro:further_relations}, which follow
from those in Theorem~\ref{the:generators_and_relations}, allow to
reduce any such expression to terms in a \(\Z\)\nb-module
basis of that entry.  This is easier for the \(j,k\)\nb-th
entries with \(j\neq k\)
because only terms of the form
\((\Ga{j}{k}\otimes \Ga{k}{j})^\ell\otimes \Ga{j}{k}\)
for \(\ell\in\N\)
may occur.  And \((\Ga{j}{k}\otimes \Ga{k}{j})^\ell\)
is a polynomial in \(\Gt{j}\)
or \(\Gs{j}\), respectively.  For the \(0,1\)-entry, we get
\[
(\Ga{0}{1}\otimes \Ga{1}{0})\otimes \Ga{0}{1}
= N(\Gt{0}) \otimes \Ga{0}{1}= p\cdot \Ga{0}{1}
\]
because of~\eqref{eq:ta_01}.  This implies by induction that
\((\Ga{0}{1}\otimes \Ga{1}{0})^\ell\otimes \Ga{0}{1}= p^\ell
\Ga{0}{1}\).
So the \(01\)-component
in the universal ring given by the relations in
Theorem~\ref{the:generators_and_relations} is at most~\(\Z\),
which is the corresponding component in~\(\Kring\).
A similar computation works for the \(10\)-component.
The \(02\)-component in the universal ring is spanned by
\[
(\Gu{0}- \Ga{0}{2}\otimes \Ga{2}{0})^\ell \otimes \Ga{0}{2}
= \Gt{0}^\ell \otimes \Ga{0}{2},
\]
and \(N(\Gt{0}) \otimes \Ga{0}{2}= 0\)
implies that it has rank at most~\(p-1\).
This is also the rank of the corresponding component in~\(\Kring\).
Similar computations work for the \(12\)-,
\(20\)-
and \(21\)-components,
which all have rank~\(p-1\)
and are identified with \(\Z[t]/(N(t))\).

Now consider the \(00\)-component.
Here \(\Ga{0}{1}\otimes \Ga{1}{0}= N(\Gt{0})\)
and \(\Ga{0}{2}\otimes \Ga{2}{0}= \Gu{0}- \Gt{0}\)
are polynomials in~\(\Gt{0}\),
and the relation \(\Gt{0}^p = \Gu{0}\)
in~\eqref{eq:tp} implies that the \(\Z\)\nb-rank
is at most~\(p\)
as desired.  A similar computation works for the \(11\)-component.
Finally, consider the \(22\)-component.
This is spanned by polynomials in \(\Gs{2}\)
and~\(\Gt{2}\);
\eqref{eq:st_2} and~\eqref{eq:ts_2} allow us to get rid of products of
these types of generators.  The relations \(\Gt{2}^p = \Gu{2}\)
and \(\Gs{2}^p = \Gu{2}\)
from \eqref{eq:tp} and~\eqref{eq:sp} reduce to polynomials in
\(\Gs{2}\)
and~\(\Gt{2}\)
of degree at most \(p-1\).
The constants in these two types of polynomials are the same, and
\eqref{eq:alpha_2} makes one more generator redundant, say,
\(\Gs{2}^{p-1}\).
So we remain with a generating set with at most \(2p-2\)
elements.  This is just enough to span the ring \(\KK^G_0(D,D)\)
computed above.  So the canonical map from the universal ring defined
by the generators and relations in
Theorem~\ref{the:generators_and_relations} to~\(\Kring\)
is an isomorphism.  This finishes the proof of
Theorem~\ref{the:generators_and_relations}.

\begin{remark}
  \label{rem:BS_automorphism}
  We have used Baaj--Skandalis duality several times to simplify the
  description of the ring~\(\Kring\).  It is visible from the
  generators and relations in
  Theorem~\ref{the:generators_and_relations} that~\(\Kring\) has an
  involutive automorphism~\(\BS\).  Namely, it is defined on
  generators by \(\BS(\Gu{j}) = \Gu{\sigma(j)}\) and
  \(\BS(\Ga{j}{k}) = \Ga{\sigma(j)}{\sigma(k)}\) for the transposition
  \(\sigma=(01)\) acting on \(j,k\in\{0,1,2\}\), \(j\neq k\).  This
  preserves the relations in
  Theorem~\ref{the:generators_and_relations}.  Hence it defines an
  automorphism~\(\BS\) of~\(\Kring\).  Since \(\BS\circ \BS\) is the
  identity on generators, it is the identity on the whole
  ring~\(\Kring\).  We warn the reader once again that~\(\BS\) does
  not preserve the \(\Z/2\)\nb-grading on~\(\Kring\).
\end{remark}

We also note a canonical anti-automorphism of~\(\Kring\):

\begin{proposition}
  \label{pro:anti-auto_duality}
  There is a unique ring anti-automorphism \(x\mapsto x^*\)
  of~\(\Kring\) with \(\Gu{j}^* = \Gu{j}\) and
  \(\Ga{j}{k}^* = \Ga{k}{j}\) for all \(j,k\in\{0,1,2\}\) with
  \(j \neq k\).  It is grading-preserving and involutive.
\end{proposition}

\begin{proof}
  The formulas for~\(x^*\) on generators imply that the idempotent
  generators~\(\Gu{j}\) and the products
  \(\Ga{j}{k}\otimes \Ga{k}{j}\) are self-adjoint.  Hence
  \(x\mapsto x^*\) leaves most of the relations in
  Theorem~\ref{the:generators_and_relations} unchanged, and it
  replaces the relations \eqref{eq:alpha_vs_1}
  and~\eqref{eq:alpha_exact} for \(j,k\) and~\(j,k,m\) by the same
  relations for \(k,j\) and \(m,k,j\), respectively.  So there is
  indeed a unique anti-automorphism of~\(\Kring\) with the given
  values on the generators.  Since it preserves the grading on
  generators and satisfies \((x^*)^* = x\) for all generators, it is
  grading-preserving and involutive.
\end{proof}

This anti-automorphism should be explicable by Poincaré duality
in~\(\KK^G\) (see \cites{Kasparov:Novikov,
  Emerson-Meyer:Dualities}).  It can be shown that the three
generators \(\C\), \(\Cont(G)\) and~\(D\) are Poincaré self-dual to
themselves.  This duality induces an anti-automorphism on the
ring~\(\Kring\).  I guess that this automorphism is the one
described in Proposition~\ref{pro:anti-auto_duality}.  The
computations needed to verify this seem excessive, however, and
therefore I decided not to discuss this any further in this article.

\subsection{The prime~2}
\label{sec:prime_2}

Let \(p=2\).
Then the ring~\(\Kring\)
is much simpler:
the group ring of~\(\Z/2\)
is quite small, namely, \(\Sring = \Z[t]/(t^2-1)\),
and \(N(t) = 1+t\)
and \(1-t\)
differ only by a sign.  The relations \eqref{eq:alpha_010},
\eqref{eq:alpha_101} and~\eqref{eq:alpha_2} are equivalent to
\begin{align}
  \label{eq:alpha_010_prime2}
  \Ga{0}{1}\otimes \Ga{1}{0}+ \Ga{0}{2}\otimes \Ga{2}{0}&= 2\cdot \Gu{0},\\
  \label{eq:alpha_101_prime2}
  \Ga{1}{0}\otimes \Ga{0}{1}+ \Ga{1}{2}\otimes \Ga{2}{1}&= 2\cdot \Gu{1},\\
  \label{eq:alpha_2_prime2}
  \Ga{2}{0}\otimes \Ga{0}{2}+ \Ga{2}{1}\otimes \Ga{1}{2}&= 2\cdot \Gu{2}.
\end{align}
In brief, we get the relation
\(\Ga{i}{j}\otimes \Ga{j}{i}+ \Ga{i}{k}\otimes \Ga{k}{i}=
2\cdot \Gu{i}\)
whenever \(\{i,j,k\} = \{0,1,2\}\).
These relations are symmetric under
\(\Ga{j}{k}\mapsto \Ga{\sigma(j)}{\sigma(k)}\)
for any permutation~\(\sigma\)
of \(\{0,1,2\}\).
The other relations in Theorem~\ref{the:generators_and_relations} are
also symmetric under index permutations.  So the ring~\(\Kring\)
carries an action of the symmetric group on three letters if \(p=2\).
In particular, the component \(\KK^G_*(D,D)\)
in the ring does not behave differently from the other components
\(\KK^G_*(\C,\C)\)
and \(\KK^G_*(\Cont(G),\Cont(G))\).
This is a special feature of the prime~\(2\).
In general, the \(\Z\)\nb-ranks
of these rings are \(2p-2\),
\(p\) and~\(p\), respectively, and these are only equal for \(p=2\).

The \(\Cst\)\nb-algebra~\(D\)
for \(p=2\)
in~\eqref{eq:cone_D_explicit} is isomorphic to \(\Cont_0(\R)\)
with the action of~\(\Z/2\)
by reflection at the origin.  The group \(\KK^{\Z/2}_*(D,A)\)
for a \(\Z/2\)\nb-\(\Cst\)-algebra
is canonically isomorphic to the \(\K\)\nb-theory
of~\(A\)
as a \(\Z/2\)\nb-graded
\(\Cst\)\nb-algebra
(see also \cites{Haag:Graded, Haag:Algebraic, Meyer:Equivariant}).
Hence this piece in Köhler's invariant is also a ``standard''
\(\K\)\nb-theory group in this case.

\begin{lemma}
  \label{lem:prime2_tensor_D}
  Let \(p=2\).
  Then the functor \(\KK^G\to\KK^G\),
  \(A\mapsto A\otimes D\),
  is an involutive automorphism up to natural equivalence, that is,
  \(A\otimes D\otimes D\)
  is naturally \(\KK^G\)-equivalent
  to~\(A\)
  for all~\(A\) in~\(\KK^G\).
  And \(\C\otimes D \sim D\),
  \(D\otimes D \sim \C\),
  \(\Cont(G) \otimes D \sim \Sigma \Cont(G)\).
  So tensoring with~\(D\)
  induces an automorphism of~\(\Kring\).
  It corresponds to the automorphism that acts on generators by the
  permutation \((02)\).
\end{lemma}

\begin{proof}
  The tensor product \(D\otimes D\) corresponds to \(\Cont_0(\R^2)\)
  with the action by point reflection at the origin.  This is
  orientation-preserving, even \(\K\)\nb-oriented -- unlike the
  action on~\(\R\).  Hence equivariant Bott periodicity gives a
  \(\KK^G\)-equivalence \(D \otimes D \sim \C\).  Thus tensoring
  with~\(D\) is an involutive automorphism of \(\KK^G\).  The tensor
  product \(\Cont(G) \otimes D\) is isomorphic to the suspension
  of~\(\Cont(G)\) because the diagonal action on
  \(\Cont(G) \otimes D\) is conjugate to the action only on the
  factor \(\Cont(G)\).
\end{proof}

\subsection{Exact modules versus group actions on C*-algebras}
\label{sec:exact_to_Cstar}

The following theorem combines Köhler's results with
Theorem~\ref{the:make_actions_simple}.

\begin{theorem}
  \label{the:exact_vs_Cstar}
  There is a bijection between isomorphism
  classes of exact, countable \(\Kring\)\nb-modules
  and \(\KK^G\)-equivalence
  classes of pointwise outer actions of~\(G\)
  on stable, UCT Kirchberg algebras in the equivariant bootstrap
  class~\(\mathfrak{B}^G\).
\end{theorem}

\begin{proof}
  By Theorem~\ref{the:range_F}, a \(\Kring\)\nb-module~\(M\)
  is of the form \(F_*(A)\)
  for an object~\(A\)
  of the equivariant bootstrap class~\(\mathfrak{B}^G\)
  if and only if it is exact and countable.  Then the underlying
  \(\Cst\)\nb-algebra
  of~\(A\)
  may be chosen to be of Type~I and hence nuclear, and it satisfies
  the UCT.  Theorem~\ref{the:make_actions_simple} shows that~\(A\)
  may be chosen to be a stable UCT Kirchberg algebra with a
  pointwise outer action of~\(G\).
  Two actions of~\(G\)
  in~\(\mathfrak{B}^G\)
  that lift the same exact, countable \(\Kring\)\nb-module
  are \(\KK^G\)-equivalent by
  Theorem~\ref{the:Koehler_invariant_UCT_classifies}.
\end{proof}

For a pointwise outer action on a (purely infinite and) simple
\(\Cst\)\nb-algebra~\(A\), the crossed product \(A\rtimes G\) is
again (purely infinite and) simple (see
\cites{Kishimoto:Outer_crossed,
  Jeong-Kodaka-Osaka:Purely_infinite_crossed}).  Hence~\(A\) viewed
as a Hilbert bimodule over the fixed point algebra~\(A^G\)
and~\(A\rtimes G\) is full, so that \(A\rtimes G\) is
Morita--Rieffel equivalent to the fixed point algebra~\(A^G\).  We
may, therefore, identify \(\K_*(A\rtimes G) \cong \K_*(A^G)\).  Then
the map
\((\Ga{1}{0})^*\colon \K_*(A)^G \cong \K_*(A\rtimes G) \to \K_*(A)\)
becomes the map induced by the inclusion \(A^G \hookrightarrow A\);
this follows from Lemma~\ref{lem:natural_trafo_alpha_10_01}.  So the
pieces \(M_0\) and~\(M_1\) and the maps \((\Ga{0}{1})^*\)
and~\((\Ga{1}{0})^*\) between them are standard ingredients in the
study of \(\Z/p\)\nb-actions on simple \(\Cst\)\nb-algebras.  The
exactness of the two sequences in
Definition~\ref{def:exact_R-module} restricts~\(M_2\) considerably.
In particular, if both \(M_0\) and~\(M_1\) vanish, then \(M=0\).
However, \(M_0\) and~\(M_1\) with the maps \(\Gt{0}^M\),
\(\Gs{1}^M\), \(\Ga{0}{1}^M\) and~\(\Ga{1}{0}^M\) do not yet
determine~\(M\).  A concrete counterexample in the case \(p\neq 2\)
is Example~\ref{exa:actions_on_Cuntz_4}.  This example describes
exact \(\Kring\)\nb-modules that depend on a parameter
\(\tau\in \{1,\dotsc,p-1\}\), which enters only in the
map~\(\Ga{2}{1}^M\).  The maps \(\Gt{0}^M\), \(\Gs{1}^M\),
\(\Ga{0}{1}^M\) and~\(\Ga{1}{0}^M\) do not depend on~\(\tau\).
Nevertheless, these examples for different choices of~\(\tau\) are
not isomorphic as \(\Kring\)\nb-modules, and there are at least two
choices for~\(\tau\) if \(p\neq2\).

\subsection{Some consequences of exactness}
\label{sec:exactness_consequences}

This section collects some general properties of exact
\(\Kring\)\nb-modules.
Our starting point is the following elementary lemma, which we are
going to apply to the composites~\(\Ga{j}{k}\otimes \Ga{k}{j}\):

\begin{lemma}
  \label{lem:ker_coker_composite}
  Let \(X\), \(Y\) and~\(Z\) be Abelian groups.  Let \(f\colon X\to Y\)
  and \(g\colon Y\to Z\)
  be composable homomorphisms and \(h=g\circ f\).
  Then \(\ker f\subseteq \ker h\)
  and \(\im h\subseteq \im g\), and
  \[
  \frac{\ker h}{\ker f} \cong \im f \cap \ker g,\qquad
  \frac{\im g}{\im h} \cong \frac{Y}{\im f + \ker g}
  \]
  by applying \(f\) or~\(g^{-1}\), respectively.
  So there is a short exact sequence
  \[
  Y/(\im f+\ker g) \into \coker h \prto \coker g.
  \]
\end{lemma}

\begin{proof}
  Let \(x\in X\).
  Then \(h(x) = 0\)
  if and only if \(f(x)\in \ker g\),
  that is, \(x\in f^{-1}(\ker g)\).
  Since~\(f\)
  maps \(f^{-1}(\ker g)\)
  onto \(\im f \cap \ker g\)
  and \(f^{-1}(\ker g) \supseteq \ker f\),
  the map~\(f\)
  induces an isomorphism
  \(\ker h\bigm/ \ker f \cong \im f \cap \ker g\).
  Now let \(z\in Z\).
  If \(z \in \im h\),
  then there is \(x\in X\)
  with \(g\bigl(f(x)\bigr) = z\).
  So \(z\in \im g\).
  The map~\(g\)
  induces a bijection from \(Y/\ker g\)
  onto \(\im g\).
  The inverse of this bijection maps \(\im h\subseteq \im g\)
  onto the set of all elements of the form \([f(x)] \in Y/\ker g\)
  for \(x\in X\).
  This is the quotient \((\im f+\ker g)/\ker g\).
  This implies the claim about \(\im g/\im h\).
  And then the claim about cokernels also follows.
\end{proof}

When we combine \eqref{eq:alpha_010_N}--\eqref{eq:alpha_212} with
Lemma~\ref{lem:ker_coker_composite}, we get twelve short exact
sequences involving kernels and cokernels of the
maps~\(\Ga{j}{k}^M\).
And these are grouped into six pairs because exactness implies that
some of the third entries in the short exact sequences are equal.
These pairs of exact sequences are collected in
Table~\ref{tab:short_exact_sequences}.  The next lemma condenses these
into isomorphisms between certain quotients of kernels or images.

\begin{table}[htbp]
  \centering
  \begin{tikzcd}
    \ker (\Ga{0}{2}^M) \arrow[r, >->, "\subseteq"] &
    \ker (\Gu{2}^M - \Gt{2}^M) \arrow[r, ->>, "\Ga{0}{2}^M"] &
    \im (\Ga{0}{2}^M) \cap \ker (\Ga{2}{0}^M) \arrow[d, equal] \\
    \ker (\Ga{0}{1}^M) \arrow[r, >->, "\subseteq"] &
    \ker N(\Gs{1}^M) \arrow[r, ->>, "\Ga{0}{1}^M"] &
    \im (\Ga{0}{1}^M) \cap \ker (\Ga{1}{0}^M)
  \end{tikzcd}

  \bigskip
  \begin{tikzcd}
    \ker (\Ga{2}{0}^M) \arrow[r, >->, "\subseteq"] &
    \ker (\Gu{0}^M - \Gt{0}^M) \arrow[r, ->>, "\Ga{2}{0}^M"] &
    \im (\Ga{2}{0}^M) \cap \ker (\Ga{0}{2}^M) \arrow[d, equal] \\
    \ker (\Ga{2}{1}^M) \arrow[r, >->, "\subseteq"] &
    \ker (\Gu{1}^M - \Gs{1}^M) \arrow[r, ->>, "\Ga{2}{1}^M"] &
    \im (\Ga{2}{1}^M) \cap \ker (\Ga{1}{2}^M)
  \end{tikzcd}

  \bigskip
  \begin{tikzcd}
    \ker (\Ga{1}{0}^M) \arrow[r, >->, "\subseteq"] &
    \ker N(\Gt{0}^M) \arrow[r, ->>, "\Ga{1}{0}^M"] &
    \im (\Ga{1}{0}^M) \cap \ker (\Ga{0}{1}^M) \arrow[d, equal] \\
    \ker (\Ga{1}{2}^M) \arrow[r, >->, "\subseteq"] &
    \ker (\Gu{2}^M - \Gs{2}^M) \arrow[r, ->>, "\Ga{1}{2}^M"] &
    \im (\Ga{1}{2}^M) \cap \ker (\Ga{2}{1}^M)
  \end{tikzcd}

  \bigskip
  \begin{tikzcd}
    M_0\bigm/ \bigl(\im (\Ga{0}{2}^M) + \ker (\Ga{2}{0}^M)\bigr)
    \arrow[d, equal] \arrow[r, >->, "\Gd{2}{0}^M"] &
    \coker (\Gu{2}^M - \Gt{2}^M) \arrow[r, ->>, "\textrm{can.}"] &
    \coker (\Ga{2}{0}^M) \\
    M_0\bigm/ \bigl(\im (\Ga{0}{1}^M) + \ker (\Ga{1}{0}^M)\bigr)
    \arrow[r, >->, "\Gd{1}{0}^M"] &
    \coker N(\Gs{1}^M) \arrow[r, ->>, "\textrm{can.}"] &
    \coker (\Ga{1}{0}^M)
  \end{tikzcd}

  \bigskip
  \begin{tikzcd}
    M_2\bigm/ \bigl(\im (\Ga{2}{0}^M) + \ker (\Ga{0}{2}^M)\bigr)
    \arrow[d, equal] \arrow[r, >->, "\Gd{0}{2}^M"] &
    \coker (\Gu{0}^M - \Gt{0}^M) \arrow[r, ->>, "\textrm{can.}"] &
    \coker (\Ga{0}{2}^M) \\
    M_2\bigm/ \bigl(\im (\Ga{2}{1}^M) + \ker (\Ga{1}{2}^M)\bigr)
    \arrow[r, >->, "\Gd{1}{2}^M"] &
    \coker (\Gu{1}^M - \Gs{1}^M) \arrow[r, ->>, "\textrm{can.}"] &
    \coker (\Ga{1}{2}^M)
  \end{tikzcd}

  \bigskip
  \begin{tikzcd}
    M_1\bigm/ \bigl(\im (\Ga{1}{0}^M) + \ker (\Ga{0}{1}^M)\bigr)
    \arrow[d, equal] \arrow[r, >->, "\Gd{0}{1}^M"] &
    \coker N(\Gt{0}^M) \arrow[r, ->>, "\textrm{can.}"] &
    \coker (\Ga{0}{1}^M) \\
    M_1\bigm/ \bigl(\im (\Ga{1}{2}^M) + \ker (\Ga{2}{1}^M)\bigr)
    \arrow[r, >->, "\Gd{2}{1}^M"] &
    \coker (\Gu{2}^M - \Gs{2}^M) \arrow[r, ->>, "\textrm{can.}"] &
    \coker (\Ga{2}{1}^M).
  \end{tikzcd}

  \bigskip
  \caption{Short exact sequences in an exact \(\Kring\)\nb-module.
    The maps denoted~``\(\subseteq\)''
    are inclusion maps, those denoted ``\(\mathrm{can.}\)''
    are the quotient maps, and those denoted~``\(\Gd{j}{k}^M\)''
    or~``\(\Ga{j}{k}^M\)''
    are induced by~\(\Ga{j}{k}^M\)
    between suitable quotients or subgroups, respectively.  Beware
    that some of the maps in the exact sequences above are
    grading-reversing.}
  \label{tab:short_exact_sequences}
\end{table}

\begin{lemma}
  \label{lem:short_exact_sequences_consequences}
  Let~\(M\)
  be an exact \(\Kring\)\nb-module.
  There are isomorphisms
  \begin{alignat*}{2}
    \frac{\ker (\Gu{2}^M - \Gt{2}^M)}{\ker (\Ga{0}{2}^M)}
    &\cong \frac{\ker N(\Gs{1}^M)}{\ker (\Ga{0}{1}^M)},&\qquad
    \frac{\im (\Ga{2}{0}^M)}{\im (\Gu{2}^M - \Gt{2}^M)}
    &\cong \frac{\im (\Ga{1}{0}^M)}{\im N(\Gs{1}^M)},\\
    \frac{\ker (\Gu{0}^M - \Gt{0}^M)}{\ker (\Ga{2}{0}^M)}
    &\cong \frac{\ker (\Gu{1}^M - \Gs{1}^M)}{\ker (\Ga{2}{1}^M)},&\qquad
    \frac{\im (\Ga{0}{2}^M)}{\im (\Gu{0}^M - \Gt{0}^M)}
    &\cong \frac{\im (\Ga{1}{2}^M)}{\im (\Gu{1}^M - \Gs{1}^M)},\\
    \frac{\ker N(\Gt{0}^M)}{\ker (\Ga{1}{0}^M)}
    &\cong \frac{\ker (\Gu{2}^M - \Gs{2}^M)}{\ker (\Ga{1}{2}^M)},&\qquad
    \frac{\im (\Ga{0}{1}^M)}{\im N(\Gt{0}^M)}
    &\cong \frac{\im (\Ga{2}{1}^M)}{\im (\Gu{2}^M - \Gs{2}^M)}.
  \end{alignat*}
  For each quotient~\(X/Y\)
  above, it is also asserted implicitly that \(Y\subseteq X\).
\end{lemma}

\section{Integral representations of groups of prime order}
\label{sec:elementary_S-modules}

The group ring of \(G=\Z/p\) is isomorphic to
\(\Sring\defeq \Z[t]/(t^p-1)\).  This section collects some facts
about \(\Sring\)\nb-modules, which are certainly known to the
experts.  Recall that \(N(t) \defeq 1 + t + \dotsb + t^{p-1}\) and
that \((1-t) \cdot N(t) = 1-t^p \equiv 0\) in~\(\Sring\).  The ring
\(\Z[t]/(N(t))\) is isomorphic to \(\Z[\vartheta]\),
where~\(\vartheta\) denotes a primitive \(p\)-th root of unity.
Let~\(M\) be an \(\Sring\)\nb-module.  As in
Definition~\ref{def:module_notation}, we define \(t^M\colon M\to M\)
by \(t^M(x) \defeq x\cdot t\) for all \(x\in M\).

\begin{lemma}
  \label{lem:S-module_inclusions}
  \(N(t^M) x = p\cdot x\) for all \(x\in \ker(1-t^M)\).  And
  \begin{gather*}
    p\cdot \ker (1-t^M) \subseteq \im N(t^M) \subseteq \ker (1-t^M),\\
    p\cdot \ker N(t^M) \subseteq \im (1-t^M) \subseteq \ker N(t^M),\\
    p\cdot (\ker N(t^M) \cap \ker (1-t^M)) = 0,\\
    p\cdot M \subseteq \im N(t^M) + \im (1-t^M).
  \end{gather*}
\end{lemma}

\begin{proof}
  If \(x=t^M x\), then \(x = (t^M)^j x\) for all \(j\in\N\) and
  hence \(N(t^M)x = p\cdot x\).  The inclusions
  \(\im N(t^M) \subseteq \ker (1-t^M)\) and
  \(\im (1-t^M) \subseteq \ker N(t^M)\) follow immediately from
  \((1-t^M) N(t^M) = N(t^M) (1-t^M) = (t^M)^p-1= 0\).  Since
  \(N(1) = p\), there is a polynomial~\(f\) with
  \(f\cdot (1-t) = N(t)-p\).  Therefore, if \(x\in \ker N(t^M)\),
  then
  \[
    p\cdot x = (p-N(t^M))\cdot x = (1-t^M)f(t^M)(x) \in \im
    (1-t^M).
  \]
  That is, \(p\cdot (\ker N(t^M) \cap \ker (1-t^M)) = 0\).  If
  \(x\in \ker N(t^M) \cap \ker (1-t^M)\), then this computation
  shows \(p\cdot x=0\).  If \(y\in M\), then
  \[
    p\cdot y
    = N(t^M)y - (1-t^M)\bigl(f(t^M)y\bigr)
    \in \im N(t^M) + \im (1-t^M).\qedhere
  \]
\end{proof}

\begin{lemma}
  \label{lem:unit_in_Ztheta}
  There is a unit~\(u\) in the ring~\(\Z[\vartheta]\) with
  \(p = u\cdot (1-\vartheta)^{p-1}\).
\end{lemma}

\begin{proof}
  Let \(j\in\{2,\dotsc,p-1\}\).
  The fraction
  \(u_j \defeq \frac{1-\vartheta^j}{1-\vartheta} = \sum_{k=0}^{j-1}
  \vartheta^k\)
  lies in~\(\Z[\vartheta]\).
  The Galois group of the extension
  \(\Q\subseteq \Q(\vartheta)\)
  acts simply transitively on the primitive \(p\)th
  roots of unity~\(\vartheta^j\),
  \(j=1,\dotsc,p-1\).
  Since~\(N\)
  is the minimal polynomial of~\(\vartheta\),
  \(N(t) = \prod_{j=1}^{p-1} (t-\vartheta^j)\).  So
  \begin{equation}
    \label{eq:p_vs_theta}
    p = N(1) = \prod_{j=1}^{p-1} (1-\vartheta^j).
  \end{equation}
  Since \(1-\vartheta^j\)
  and \(1-\vartheta\)
  are Galois conjugates, they have the same norm.  Hence~\(u_j\)
  is a unit in~\(\Z[\vartheta]\).  Now~\eqref{eq:p_vs_theta} implies
  \begin{equation}
    \label{eq:p_vs_theta_order}
    p = \prod_{j=1}^{p-1} (1-\vartheta^j)
    = \prod_{j=1}^{p-1} \bigl(u_j\cdot (1-\vartheta)\bigr)
    = (1-\vartheta)^{p-1} \prod_{j=1}^{p-1} u_j.
  \end{equation}
  So the unit \(u \defeq \prod_{j=1}^{p-1} u_j\) does the job.
\end{proof}

\begin{definition}
  \label{def:divisible}
  An Abelian group~\(M\)
  (or a module over some ring) is called \emph{uniquely
    \(p\)\nb-divisible}
  if multiplication by~\(p\) on~\(M\) is invertible.
\end{definition}

\begin{example}
  \label{exa:finite_group_divisible}
  A finite group is uniquely \(p\)\nb-divisible
  if and only if~\(p\) does not divide its order.
\end{example}

\begin{definition}
  \label{def:coh_trivial}
  An \(\Sring\)\nb-module~\(M\)
  is called \emph{cohomologically trivial} if
  \(\ker (1-t^M) = \im N(t^M)\) and \(\im (1-t^M) = \ker N(t^M)\).
\end{definition}

\begin{proposition}
  \label{pro:divisible_S-module}
  Let~\(M\) be an \(\Sring\)\nb-module.  The following are equivalent:
  \begin{enumerate}
  \item \label{pro:divisible_S-module_u}%
    \(M\) is uniquely \(p\)\nb-divisible;
  \item \label{pro:divisible_S-module_c}%
    \(M\)
    is cohomologically trivial and
    \(M = \ker N(t^M) \oplus \ker (1-t^M)\);
  \item \label{pro:divisible_S-module_i}%
    \(1-t^M\)
    restricts to an invertible map \(\ker N(t^M) \to \ker N(t^M)\)
    and \(N(t^M)\)
    restricts to an invertible map \(\ker (1-t^M) \to \ker (1-t^M)\).
  \end{enumerate}
\end{proposition}

\begin{proof}
  We prove that~\ref{pro:divisible_S-module_u}
  implies~\ref{pro:divisible_S-module_c}.  So assume~\(M\)
  to be uniquely \(p\)\nb-divisible.
  If \(x\in \ker (1-t^M)\),
  then \(p^{-1}\cdot x \in \ker(1-t^M)\)
  as well because~\(M\)
  is uniquely \(p\)\nb-divisible.
  So \(\ker (1-t^M) = p\cdot \ker (1-t^M)\).
  Then Lemma~\ref{lem:S-module_inclusions} implies
  \(\ker (1-t^M) = \im N(t^M)\).
  Similarly, \(\ker N(t^M) = p\cdot \ker N(t^M)\)
  and hence \(\ker N(t^M) = \im (1-t^M)\).
  Since~\(M\)
  is uniquely \(p\)\nb-divisible,
  the last two equations in Lemma~\ref{lem:S-module_inclusions} imply
  \(\ker N(t^M) \cap \ker (1-t^M) = 0\)
  and \(M = \im N(t^M) + \im (1-t^M)\).
  Now \(M = \ker N(t^M) \oplus \ker (1-t^M)\)
  follows.  We have shown that~\ref{pro:divisible_S-module_u}
  implies~\ref{pro:divisible_S-module_c}.

  Next we prove that~\ref{pro:divisible_S-module_c}
  implies~\ref{pro:divisible_S-module_i}.  So we assume
  \(M = \ker N(t^M) \oplus \ker (1-t^M)\)
  and that~\(M\)
  be cohomologically trivial.  Then
  \(\ker N(t^M) \cap \ker (1-t^M) = 0\).
  So the restrictions of \((1-t^M)\)
  to \(\ker N(t^M)\)
  and of \(N(t^M)\)
  to \(\ker (1-t^M)\)
  are injective.  Since \(1-t^M\)
  vanishes on the complementary summand \(\ker (1-t^M)\),
  it follows that \((1-t^M)(\ker N(t^M)) = \im(1-t^M) = \ker N(t^M)\).
  Similarly, \(N(t^M)(\ker (1-t^M)) = \im N(t^M) = \ker (1-t^M)\).
  So \(1-t^M\)
  restricts to an invertible map \(\ker N(t^M) \to \ker N(t^M)\)
  and \(N(t^M)\)
  restricts to an invertible map \(\ker (1-t^M) \to \ker (1-t^M)\).
  We have shown that~\ref{pro:divisible_S-module_c}
  implies~\ref{pro:divisible_S-module_i}.

  We prove that~\ref{pro:divisible_S-module_i}
  implies~\ref{pro:divisible_S-module_u}.  The restriction
  of~\(N(t^M)\)
  to \(\ker (1-t^M)\)
  is multiplication by~\(p\)
  by Lemma~\ref{lem:S-module_inclusions}.
  Hence~\ref{pro:divisible_S-module_i} implies that \(\ker (1-t^M)\)
  is uniquely \(p\)\nb-divisible.
  Then \(\ker (1-t^M) = \im N(t^M)\)
  by Lemma~\ref{lem:S-module_inclusions}.
  The \(\Sring\)\nb-module
  structure on \(\ker N(t^M)\)
  descends to one over the ring \(\Z[t]/(N(t)) \cong \Z[\vartheta]\).
  Since~\(1-t^M\) is invertible on \(\ker N(t^M)\),
  Lemma~\ref{lem:unit_in_Ztheta} shows that multiplicationby~\(p\)
  is invertible on \(\ker N(t^M)\).  That is,
  \(\ker N(t^M)\) is uniquely \(p\)\nb-divisible.  There is a short
  exact sequence \(\ker N(t^M) \into M \prto \im N(t^M)\).  Since
  the kernel and quotient in it are uniquely \(p\)\nb-divisible, the
  Snake Lemma implies that~\(M\) is uniquely \(p\)\nb-divisible.
\end{proof}

Next we describe \(\Sring\)\nb-modules
where \(\ker (1-t^M)=0\)
or \(\coker (1-t^M)=0\).
These will appear in the classification of exact \(\Kring\)\nb-modules.

\begin{lemma}
  \label{lem:S-module_injective}
  Let~\(M\)
  be an \(\Sring\)\nb-module
  with \(\ker (1-t^M)=0\).
  Then \(N(t^M)=0\)
  and \(\coker (1-t^M)\)
  is a \(\Z/p\)\nb-vector
  space.  Let \((e_i)_{i\in I}\)
  be elements of~\(M\)
  that represent a \(\Z/p\)\nb-basis
  of \(\coker (1-t^M)\).  Then the map
  \[
  \psi\colon \bigoplus_{i\in I} \Z[t]/(N(t)) \to M,\qquad
  (f_i)_{i\in I} \mapsto \sum_{i\in I} f_i(t^M) e_i,
  \]
  is injective, and its cokernel is uniquely \(p\)\nb-divisible.
\end{lemma}

\begin{proof}
  Lemma~\ref{lem:S-module_inclusions} implies \(\im N(t^M)=0\)
  -- that is, \(N(t^M)=0\)
  -- and
  \(p\cdot M \subseteq \im N(t^M) + \im (1-t^M) = \im (1-t^M)\).
  Thus \(\coker (1-t^M) = M/\im (1-t^M)\)
  is a \(\Z/p\)\nb-vector
  space.  And the map~\(\psi\)
  is well defined.  Since \(1-t^M\)
  acts by an injective map on~\(M\),
  so does \((1-t^M)^k\) for all \(k\in\N\).  We claim that the maps
  \[
  (1- t^M)^k\colon
  M/(1-t^M)M \to (1-t^M)^k M \bigm/ (1-t^M)^{k+1} M
  \]
  are isomorphisms for all \(k\in\N\).
  They are well defined surjections by definition.  If \(\xi\in M\)
  is such that \((1-t^M)^k \xi \in(1-t^M)^{k+1} M\),
  then \((1-t^M)^k \xi =(1-t^M)^{k+1} \eta\)
  for some \(\eta\in M\).
  Then \(\xi = (1-t^M)\eta\)
  because \((1-t^M)^k\)
  is injective.  So the map above is indeed bijective.

  Now let~\((f_i)_{i\in I}\)
  belong to the kernel of~\(\psi\).
  That is, \(\sum_{i\in I} f_i(t^M) e_i=0\).
  In particular, \(\sum_{i\in I} f_i(t^M) e_i\in (1-t^M)M\).
  Since~\(e_i\)
  is a basis for \(M/(1-t^M)M\),
  this implies \(f_i(t^M) e_i \in \im (1-t^M)\)
  for all \(i\in I\).
  Then \(1-t \mid f_i\)
  for all \(i\in I\)
  because the map \(\Z[t]/(t-1,N(t)) \to \Z/p\),
  \(f\mapsto [f(1) \bmod p]\),
  is an isomorphism.
  Similarly, an induction over~\(k\)
  shows that \((1-t)^k\mid f_i\)
  for all \(i\in I\),
  using that \((1-t^M)^{k-1} e_i\)
  is a basis for \((1-t^M)^{k-1} M/(1-t^M)^k M\)
  as a \(\Z/p\)\nb-vector
  space.  But \((1-t)^k\mid f_i\)
  for all \(k\in\N\)
  forces \(f_i=0\)
  in the ring \(\Z[t]/(N(t))\).  So the map~\(\psi\) is injective.

  The quotient \(X\defeq \coker \psi\)
  is a module over \(\Z[\vartheta] \cong \Z[t]/(N(t))\).
  The cokernel of multiplication by \(1-t\)
  on~\(\Z[t]/(N(t))\)
  is \(\Z[t]/(1-t,N(t)) \cong \Z/p\).
  Therefore, \(\psi\)
  induces an isomorphism on the cokernels of multiplication
  by~\(1-t\).  Now the Snake Lemma for the morphism of extensions
  \[
  \begin{tikzcd}
    \bigoplus_{i\in I} \Z[\vartheta] \arrow[r, >->, "\psi"] \arrow[d, "1-\vartheta"]  &
    M \arrow[r, ->>] \arrow[d, "1-t^M"]  &
    X \arrow[d, "1-t^X"]  \\
    \bigoplus_{i\in I} \Z[\vartheta] \arrow[r, >->, "\psi"] &
    M \arrow[r, ->>] &
    X
  \end{tikzcd}
  \]
  shows \(\ker (1-t^X)=0\)
  and \(\coker (1-t^X)=0\).
  This is equivalent to~\(X\)
  being uniquely \(p\)\nb-divisible
  by Proposition~\ref{pro:divisible_S-module}.
\end{proof}

\begin{lemma}
  \label{lem:S-module_surjective}
  Let~\(M\)
  be an \(\Sring\)\nb-module
  with \(\im (1-t^M)=M\).
  Then \(N(t^M)=0\)
  and \(\ker (1-t^M)\)
  is a \(\Z/p\)\nb-vector
  space.  Let \((e_i)_{i\in I}\)
  be a basis of \(\ker (1-t^M)\).
  Then there is an injective \(\Sring\)\nb-module map
  \[
  \psi\colon \bigoplus_{i\in I}
  \frac{\Z[\vartheta,1/p]}{(1-\vartheta)\Z[\vartheta]} \to M
  \]
  with uniquely \(p\)\nb-divisible cokernel.
\end{lemma}

\begin{proof}
  Lemma~\ref{lem:S-module_inclusions} implies \(\ker N(t^M)=M\)
  --~that is, \(N(t^M)=0\)~-- and
  \(p\cdot \ker (1-t^M) \subseteq \im N(t^M) = 0\).  That is,
  \(\ker (1-t^M)\) is a \(\Z/p\)\nb-vector space.  By
  Lemma~\ref{lem:unit_in_Ztheta}, making \(p\) or \(1-\vartheta\)
  invertible has the same effect:
  \(\Z[\vartheta,1/p] = \Z[\vartheta,(1-\vartheta)^{-1}]\).  We may
  write any element of this ring as \(f\cdot (1-\vartheta)^{-n}\)
  for some \(f\in \Z[\vartheta]\), \(n\in\N\).

  Since~\(1-t^M\)
  is surjective, we may recursively choose \(e_{i,n}\in M\)
  with \(e_{i,0} = e_i\)
  and \((1-t^M) e_{i,n+1} = e_{i,n}\)
  for all \(n\in\N\).
  There is a well defined map
  \[
  \bigoplus_{i\in I} \Z[\vartheta,(1-\vartheta)^{-1}] \to M,\qquad
  \sum_{i\in I} f_i\cdot (1-\vartheta)^{-n_i} \delta_i\mapsto
  \sum_{i\in I} f_i(t^M) e_{i,n_i},
  \]
  because \((1-t^M) e_{i,n+1} = e_{i,n}\)
  for all \(i\in I\),
  \(n\in\N\).
  This map vanishes on elements of
  \(\bigoplus_{i\in I} (1-\vartheta)\Z[\vartheta]\)
  because \((1-t^M)^{n+1}e_{i,n} = (1-t^M) e_i = 0\).
  We claim that the induced map~\(\psi\) on the quotient
  \[
  \frac{\bigoplus_{i\in I} \Z[\vartheta,(1-\vartheta)^{-1}]}
  {\bigoplus_{i\in I} (1-\vartheta)\Z[\vartheta]}
  \cong  \bigoplus_{i\in I}  \frac{\Z[\vartheta,(1-\vartheta)^{-1}]}
  {(1-\vartheta)\Z[\vartheta]}
  \]
  is injective.  An element of
  \(\bigoplus_{i\in I} \Z[\vartheta,(1-\vartheta)^{-1}]\)
  may be written as a finite linear combination
  \((1-\vartheta)^{-n} \sum_{i\in S} f_i \delta_i\)
  with \(f_i\in \Z[\vartheta]\),
  \(S\subseteq I\)
  finite, and \(n\ge0\).
  Assume that this is in the kernel of~\(\psi\),
  that is, \(\sum_{i\in S} f_i(t^M) e_{i,n} = 0\).
  Applying \((1-t^M)^{n-1}\)
  shows that \(\sum_{i\in S} f_i(t^M) e_i = 0\).
  Since the elements~\(e_i\)
  form a basis for \(\ker (1-t^M)\),
  this implies \(f_i \in \ker (\Z[\vartheta] \to \Z/p)\)
  for all \(i\in S\).
  This kernel is the ideal generated by~\(1-\vartheta\)
  because \(\Z[t]/(t-1,N(t)) = \Z[t]/(t-1,p) \cong \Z/p\).
  So we may rewrite \(f_i = (1-\vartheta) g_i\)
  for some \(g_i\in \Z[\vartheta]\).  Thus
  \[
  (1-\vartheta)^{-n} \sum_{i\in S} f_i \delta_i
  = (1-\vartheta)^{-(n-1)} \sum_{i\in S} g_i \delta_i.
  \]
  We continue like this until \(n=0\).
  Then we reach a term in
  \(\bigoplus_{i\in S} (1-\vartheta)\Z[\vartheta]\).
  This proves that~\(\psi\)
  is injective.  Multiplication with \(1-\vartheta\)
  is surjective on
  \(\bigoplus_{i\in I} \Z[\vartheta,1/p]\bigm/
  (1-\vartheta)\Z[\vartheta]\),
  and its kernel is the \(\Z/p\)\nb-linear
  span of~\(\delta_i\)
  for \(i\in I\).
  Hence~\(\psi\)
  induces an isomorphism between the kernels and cokernels of \(1-t\)
  acting on the domain and codomain of~\(\psi\).
  Now the Snake Lemma implies as in the proof of
  Lemma~\ref{lem:S-module_injective} that \(1-t\)
  acts by an invertible map on \(\coker \psi\).
  Then \(\coker \psi\)
  is uniquely \(p\)\nb-divisible
  by Proposition~\ref{pro:divisible_S-module}.
\end{proof}

\section{The uniquely divisible case}
\label{sec:uniquely_divisible}

We are going to describe uniquely \(p\)\nb-divisible exact
\(\Kring\)\nb-modules in terms of more simple objects.  This
is used in~\cite{dellAmbrogio--Meyer:Spectrum_Kasparov_prime} to
describe the Balmer spectrum of~\(\mathcal{B}^G\).

\begin{example}
  \label{exa:divisible}
  Let~\(X\)
  be a \(\Z/2\)-graded Abelian group and let \(Y\) and~\(Z\)
  be two \(\Z/2\)-graded \(\Z[\vartheta]\)\nb-modules.
  Assume that they are all uniquely \(p\)\nb-divisble.
  Equivalently, they are \(\Z/2\)-graded modules over the rings \(\Z[1/p]\)
  and \(\Z[\vartheta,1/p]\),
  respectively.  We define an exact \(\Kring\)\nb-module:
  \begin{alignat*}{3}
    M_0 &\defeq X\oplus Y,&\qquad
    M_1 &\defeq X\oplus Z,&\qquad
    M_2 &\defeq Y\oplus \Sigma Z,\\
    \Ga{0}{1}^M &= \begin{pmatrix} 1^X&0\\0&0 \end{pmatrix},&\quad
    \Ga{1}{2}^M &= \begin{pmatrix} 0&0\\0&(1-\vartheta)^Z \end{pmatrix},&\quad
    \Ga{2}{0}^M &= \begin{pmatrix} 0&1^Y\\0&0 \end{pmatrix},\\
    \Ga{1}{0}^M &= \begin{pmatrix} p^X&0\\0&0 \end{pmatrix},&\quad
    \Ga{2}{1}^M &= \begin{pmatrix} 0&0\\0&1^Z \end{pmatrix},&\quad
    \Ga{0}{2}^M &= \begin{pmatrix} 0&0\\(1-\vartheta)^Y&0 \end{pmatrix}.
  \end{alignat*}
  Here~\(\Sigma Z\) means~\(Z\) with opposite parity.  So
  \(\Ga{2}{1}^M\) and \(\Ga{1}{2}^M\) are odd.  The relations
  \eqref{eq:1_vs_1}--\eqref{eq:alpha_exact} are easy to check.  And
  we compute
  \[
    \Gt{0}^M = \begin{pmatrix} 1&0\\0&\vartheta \end{pmatrix},\quad
    \Gs{1}^M = \begin{pmatrix} 1&0\\0&\vartheta \end{pmatrix},\quad
    \Gt{2}^M = \begin{pmatrix} \vartheta&0\\0&1 \end{pmatrix},\quad
    \Gs{2}^M = \begin{pmatrix} 1&0\\0&\vartheta \end{pmatrix}.
  \]
  Hence
  \[
    N(\Gt{0}^M) = \begin{pmatrix} p&0\\0&0 \end{pmatrix},\
    N(\Gs{1}^M) = \begin{pmatrix} p&0\\0&0 \end{pmatrix},\
    N(\Gt{2}^M) = \begin{pmatrix} 0&0\\0&p \end{pmatrix},\
    N(\Gs{2}^M) = \begin{pmatrix} p&0\\0&0 \end{pmatrix}.
  \]
  The relations \eqref{eq:alpha_010}--\eqref{eq:alpha_2} follow.  So
  we have defined a \(\Kring\)\nb-module
  by Theorem~\ref{the:generators_and_relations}.
  The two sequences in Definition~\ref{def:exact_R-module} are exact
  because multiplication by~\(p\)
  is invertible on~\(X\)
  and multiplication by~\(1-\vartheta\)
  is invertible on \(Y\) and~\(Z\)
  by Proposition~\ref{pro:divisible_S-module}.
\end{example}

The following theorem implies that this already gives all uniquely
\(p\)\nb-divisible exact \(\Kring\)\nb-modules:

\begin{theorem}
  \label{the:divisible_exact_module}
  Let~\(M\)
  be an exact \(\Kring\)\nb-module.
  Assume that one of the pieces \(M_0\), \(M_1\) or~\(M_2\)
  is uniquely \(p\)\nb-divisible.
  Then the other two pieces are uniquely \(p\)\nb-divisible
  as well.  And
  there are modules \(X\) over \(\Z[1/p]\)
  and \(Y,Z\) over \(\Z[\vartheta,1/p]\)
  such that~\(M\)
  is isomorphic to the exact \(\Kring\)\nb-module
  built out of \(X,Y,Z\) in Example~\textup{\ref{exa:divisible}}.
\end{theorem}

Some ideas used in the proof are also useful for more general
examples.  We develop these systematically in
Section~\ref{sec:general_techniques} and then use them in
Section~\ref{sec:proof_divisible} to prove
Theorem~\ref{the:divisible_exact_module}.
Section~\ref{sec:divisible_Cuntz_algebras} spells out what
Theorem~\ref{the:divisible_exact_module} says about
\(\Z/p\)\nb-actions on Cuntz algebras with uniquely
\(p\)\nb-divisible \(\K\)\nb-theory.  In
Section~\ref{sec:remove_divisible}, we generalise
Theorem~\ref{the:divisible_exact_module} to a statement that extends
a uniquely \(p\)\nb-divisible submodule in~\(M_1\) to a uniquely
\(p\)\nb-divisible submodule in~\(M\).  This will be used to reduce
the study of actions on Cuntz algebras with arbitrary finite
\(\K\)\nb-theory to the special case where the \(\K\)\nb-theory is
purely \(p\)\nb-torsion.

\subsection{Some general proof techniques for exact \texorpdfstring{$\Kring$}{\Kring}-modules}
\label{sec:general_techniques}

\begin{proposition}
  \label{pro:coh_trivial_alpha}
  Let~\(M\)
  be an exact \(\Kring\)\nb-module.
  \begin{itemize}
  \item If the \(\Sring\)\nb-module \((M_0,\Gt{0})\) is
    cohomologically trivial, then
    \begin{align}
      \label{eq:coh_trivial_at_0a}
      \im N(\Gt{0}^M)
      &= \im (\Ga{0}{1}^M)
        = \ker (\Ga{2}{0}^M)
        = \ker (\Gu{0}^M - \Gt{0}^M),\\
      \label{eq:coh_trivial_at_0b}
      \im (\Gu{0}^M - \Gt{0}^M)
      &= \im (\Ga{0}{2}^M)
        = \ker (\Ga{1}{0}^M)
        = \ker N(\Gt{0}^M).
    \end{align}
  \item If the \(\Sring\)\nb-module \((M_1,\Gs{1})\) is
    cohomologically trivial, then
    \begin{align}
      \label{eq:coh_trivial_at_1a}
      \im (\Gu{1}^M - \Gs{1}^M)
      &= \im (\Ga{1}{2}^M)
        = \ker (\Ga{0}{1}^M)
        = \ker N(\Gs{1}^M),\\
      \label{eq:coh_trivial_at_1b}
      \im N(\Gs{1}^M)
      &= \im (\Ga{1}{0}^M)
        = \ker (\Ga{2}{1}^M)
        = \ker (\Gu{1}^M - \Gs{1}^M).
    \end{align}
  \item If
    \begin{equation}
      \label{eq:coh_trivial_at_2_assum}
      \im (\Gu{2}^M - \Gt{2}^M) = \ker (\Gu{2}^M - \Gs{2}^M)
      \quad\text{and}\quad
      \im (\Gu{2}^M - \Gs{2}^M) = \ker (\Gu{2}^M - \Gt{2}^M),
    \end{equation}
    then
    \begin{align}
      \label{eq:coh_trivial_at_2a}
      \im (\Gu{2}^M - \Gt{2}^M)
      &= \im (\Ga{2}{0}^M)
        = \ker (\Ga{1}{2}^M)
        = \ker (\Gu{2}^M - \Gs{2}^M),\\
      \label{eq:coh_trivial_at_2b}
      \im (\Gu{2}^M - \Gs{2}^M)
      &= \im (\Ga{2}{1}^M)
        =  \ker (\Ga{0}{2}^M)
        = \ker (\Gu{2}^M - \Gt{2}^M).
    \end{align}
  \end{itemize}
\end{proposition}

\begin{proof}
  The middle equality in the six equations that we assert are the six
  conditions that say that a \(\Kring\)\nb-module
  is exact.  When we replace the first and third equality in our six
  equations by~``\(\subseteq\)'',
  they become the twelve inclusions that are implicitly asserted in
  Lemma~\ref{lem:short_exact_sequences_consequences}.  In each
  equation, the first and last item are equal by cohomological
  triviality or by~\eqref{eq:coh_trivial_at_2_assum}, respectively.
  This forces the inclusions in between to be equalities.
\end{proof}

\begin{corollary}
  \label{cor:M0_M1_divisible}
  If \(M_0\)
  is uniquely \(p\)\nb-divisible,
  then \eqref{eq:coh_trivial_at_0a} and~\eqref{eq:coh_trivial_at_0b}
  hold.
  If \(M_1\)
  is uniquely \(p\)\nb-divisible,
  then \eqref{eq:coh_trivial_at_1a} and~\eqref{eq:coh_trivial_at_1b}
  hold.
\end{corollary}

\begin{proof}
  The pieces \(M_0\)
  and~\(M_1\)
  in an exact \(\Kring\)\nb-module
  are \(\Sring\)\nb-modules.
  Uniquely \(p\)\nb-divisible
  \(\Sring\)\nb-modules
  are cohomologically trivial by
  Proposition~\ref{pro:divisible_S-module}.  Therefore, the statements
  follow from Proposition~\ref{pro:coh_trivial_alpha}.
\end{proof}

We now carry this argument over to the case where~\(M_2\)
is uniquely \(p\)\nb-divisible.
The group~\(M_2\)
in a \(\Kring\)\nb-module is a module over the commutative ring
\begin{align}
  \label{eq:S2}
  \Sring_2 &\defeq \Z[s,t] \bigm/ \bigl((1-s)\cdot(1-t),N(s)+N(t)-p\bigr)
  \\&= \Z[s,t] \bigm/ \bigl((1-s)\cdot(1-t),N(s)+N(t)-p, t^p-1,s^p-1\bigr). \notag
\end{align}
The relations \(t^p-1=0\) and \(s^p-1=0\) hold in~\(\Sring_2\) because
\begin{align*}
  s^p-1 &= (s-1)(N(s)+N(t)-p) - (s-1)(t-1) \frac{N(t)-p}{t-1},\\
  t^p-1 &= (t-1)(N(s)+N(t)-p) - (t-1)(s-1) \frac{N(s)-p}{s-1}.
\end{align*}

\begin{definition}
  \label{def:cohomologically_trivial_S2}
  An \(\Sring_2\)\nb-module
  is called \emph{cohomologically trivial} if it satisfies the two
  conditions \(\im (1-t^M) = \ker (1-s^M)\)
  and \(\im (1-s^M) = \ker (1-t^M)\)
  in~\eqref{eq:coh_trivial_at_2_assum}.
\end{definition}

\begin{lemma}
  \label{lem:S2_module_im_ker}
  Let~\(M\) be an \(\Sring_2\)\nb-module.  Then
  \begin{align}
    \label{eq:S2_imNs}
    \im N(s^M)
    &\subseteq \im (1-t^M)
      \subseteq \ker (1-s^M)
      \subseteq \ker N(t^M),\\
    \label{eq:S2_imNt}
    \im N(t^M)
    &\subseteq \im (1-s^M)
      \subseteq \ker (1-t^M)
      \subseteq \ker N(s^M),
  \end{align}
\end{lemma}

\begin{proof}
  Let \(x\in M\).
  Then \(N(s^M)(x) = (p-N(t^M))(x)\)
  because \(N(s)+N(t)-p=0\)
  in~\(\Sring_2\).
  Since \(1-t \mid p-N(t)\),
  it follows that \(N(s^M)(x) \in \im (1-t^M)\).
  The relation \((1-s)(1-t)=0\)
  in~\(\Sring_2\)
  implies \(\im (1-t^M) \subseteq \ker(1-s^M)\).
  If \((1-s^M)(x)=0\),
  then \((N(s^M)-p)(x)=0\).
  So the relation \(N(s)+N(t) -p=0\)
  in~\(\Sring_2\)
  implies \(N(t^M)(x)=0\).
  That is, \(\ker (1-s^M) \subseteq \ker N(t^M)\).
  The proof of~\eqref{eq:S2_imNt} is analogous.
\end{proof}

\begin{proposition}
  \label{pro:divisible_S2-module}
  Let~\(M\) be an \(\Sring_2\)\nb-module.  The following are equivalent:
  \begin{enumerate}
  \item \label{pro:divisible_S2-module_u}%
    \(M\) is uniquely \(p\)\nb-divisible;
  \item \label{pro:divisible_S2-module_c}%
    \(M\)
    is cohomologically trivial and
    \(M = \ker (1-s^M) \oplus \ker (1-t^M)\);
  \item \label{pro:divisible_S2-module_i}%
    \(1-s^M\)
    restricts to an invertible map \(\ker (1-t^M) \to \ker (1-t^M)\)
    and \(1-t^M\)
    restricts to an invertible map \(\ker (1-s^M) \to \ker (1-s^M)\).
  \end{enumerate}
\end{proposition}

\begin{proof}
  If~\ref{pro:divisible_S2-module_u} holds, then
  Proposition~\ref{pro:divisible_S-module} implies that~\(M\)
  is cohomologically trivial as a module over the two copies of~\(\Sring\)
  inside~\(\Sring_2\).
  Then \(\im N(s^M) = \im (1-t^M) = \ker (1-s^M)\)
  and \(\im N(t^M) = \im (1-s^M) = \ker (1-t^M)\)
  by Lemma~\ref{lem:S2_module_im_ker}.  So~\(M\)
  is cohomologically trivial as an \(\Sring_2\)\nb-module
  and \(M = \im N(s^M) \oplus \im (1-s^M) = \ker (1-s^M) \oplus \ker (1-t^M)\).
  So~\ref{pro:divisible_S2-module_u}
  implies~\ref{pro:divisible_S2-module_c}.

  Now assume~\ref{pro:divisible_S2-module_c}.  Then
  \(\ker (1-s^M) \cap \ker (1-t^M) = 0\)
  says that \(1-s^M\)
  is injective on \(\ker (1-t^M)\).
  And \(M = \ker (1-s^M) + \ker (1-t^M)\)
  implies \((1-s^M)(M) = (1-s^M)(\ker (1-t^M))\).
  Thus \(1-s^M\colon \ker (1-t^M) \to \im (1-s^M) = \ker (1-t^M)\)
  is also surjective.  The same argument works with~\(t\)
  instead of~\(s\).
  So~\ref{pro:divisible_S2-module_c}
  implies~\ref{pro:divisible_S2-module_i}.

  Now assume~\ref{pro:divisible_S2-module_i}.  The subgroup
  \(\ker (1-s^M)\) is a module over \(\Z[t]/(N(t))\) because both
  \(1-s^M\) and \(N(t^M)\) vanish there by~\eqref{eq:S2_imNs}.
  Lemma~\ref{lem:unit_in_Ztheta} says that \(p\)
  and~\((1-\vartheta)^{p-1}\) differ only by a unit in the ring
  \(\Z[t]/(N(t))\).  Since~\(1-t^M\) is invertible on
  \(\ker (1-s^M)\), \(\ker (1-s^M)\) is uniquely \(p\)\nb-divisible.
  Similarly, \(\ker (1-t^M)\) is uniquely \(p\)\nb-divisible.  Hence
  \(\ker (1-t^M) = \im N(t^M)\) by
  Lemma~\ref{lem:S-module_inclusions}.  Then
  \(\ker (1-t^M) = \im (1-s^M)\) by
  Lemma~\ref{lem:S2_module_im_ker}.  Then the Snake Lemma applied to
  multiplication by~\(p\) in the short exact sequence
  \(\ker (1-s^M) \into M \prto \im (1-s^M)\) shows that~\(M\) is
  uniquely \(p\)\nb-divisible.
  Hence~\ref{pro:divisible_S2-module_i}
  implies~\ref{pro:divisible_S2-module_u}.
\end{proof}

Now the same argument as above gives us an analogue of
Corollary~\ref{cor:M0_M1_divisible} for~\(M_2\):

\begin{corollary}
  \label{cor:M2_divisible}
  If~\(M_2\)
  is uniquely \(p\)\nb-divisible,
  then \eqref{eq:coh_trivial_at_2a} and~\eqref{eq:coh_trivial_at_2b}
  hold.
\end{corollary}

Our next goal are criteria for an exact \(\Kring\)\nb-module
to satisfy \(M_0 \cong \im (\Ga{0}{1}^M) \oplus \im(\Ga{0}{2}^M)\),
and similarly for \(M_1\)
or~\(M_2\).
These criteria help to prove Theorem~\ref{the:divisible_exact_module}
because such decompositions clearly happen in the exact
\(\Kring\)\nb-modules
in Example~\ref{exa:divisible}.

\begin{lemma}
  \label{lem:decomposable_case_criterion_0}
  Let~\(M\) be an exact \(\Kring\)\nb-module.  The following are equivalent:
  \begin{enumerate}
  \item  \label{lem:decomposable_case_criterion1}%
    \(M_0 \cong \im (\Ga{0}{1}^M) \oplus \im(\Ga{0}{2}^M)\);
  \item \label{lem:decomposable_case_criterion2}%
    \(\ker (\Ga{0}{1}^M) = \ker N(\Gs{1}^M)\)
    and \(\im (\Ga{1}{0}^M) = \im N(\Gs{1}^M)\);
  \item \label{lem:decomposable_case_criterion3}%
    \(\ker (\Ga{0}{2}^M) = \ker (\Gu{2}^M - \Gt{2}^M)\)
    and \(\im (\Ga{2}{0}^M) = \im (\Gu{2}^M - \Gt{2}^M)\).
  \end{enumerate}
  These equivalent conditions hold if the \(\Sring\)\nb-module
  \((M_1,\Gs{1})\)
  is cohomologically trivial or if the \(\Sring_2\)\nb-module
  \((M_2,\Gs{2},\Gt{2})\) is cohomologically trivial.
\end{lemma}

\begin{proof}
  The direct sum decomposition
  in~\ref{lem:decomposable_case_criterion1} is equivalent to
  \(\im (\Ga{0}{1}^M) \cap \im(\Ga{0}{2}^M) = 0\)
  and \(M_0\bigm/ (\im (\Ga{0}{1}^M) + \im(\Ga{0}{2}^M)\bigr) = 0\).
  By the exact sequences in Table~\ref{tab:short_exact_sequences}, the
  first condition holds if and only if
  \(\ker (\Ga{0}{2}^M) = \ker (\Gu{2}^M - \Gt{2}^M)\),
  if and only if \(\ker N(\Gs{1}^M) = \ker (\Ga{0}{1}^M)\);
  and the second condition holds if and only if
  \(\im (\Ga{2}{0}^M) = \im (\Gu{2}^M - \Gt{2}^M)\),
  if and only if \(\im (\Ga{1}{0}^M) = \im N(\Gs{1}^M)\).
  Here we use that two maps with the same codomain have equal image if
  and only if they have equal cokernels.

  If~\(M_1\)
  is cohomologically trivial, then \eqref{eq:coh_trivial_at_1a}
  and~\eqref{eq:coh_trivial_at_1b} are valid and imply our
  statement~\ref{lem:decomposable_case_criterion2}.  If~\(M_2\)
  is cohomologically trivial, then \eqref{eq:coh_trivial_at_2a}
  and~\eqref{eq:coh_trivial_at_2b} are valid and
  imply~\ref{lem:decomposable_case_criterion3}.
\end{proof}

Assume that we are in the situation of
Lemma~\ref{lem:decomposable_case_criterion_0}.  The direct sum
decomposition in
Lemma~\ref{lem:decomposable_case_criterion_0}.\ref{lem:decomposable_case_criterion1}
is one of \(\Sring\)\nb-modules
because both \(\im (\Ga{0}{1}^M)\)
and \(\im(\Ga{0}{2}^M)\)
are \(\Sring\)\nb-submodules
by \eqref{eq:ta_01} and~\eqref{eq:ta_02}.  Using
Lemma~\ref{lem:decomposable_case_criterion_0} and exactness, we
identify
\begin{align*}
  \im(\Ga{0}{1}^M)
  &\cong \frac{M_1}{\ker(\Ga{0}{1}^M)}
    = \frac{M_1}{\ker N(\Gs{1}^M)}
    \cong \im N(\Gs{1}^M),\\
  \im(\Ga{0}{2}^M)
  &\cong \frac{M_2}{\ker(\Ga{0}{2}^M)}
    = \frac{M_2}{\ker (\Gu{2}^M - \Gt{2}^M)}
    \cong \im (\Gu{2}^M - \Gt{2}^M).
\end{align*}
Let~\(\varphi\) temporarily denote the first isomorphism.  Then
\(\varphi\circ \Ga{0}{1}^M = N(\Gs{1}^M)\) and
\(\Ga{1}{0}^M \circ \varphi^{-1}= \Ga{1}{0}^M \circ \Ga{0}{1}^M =
N(\Gs{1}^M)\).  Therefore, when we use~\(\varphi\) to identify
\(\im(\Ga{0}{1}^M)\) with \(\im N(\Gs{1}^M)\), then \(\Ga{0}{1}^M\)
becomes \(N(\Gs{1}^M)\) and \(\Ga{1}{0}^M\) becomes the inclusion
map on \(\im N(\Gs{1}^M)\).  And \(\Ga{1}{0}^M\) vanishes on the
other summand \(\im(\Ga{0}{2}^M)\) by~\eqref{eq:alpha_exact}.
Similarly, \(\Ga{0}{2}^M\) becomes \(\Gu{2}^M - \Gt{2}^M\),
and~\(\Ga{2}{0}^M\) becomes the inclusion of the summand
\(\im (\Gu{2}^M - \Gt{2}^M)\) and vanishes on the other summand.  So
the \(\Kring\)\nb-module is determined by \(M_1\) and~\(M_2\) with
the maps \(\Gt{2}^M\), \(\Ga{1}{2}^M\) and~\(\Ga{2}{1}^M\);
these determine~\(\Gs{j}^M\) for \(j=1,2\) by \eqref{eq:alpha_121}
and~\eqref{eq:alpha_212}.  This data gives an exact
\(\Kring\)\nb-module if and only if
\begin{enumerate}
\item \label{en:decomposable_1}%
  \(N(\Gt{2}^M) + N(\Gs{2}^M) = p\cdot \Gu{2}^M\);
\item \label{en:decomposable_2}%
  \(\ker (\Ga{1}{2}^M) = \im (\Gu{2}^M - \Gt{2}^M)\),
  \(\im (\Ga{2}{1}^M) = \ker (\Gu{2}^M - \Gt{2}^M)\);
\item \label{en:decomposable_3}%
  \(\ker (\Ga{2}{1}^M) = \im N(\Gs{1}^M)\),
  \(\im (\Ga{1}{2}^M) = \ker N(\Gs{1}^M)\).
\end{enumerate}
We may arrange the exactness conditions in \ref{en:decomposable_2}
and~\ref{en:decomposable_3} in a periodic long exact sequence
\begin{equation}
  \label{eq:decomposable_case_long_exact}
  \dotsb \to M_1 \xrightarrow{N(\Gs{1}^M)}
  M_1 \xrightarrow{\Ga{2}{1}^M}
  M_2 \xrightarrow{\Gu{2}^M - \Gt{2}^M}
  M_2 \xrightarrow{\Ga{1}{2}^M}
  M_1 \xrightarrow{N(\Gs{1}^M)} M_1 \to \dotsb.
\end{equation}

Lemmas \ref{lem:decomposable_case_criterion_1}
and~\ref{lem:decomposable_case_criterion_2} are the analogues of
Lemma~\ref{lem:decomposable_case_criterion_0} for \(M_1\)
and~\(M_2\).
Lemma~\ref{lem:decomposable_case_criterion_1} is formally equivalent
to Lemma~\ref{lem:decomposable_case_criterion_0} because of the ring
automorphism~\(\BS\)
of~\(\Kring\).
The proof of Lemma~\ref{lem:decomposable_case_criterion_2} is similar
to the proof of Lemma~\ref{lem:decomposable_case_criterion_0}.

\begin{lemma}
  \label{lem:decomposable_case_criterion_1}
  Let~\(M\) be an exact \(\Kring\)\nb-module.  The following are equivalent:
  \begin{enumerate}
  \item  \label{lem:decomposable_case_criterion4}%
    \(M_1 \cong \im (\Ga{1}{0}^M) \oplus \im(\Ga{1}{2}^M)\);
  \item \label{lem:decomposable_case_criterion5}%
    \(\ker (\Ga{1}{0}^M) = \ker N(\Gt{0}^M)\)
    and \(\im (\Ga{0}{1}^M) = \im N(\Gt{0}^M)\);
  \item \label{lem:decomposable_case_criterion6}%
    \(\ker (\Ga{1}{2}^M) = \ker (\Gu{2}^M - \Gs{2}^M)\)
    and \(\im (\Ga{2}{1}^M) = \im (\Gu{2}^M - \Gs{2}^M)\);
  \end{enumerate}
  These equivalent conditions hold if the \(\Sring\)\nb-module
  \((M_0,\Gt{0})\)
  is cohomologically trivial or if the \(\Sring_2\)\nb-module
  \((M_2,\Gs{2},\Gt{2})\) is cohomologically trivial.
\end{lemma}

\begin{lemma}
  \label{lem:decomposable_case_criterion_2}
  Let~\(M\) be an exact \(\Kring\)\nb-module.  The following are equivalent:
  \begin{enumerate}
  \item  \label{lem:decomposable_case_criterion7}%
    \(M_2 \cong \im (\Ga{2}{0}^M) \oplus \im(\Ga{2}{1}^M)\);
  \item \label{lem:decomposable_case_criterion8}%
    \(\ker (\Ga{2}{0}^M) = \ker (\Gu{0}^M - \Gt{0}^M)\)
    and \(\im (\Ga{0}{2}^M) = \im (\Gu{0}^M - \Gt{0}^M)\);
  \item \label{lem:decomposable_case_criterion9}%
    \(\ker (\Ga{2}{1}^M) = \ker (\Gu{1}^M - \Gs{1}^M)\)
    and \(\im (\Ga{1}{2}^M) = \im (\Gu{1}^M - \Gs{1}^M)\).
  \end{enumerate}
  These equivalent conditions hold if one of the \(\Sring\)\nb-modules
  \((M_0,\Gt{0})\) or \((M_1,\Gs{1})\) is cohomologically trivial.
\end{lemma}

\subsection{Proof of Theorem~\ref{the:divisible_exact_module}}
\label{sec:proof_divisible}

Assume first that~\(M_0\)
is uniquely \(p\)\nb-divisible.
Then~\(M_0\)
is cohomologically trivial and
\(M_0 \cong \ker N(\Gt{0}^M) \oplus \ker (\Gu{0}^M - \Gt{0}^M)\)
by Proposition~\ref{pro:divisible_S-module}.  And
\eqref{eq:coh_trivial_at_0a} and~\eqref{eq:coh_trivial_at_0b} hold.
Hence the statements~\ref{lem:decomposable_case_criterion1} in
Lemma~\ref{lem:decomposable_case_criterion_0},
\ref{lem:decomposable_case_criterion5} in
Lemma~\ref{lem:decomposable_case_criterion_1}
and~\ref{lem:decomposable_case_criterion8} in
Lemma~\ref{lem:decomposable_case_criterion_2} are true.  So the other
equivalent statements in these three lemmas are also true.  In
particular,
\begin{align*}
  M_0
  &\cong \im (\Ga{0}{2}^M) \oplus \im(\Ga{0}{1}^M)
  = \ker N(\Gt{0}^M) \oplus \ker (\Gu{0}^M - \Gt{0}^M),\\
  M_1
  &\cong \im (\Ga{1}{0}^M) \oplus \im(\Ga{1}{2}^M)
  = \im N(\Gs{1}^M)  \oplus \im (\Gu{1}^M - \Gs{1}^M),\\
  M_2 &\cong \im (\Ga{2}{0}^M) \oplus \im(\Ga{2}{1}^M)
  \cong \im (\Gu{2}^M - \Gt{2}^M) \oplus \im (\Gu{2}^M - \Gs{2}^M).
\end{align*}
Let \(j,k\in\{0,1,2\}\)
with \(j\neq k\).  The map~\(\Ga{k}{j}^M\)
maps the summand \(\im (\Ga{j}{k}^M)\)
to the summand \(\im (\Ga{k}{j}^M)\)
in~\(M_j\)
and vanishes on the other summand in~\(M_k\).
The exactness of~\(M\) shows that the restriction
\[
\Ga{k}{j}^M\colon \im (\Ga{j}{k}^M) \to \im (\Ga{k}{j}^M)
\]
is invertible.  As a consequence, \(\Ga{k}{j}^M\circ \Ga{j}{k}^M\)
is invertible on the direct summand \(\im (\Ga{k}{j}^M)\).
These statements for \(k=1,2\)
are one of the equivalent characterisations of unique
\(p\)\nb-divisibility
in Propositions \ref{pro:divisible_S-module}
and~\ref{pro:divisible_S2-module}, respectively.  So \(M_0\)
and~\(M_2\)
are also uniquely \(p\)\nb-divisible,
and then so are all direct summands in them.  Now it is routine to
check that our exact \(\Kring\)\nb-module
is isomorphic to the one built in Example~\ref{exa:divisible} with
\(X = \im (\Ga{0}{1}^M) \cong \im (\Ga{1}{0}^M)\),
\(Y = \im (\Ga{0}{2}^M) \cong \im (\Ga{2}{0}^M)\),
\(Z = \im (\Ga{1}{2}^M) \cong \im (\Ga{2}{1}^M)\).
This proves the theorem when~\(M_0\)
is uniquely \(p\)\nb-divisible.
We may reduce the case where~\(M_1\) is assumed uniquely
\(p\)\nb-divisible to this case using the ring automorphism~\(\BS\),
which exchanges the roles of \(M_0\) and~\(M_1\).

Now let~\(M_2\)
be uniquely \(p\)\nb-divisible.
Proposition~\ref{pro:divisible_S2-module} shows that~\(M_2\)
is cohomologically trivial as an \(\Sring_2\)\nb-module
and decomposes as
\(M_2 = \im (\Gu{2}-\Gs{2}^M) \oplus \im (\Gu{2}-\Gt{2}^M)\).
Thus the statements in \eqref{eq:coh_trivial_at_2a}
and~\eqref{eq:coh_trivial_at_2b} hold.  Therefore,
\(M_2 = \im (\Ga{2}{1}^M) \oplus \im (\Ga{2}{0}^M)\)
as well.  Then in each of the Lemmas
\ref{lem:decomposable_case_criterion_0},
\ref{lem:decomposable_case_criterion_1}
and~\ref{lem:decomposable_case_criterion_2}, one of the equivalent
statements is true.  So all the equivalent statements in these lemmas
are true.  And now the argument proceeds as above.  This finishes
the proof of Theorem~\ref{the:divisible_exact_module}.

\subsection{Application to Cuntz algebras}
\label{sec:divisible_Cuntz_algebras}

Theorem~\ref{the:divisible_exact_module} applies, in particular, if
\(M_1=0\)
because the trivial group is uniquely \(p\)\nb-divisible.
So any exact \(\Kring\)\nb-module
with \(M_1=0\)
is isomorphic to one as in Example~\ref{exa:divisible}.  And \(M_1=0\)
forces \(X=0\)
and \(Z=0\).
So the only remaining ingredient is a uniquely \(p\)\nb-divisible
\(\Z/2\)\nb-graded
\(\Z[\vartheta]\)-module~\(Y\),
which may be chosen at will.  The resulting exact \(\Kring\)\nb-module
with \(M_1= 0\)
has \(M_0 = Y\),
\(M_2 = Y\),
\(\Ga{2}{0}^M = 1\)
and \(\Ga{0}{2}^M = 1-\vartheta\).
The following theorem summarises our results in this case:

\begin{theorem}
  \label{the:action_O2}
  Let~\(Y\)
  be a countable, uniquely \(p\)\nb-divisible
  Abelian group with a representation of~\(\Z/p\),
  \(k\mapsto t^k\),
  such that \(N(t)=0\).
  Then there is a pointwise outer action of~\(\Z/p\)
  on~\(\mathcal{O}_2\otimes \Comp(\ell^2\N)\)
  that belongs to the equivariant bootstrap class~\(\mathfrak{B}^G\),
  such that
  \(\K_*(\mathcal{O}_2\otimes\Comp(\ell^2\N) \rtimes \Z/p) \cong Y\)
  and the action on \(K\)\nb-theory
  induced by the dual action becomes the one on~\(Y\)
  generated by~\(t\).
  The \(G\)\nb-action on~\(\mathcal{O}_2\otimes \Comp(\ell^2\N)\)
  above is unique up to \(\KK^G\)-equivalence.
\end{theorem}

\begin{proof}
  Combine Theorem~\ref{the:exact_vs_Cstar} with the
  classification of exact \(\Kring\)\nb-modules
  with \(M_1=0\).
  Since \(F_*(A)_1 = \K_*(A)\),
  the condition \(M_1=0\)
  corresponds to an action on a \(\Cst\)\nb-algebra
  with vanishing \(\K\)\nb-theory such as~\(\mathcal{O}_2\).
  And then \(M_0=\K_*(A\rtimes G)\).
\end{proof}

A representation of~\(\Z/p\)
on a uniquely \(p\)\nb-divisible group with \(N(t)=0\)
may also be considered as a module over the ring
\(\Z[\vartheta,1/p]\).
The action of the group ring \(\Z[\Z/p]\)
descends to an action of \(\Z[\vartheta]\)
if and only if \(N(t)=0\).
And the underlying group is uniquely \(p\)\nb-divisible
if and only if the action of \(\Z[\vartheta]\)
extends to \(\Z[\vartheta,1/p]\).

More generally, Theorem~\ref{the:divisible_exact_module} classifies
the actions in the equivariant bootstrap class on any UCT Kirchberg
algebra~\(A\)
whose \(\K\)\nb-theory
is uniquely \(p\)\nb-divisible.
Namely, \(\KK^G\)-equivalence
classes of such actions are in bijection with isomorphism classes of
countable modules over \(\Z[\vartheta,1/p]\),
exactly as for~\(\mathcal{O}_2\).
The exact \(\Kring\)\nb-module
corresponding to a given countable module~\(Y\)
over \(\Z[\vartheta,1/p]\)
is as in Example~\ref{exa:divisible}, where \(X\)
and~\(Z\)
are obtained by decomposing \(\K_*(A)\)
into \(\ker (1-t) = \im N(t)\)
and \(\im (1-t) = \ker N(t)\),
where \(t\colon \K_*(A)\to\K_*(A)\)
is induced by the generator of the \(\Z/p\)\nb-action on~\(A\).

We consider the classification of \(\Z/p\)-actions
on Cuntz algebras as a test case for our theory.  The result in the
uniquely \(p\)\nb-divisible
case covers \(\Z/p\)\nb-actions
on~\(\mathcal{O}_{n+1}\)
for all \(n\in\N_{\ge1}\)
that are coprime to~\(p\).
In each case, there are as many such actions as there are countable
modules over \(\Z[\vartheta,1/p]\).

An exact \(\Kring\)\nb-module
as in Example~\ref{the:divisible_exact_module} is a direct sum of
three pieces, which are exact \(\Kring\)\nb-modules where two of the
three pieces \(X,Y,Z\)
vanish.  Theorem~\ref{the:divisible_exact_module} says that any
object of \(\mathfrak{B}^G\)
where one of the pieces in \(F_*(A)\)
is uniquely \(p\)\nb-divisible
is such a direct sum.  Hence the rather strong lack of uniqueness of
\(\Z/p\)-actions in the situation above is caused by the many
actions of~\(\Z/p\)
on~\(\mathcal{O}_2\).
Given any object~\(A\) of~\(\mathfrak{B}^G\)
and any action on~\(\mathcal{O}_2\),
we get an action on \(A\oplus\mathcal{O}_2\).
We may then use Theorem~\ref{the:make_actions_simple} to replace
\(A\oplus\mathcal{O}_2\)
by a stable UCT Kirchberg algebra, which is again isomorphic
to~\(A\)
if~\(A\)
is already a stable UCT Kirchberg algebra.

\subsection{Removing a uniquely divisible submodule}
\label{sec:remove_divisible}

We shall need a generalisation of
Theorem~\ref{the:divisible_exact_module} where we weaken the
assumption that~\(M_1\) be uniquely \(p\)\nb-divisible and only
require a given uniquely \(p\)\nb-divisible submodule of~\(M_1\).
Similar statements hold for submodules in the other entries
of~\(M\), but we shall only use this case.

\begin{theorem}
  \label{the:remove_divisible}
  Let~\(M\)
  be an exact \(\Kring\)\nb-module.
  Let \(M_1' \subseteq M_1\)
  be an \(\Gs{1}^M\)\nb-invariant,
  uniquely \(p\)\nb-divisible
  \(\Z/2\)\nb-graded
  subgroup.  Let \(M'_0 \defeq \Ga{0}{1}^M(M'_1)\)
  and \(M'_2 \defeq \Ga{2}{1}^M(M'_1)\).
  Then~\(M'\)
  is an exact \(\Kring\)\nb-submodule
  of~\(M\).
  It is uniquely \(p\)\nb-divisible
  and isomorphic to an exact \(\Kring\)\nb-module
  as in Example~\textup{\ref{exa:divisible}}, with \(X = M'_0\)
  and \(Z = M'_2\).  The quotient~\(M/M'\) is also exact.
\end{theorem}

\begin{proof}
  Each generator~\(\Ga{j}{k}\)
  maps~\(M'\)
  into itself because~\(M'_1\)
  is \(\Gs{1}^M\)\nb-invariant.
  So~\(M'\)
  is a \(\Kring\)\nb-submodule.
  By construction, \(\Ga{0}{1}^{M'}\)
  and \(\Ga{2}{1}^{M'}\)
  are surjective and hence \(\Ga{2}{0}^{M'}=0\)
  and \(\Ga{0}{2}^{M'}=0\).
  Next we claim that \(\Ga{1}{0}^{M'}\)
  and~\(\Ga{1}{2}^{M'}\)
  are injective.  To prove this, we apply
  Proposition~\ref{pro:divisible_S-module} to~\(M'\).
  This gives us equalities
  \begin{align*}
     \ker N(\Gs{1}^{M'}) &= \im (\Gu{1}^{M'} - \Gs{1}^{M'}),\\
     \im N(\Gs{1}^{M'}) &= \ker (\Gu{1}^{M'} - \Gs{1}^{M'}),\\
    M' &= \im (\Gu{1}^{M'} - \Gs{1}^{M'}) \oplus \im N(\Gs{1}^{M'}),
  \end{align*}
  and shows that the restrictions of \(\Gu{1}^{M'} - \Gs{1}^{M'}\)
  and \(N(\Gs{1}^{M'})\)
  to the appropriate direct summands are invertible.  The
  map~\(\Ga{0}{1}^{M'}\)
  vanishes on \(\im (\Gu{1}^{M'} - \Gs{1}^{M'})\)
  and is injective on \(\im N(\Gs{1}^{M'})\)
  because \(N(\Gs{1}^{M'}) = \Ga{1}{0}^{M'}\circ\Ga{0}{1}^{M'}\)
  is injective there.  So
  \(M'_0 = \Ga{0}{1}^{M'}(M'_1) \cong \im N(\Gs{1}^{M'})\).
  Similarly,
  \(M'_2 = \Ga{2}{1}^{M'}(M'_1) \cong \im (\Gu{1}^{M'} -
  \Gs{1}^{M'})\).
  And since \(N(\Gs{1}^{M'})\)
  and \(\Gu{1}^{M'} - \Gs{1}^{M'}\)
  are invertible on these direct summands in~\(M_1'\),
  the maps \(\Ga{1}{0}^{M'}\)
  and~\(\Ga{1}{2}^{M'}\)
  are injective, and their images are \(\im N(\Gs{1}^{M'})\)
  and \(\im (\Gu{1}^{M'} - \Gs{1}^{M'})\),
  respectively, which are also the kernels of the maps
  \(\Ga{2}{1}^{M'}\)
  and~\(\Ga{0}{1}^{M'}\).
  So the \(\Kring\)\nb-module~\(M'\)
  is exact.  The proof also describes it so explicitly that its unique
  \(p\)\nb-divisibility
  is manifest.  And we have directly shown that it has the form of
  Example~\ref{exa:divisible} with \(X=\im N(\Gs{1}^{M'})\),
  \(Y=0\) and \(Z=\im (\Gu{1}^{M'} - \Gs{1}^{M'})\).  A quotient
  of an exact \(\Kring\)\nb-module by an exact
  \(\Kring\)\nb-submodule inherits exactness by the long exact
  sequence in homology.  So \(M/M'\) is exact.
\end{proof}

\section{Exact modules where one generator vanishes}
\label{sec:generator_vanishes}

Section~\ref{sec:uniquely_divisible} already describes exact
\(\Kring\)\nb-modules
where one of the pieces~\(M_j\)
vanishes or, equivalently, where one of the idempotents~\(\Gu{j}\)
vanishes.  Now we study exact \(\Kring\)\nb-modules
where one of the generators~\(\Ga{j}{k}\)
vanishes.  Then one of the long exact sequences in
Definition~\ref{def:exact_R-module} reduces to a short exact
sequence.  This is particularly helpful if the object in the middle
of the short exact sequence is not~\(M_2\)
because the internal structure of \(M_0\)
and~\(M_1\)
is simpler than that of~\(M_2\).

\subsection{The case \texorpdfstring{$\Ga{2}{1}^M=0$}{α₂₁=0}}
\label{sec:a12_zero}

Let~\(M\)
be an exact \(\Kring\)\nb-module
on which~\(\Ga{2}{1}^M\)
vanishes.  This puts a restriction on~\(M_1\):
if \(\Ga{2}{1}^M=0\),
then \(\Gs{1}^M = \Gu{1}^M\)
and \(\Gs{2}^M = \Gu{2}^M\)
by \eqref{eq:alpha_121} and~\eqref{eq:alpha_212}.  By exactness,
\(\Ga{2}{1}^M=0\)
is equivalent to~\(\Ga{1}{0}^M\)
being surjective and to~\(\Ga{0}{2}^M\)
being injective.  And there is a short exact sequence
\begin{equation}
  \label{eq:short_exact_sequence_201}
  \begin{tikzcd}
    M_2 \arrow[r, >->, "\Ga{0}{2}^M"] &
    M_0 \arrow[r, ->>, "\Ga{1}{0}^M"] &
    M_1.
  \end{tikzcd}
\end{equation}

Theorem~\ref{the:a12_vanishes} will describe the category of exact
\(\Kring\)\nb-modules
with \(\Ga{2}{1}^M = 0\)
through certain triples.  To clarify this, we first simplify an exact
\(\Kring\)\nb-module
with \(\Ga{2}{1}^M=0\).
Its central piece is~\(M_0\),
which is a module over \(\Sring \defeq\Z[t]/(t^p-1)\).
We briefly write~\(t\)
instead of~\(\Gt{0}^M\)
to reduce clutter.  This operator will also determine the
operator~\(\Gt{2}^M\).
And \(\Ga{2}{1}^M=0\)
and \eqref{eq:alpha_212} and~\eqref{eq:alpha_121} imply
\(\Gs{1}= \Gu{1}\) and \(\Gs{2}= \Gu{2}\).

First, we use the exact sequence above to identify \(M_2\)
and~\(M_1\)
with the image and the cokernel of the injective map~\(\Ga{0}{2}^M\).
So~\(M_2\)
becomes a \(\Z/2\)\nb-graded
subgroup of~\(M_0\)
and~\(M_1\)
becomes the quotient \(M_0/M_2\),
and \(\Ga{0}{2}^M\)
and~\(\Ga{1}{0}^M\)
become the inclusion map \(M_2 \into M_0\)
and the quotient map \(M_0 \prto M_1\), respectively.

The relations \eqref{eq:alpha_020} and~\eqref{eq:alpha_010} imply that
\(\Ga{2}{0}^M\colon M_0 \to M_2 \subseteq M_0\)
is the corestriction of \(1-t\colon M_0 \to M_0\)
and that \(\Ga{0}{1}^M\colon M_1 \to M_0\)
is the map on the quotient induced by~\(N(t)\).
For this to make sense, we need
\begin{equation}
  \label{eq:a21_condition_M2}
  \im (1-t) \subseteq M_2 \subseteq \ker(N(t)).
\end{equation}
Let
\[
  H_0(M_0) \defeq \frac{\ker N(t)}{\im (1-t)}.
\]
The \(\Z/2\)\nb-graded
\(\Sring\)\nb-module
structure on~\(M_0\)
makes this subquotient a \(\Z/2\)\nb-graded
module over \(\Z[t]/(1-t,N(t)) \cong \Z/p\).
A \(\Z/2\)\nb-graded
subgroup~\(M_2\)
as in~\eqref{eq:a21_condition_M2} is equivalent to a
\(\Z/2\)\nb-graded
\(\Z/p\)-vector
subspace \(\dot{M}_2 \subseteq H_0(M_0)\).
For an exact \(\Kring\)\nb-module,
the kernel and image of \(\Ga{1}{2}^M\colon M_2 \to M_1\) must be
\begin{align*}
  \ker \Ga{1}{2}^M &= \im \Ga{2}{0}^M = \im (1-t) \subseteq M_2,\\
  \im \Ga{1}{2}^M &= \ker \Ga{0}{1}^M = \ker N(t)/M_2 \subseteq M_0/M_2 = M_1.
\end{align*}
So~\(\Ga{1}{2}^M\) is induced by an isomorphism
\[
  \dot\alpha\colon \dot{M}_2
  = \frac{M_2}{\im (1-t)} \cong \frac{\ker N(t)}{M_2} \cong
  \frac{H_0(M_0)}{\dot{M}_2}.
\]
The map~\(\dot\alpha\)
must reverse the grading because~\(\Ga{1}{2}^M\)
is odd, and
\begin{equation}
  \label{eq:a21_vanishes_image_a12}
  \dot{M}_2 \cong \im(\Ga{1}{2}^M).
\end{equation}

\begin{theorem}
  \label{the:a12_vanishes}
  The category of exact \(\Kring\)\nb-modules
  with \(\Ga{2}{1}^M=0\)
  is equivalent to the category of triples
  \((M_0,\dot{M}_2,\dot\alpha)\)
  where
  \begin{enumerate}
  \item \(M_0\)
    is a \(\Z/2\)\nb-graded
    \(\Sring\)\nb-module with \(\ker (1-t) = \im N(t)\);
  \item \(\dot{M}_2 \subseteq H_0(M_0)\)
    is a \(\Z/2\)\nb-graded \(\Z/p\)\nb-vector subspace;
  \item \(\dot\alpha\)
    is a grading-reversing isomorphism
    \(\dot\alpha\colon \dot{M}_2 \to H_0(M_0)/\dot{M}_2\).
  \end{enumerate}
  Morphisms of triples
  \((M_0,\dot{M}_2,\dot\alpha) \to (L_0,\dot{L}_2,\dot\beta)\)
  are grading-preserving \(\Sring\)\nb-module
  homomorphisms \(f_0\colon M_0 \to L_0\)
  such that \(H_0(f_0)\colon H_0(M_0) \to H_0(L_0)\)
  maps \(\dot{M}_2\)
  to~\(\dot{L}_2\)
  and intertwines the maps \(\dot\alpha\)
  and~\(\dot\beta\).
  The homomorphism of exact \(\Kring\)\nb-modules
  induced by a morphism of triples \(f_0\colon M_0 \to L_0\)
  is injective if and only if~\(f_0\)
  is injective and \(f_0^{-1}(L_2) = M_2\);
  it is surjective if and only if~\(f_0\)
  is surjective and \(f_0(M_2) = L_2\).
\end{theorem}

\begin{proof}
  We have already associated \(M_0\),
  \(\dot{M}_2\)
  and~\(\dot\alpha\)
  to an exact \(\Kring\)\nb-module.
  The exactness of one of the two sequences in
  Definition~\ref{def:exact_R-module} is built into our Ansatz.  By
  construction, the sequence \(\Ga{0}{1}^M, \Ga{2}{0}^M, \Ga{1}{2}^M\)
  is exact at \(M_2\)
  and~\(M_1\)
  if and only if the map~\(\dot\alpha\)
  is invertible.  In addition,
  \[
    \im (\Ga{0}{1}^M)
    = \im(\Ga{0}{1}^M \circ \Ga{1}{0}^M)
    = \im N(t)
  \]
  because~\(\Ga{1}{0}^M\)
  is surjective, and
  \[
    \ker (\Ga{2}{0}^M)
    = \ker(\Ga{0}{2}^M \circ \Ga{2}{0}^M)
    = \ker (1-t)
  \]
  because~\(\Ga{0}{2}^M\)
  is injective.  Exactness at~\(M_0\)
  requires \(\ker (\Ga{2}{0}^M) = \im (\Ga{0}{1}^M)\),
  which is equivalent to \(\ker (1-t) = \im N(t)\).
  So the triple \((M_0,\dot{M}_2,\dot\alpha)\)
  associated to an exact \(\Kring\)\nb-module
  satisfies the conditions in the theorem.  Conversely, assume that
  such a triple is given.  Reading our construction backwards, let
  \(M_2\subseteq \im(1-t)\)
  be the pre-image of~\(\dot{M}_2\)
  and let \(M_1\defeq M_0/M_2\).
  Let~\(\Ga{0}{2}^M\)
  be the inclusion map.  Let~\(\Ga{1}{0}^M\)
  be the quotient map.  Then we may define \(\Ga{2}{0}^M\)
  and~\(\Ga{0}{1}^M\)
  as above, induced by \(1-t\)
  and~\(N(t)\).
  Let~\(\Ga{1}{2}^M\)
  be induced by~\(\dot\alpha\)
  and let \(\Ga{2}{1}^M\defeq 0\).
  Let \(M= M_0 \oplus M_1 \oplus M_2\)
  and let~\(\Gu{j}\)
  for \(j=0,1,2\)
  be the projections onto these summands.  The operators defined above
  clearly satisfy the relations
  \eqref{eq:1_vs_1}--\eqref{eq:alpha_exact}.  By construction,
  \(\Gt{0}^M\)
  is the given operator~\(t\)
  on~\(M_0\),
  and~\eqref{eq:alpha_010} holds.  Since~\(M_1\)
  is a quotient of~\(M_0\),
  this implies~\eqref{eq:alpha_101}.  And~\eqref{eq:alpha_2} is
  equivalent to \(N(t)|_{M_2} = 0\),
  which holds because \(M_2 \subseteq \ker N(t)\).
  So all conditions in Theorem~\ref{the:generators_and_relations} hold
  and we have built a \(\Kring\)\nb-module.
  The arguments above also show that it is exact.

  Let \(M\)
  and~\(L\)
  be exact \(\Kring\)\nb-modules
  with \(\Ga{2}{1}^M = 0\)
  and \(\Ga{2}{1}^L = 0\)
  and let \(f\colon M \to L\)
  be a \(\Kring\)\nb-module
  homomorphism.  This specialises to an \(\Sring\)\nb-module
  homomorphism \(f_0\colon M_0 \to L_0\).
  Any \(\Sring\)\nb-module
  homomorphism maps \(\im N(t^M)\),
  \(\ker N(t^M)\),
  \(\im (1-t^M)\)
  and \(\ker (1-t^M)\)
  to the corresponding subgroups in~\(L_0\).
  So the condition
  \(f_0(\Ga{0}{2}^M(M_2)) \subseteq \Ga{0}{2}^L(L_2)\)
  is equivalent to \(H_0(f_0)(\dot{M}_2) \subseteq \dot{L}_2\).
  If an \(\Sring\)\nb-module
  homomorphism~\(f_0\)
  satisfies this, then it induces maps \(f_j\colon M_j \to L_j\)
  for \(j=1,2\)
  that intertwine the actions of \(\Ga{0}{1}\),
  \(\Ga{1}{0}\),
  \(\Ga{2}{0}\),
  \(\Ga{0}{2}\).
  They also intertwine the vanishing actions of~\(\Ga{2}{1}\).
  And \(f_1 \circ \Ga{1}{2}^M = \Ga{1}{2}^M \circ f_2\)
  is equivalent to
  \(H_0(f_0) \circ \dot\alpha = \dot\beta \circ H_0(f_0)\).
  This proves the asserted equivalence of categories.

  The homomorphism~\(f\)
  is injective or surjective if and only if \(f_2,f_0,f_1\)
  are injective, respectively.  View these maps as a transformation
  between the short exact
  sequences~\eqref{eq:short_exact_sequence_201} for our two exact
  \(\Kring\)\nb-modules.
  If~\(f_0\)
  is injective, then so is~\(f_2\),
  and~\(f_1\)
  is injective if and only if \(f_0^{-1}(L_2) = M_2\).
  If~\(f_0\)
  is surjective, so is~\(f_1\),
  and~\(f_2\)
  is surjective if and only if \(f_0(M_2) = L_2\).
  These are the criteria in the theorem for~\(f\)
  to be injective or surjective, respectively.
\end{proof}

The case where both \(\Ga{1}{2}^M\)
and~\(\Ga{2}{1}^M\)
vanish is particularly simple:

\begin{corollary}
  \label{cor:a_2112_exact}
  The category of exact \(\Kring\)\nb-modules
  with \(\Ga{1}{2}^M=0\)
  and \(\Ga{2}{1}^M=0\)
  is equivalent to the category of \(\Z/2\)\nb-graded
  cohomologically trivial \(\Sring\)\nb-modules.
  A homomorphism \(f_0\colon M_0\to L_0\)
  between two \(\Z/2\)\nb-graded
  cohomologically trivial \(\Sring\)\nb-modules
  induces an injective or surjective homomorphism of exact
  \(\Kring\)\nb-modules
  if and only if~\(f_0\) is injective or surjective, respectively.
\end{corollary}

\begin{proof}
  Under the equivalence in Theorem~\ref{the:a12_vanishes}, the
  condition \(\Ga{1}{2}^M=0\)
  corresponds to~\(\dot\alpha=0\).
  Then \(\dot{M}_2=\{0\}\)
  and hence \(\im (1-t) = \ker N(t)\).
  So the pieces \(\dot{M}_2\)
  and~\(\dot\alpha\)
  in the triples in Theorem~\ref{the:a12_vanishes} become redundant in
  this case.  Both \(M_1\)
  and~\(M_2\)
  are naturally isomorphic to subobjects and quotients of~\(M_0\).
  This implies the criterion for injectivity and surjectivity of the
  induced \(\Kring\)\nb-module homomorphism.
\end{proof}

The exact \(\Kring\)\nb-module
built from a cohomologically trivial \(\Sring\)\nb-module~\(M_0\) has
\begin{align*}
  M_1 &= \im N(t) = \ker (1-t),\\
  M_2 &= M_0/M_1 = M_0/\ker(1-t) \cong \im(1-t) = \ker N(t).
\end{align*}
The map \(\Ga{0}{1}^M\colon \im N(t) \into M_0\)
is the canonical inclusion map and
\(\Ga{1}{0}^M\colon M_0 \prto \im N(t)\)
is~\(N(t)\).
When we identify~\(M_2\)
with \(\im(1-t)\),
then \(\Ga{0}{2}^M\colon \im(1-t) \into M_0\)
also becomes the canonical inclusion map and
\(\Ga{2}{0}^M\colon M_0 \prto \im (1-t)\)
becomes \(1-t\),
respectively.  This is because the isomorphism
\(M_0/\ker(1-t) \cong \im(1-t)\)
applies the map~\(1-t\).
If~\(M_0\)
is \(\Z/2\)\nb-graded,
then the pieces \(M_1\)
and~\(M_2\) have the induced \(\Z/2\)\nb-grading.

% \begin{example}
%   \label{exa:cohomologically_trivial}
%   The ring~\(\Sring\)
%   itself is easily seen and well known to be cohomologically trivial.
%   The image of~\(N(t)\)
%   in~\(\Sring\)
%   is isomorphic to~\(\Z\).
%   Let~\(X\)
%   be a \(\Z/2\)\nb-graded
%   Abelian group.  Let \(M_0 \defeq X \otimes_\Z \Sring\).
%   Since the subgroups \(\ker N(t)\),
%   \(\im N(t)\),
%   \(\ker (1-t)\),
%   \(\im (1-t)\)
%   of~\(\Sring\)
%   are all torsion-free, the tensor product~\(M_0\)
%   inherits the property of being cohomologically trivial from~\(\Sring\).
%   And the image of~\(N(t)\)
%   in~\(M_0\)
%   is the tensor product of~\(X\)
%   with the image of~\(N(t)\)
%   in~\(\Sring\).  So~\(M_1\) is the given group~\(X\).
% \end{example}

% This example will reappear as Example~\ref{exa:actions_on_Cuntz_1}.
% Examples \ref{exa:actions_on_Cuntz_8} and~\ref{exa:actions_on_Cuntz_9}
% describe more complicated examples of cohomologically trivial
% \(\Sring\)\nb-modules.

Izumi~\cite{Izumi:Finite_group} has met a stronger property than
cohomological triviality in his classification of group actions with
the Rokhlin property.  Roughly speaking, his assumption of complete
cohomological triviality says that an \(\Sring\)\nb-module
and all the submodules of \(n\)\nb-torsion
elements for \(n\in\N_{\ge2}\)
are cohomologically trivial.  He shows that this condition is
necessary and sufficient for an \(\Sring\)\nb-module
to be isomorphic to~\(\K_*(A)\)
for an action with the Rokhlin property on a simple unital
\(\Cst\)\nb-algebra.
Whereas Izumi can study all finite groups, we must limit our study to
cyclic groups of prime order.  As a benefit, we can treat actions
without the Rokhlin property.  The class of examples treated in
Corollary~\ref{cor:a_2112_exact} seems \(\K\)\nb-theoretically
quite close to the examples studied by Izumi
in~\cite{Izumi:Finite_group}.  Our classification theorem for these
modules only needs the \(\Sring\)\nb-module~\(\K_*(A)\)
and no further data.  There are cohomologically trivial
\(\Sring\)\nb-modules
that are not completely cohomologically trivial and hence cannot
correspond to actions with the Rokhlin property.  In fact,
Example~\ref{exa:actions_on_Cuntz_9} is of this type.  It corresponds
to certain \(\Z/p\)\nb-actions
on the Cuntz algebras \(\mathcal{O}_{p^k+1}\) for \(k\ge1\).

\subsection{The case \texorpdfstring{$\Ga{1}{2}^M=0$}{α₁₂=0}}
\label{sec:a21_zero}

Now let~\(M\)
be an exact \(\Kring\)\nb-module
on which~\(\Ga{1}{2}^M\)
vanishes.  By exactness, this is equivalent to~\(\Ga{2}{0}^M\)
being surjective and to~\(\Ga{0}{1}^M\)
being injective.  In other words, there is a short exact sequence
\begin{equation}
  \label{eq:short_exact_sequence_102}
  \begin{tikzcd}
    M_1 \arrow[r, >->, "\Ga{0}{1}^M"] &
    M_0 \arrow[r, ->>, "\Ga{2}{0}^M"] &
    M_2.
  \end{tikzcd}
\end{equation}
This case is very similar to the case \(\Ga{2}{1}^M=0\)
studied in Section~\ref{sec:a12_zero}.  Hence we go through the
construction more quickly.  We write \(t\defeq\Gt{0}^M\)
for the generator of the \(\Sring\)\nb-module
structure on~\(M_0\).
We identify~\(M_1\)
with a subgroup of~\(M_0\)
and~\(M_2\)
with the quotient \(M_0/M_1\),
such that~\(\Ga{0}{1}^M\)
and~\(\Ga{2}{0}^M\)
become the inclusion map \(M_1 \into M_0\)
and the quotient map \(M_0 \prto M_1\),
respectively.  The map
\(\Ga{1}{0}^M\colon M_0 \to M_1 \subseteq M_0\)
is the corestriction of \(N(t)\colon M_0 \to M_0\),
and \(\Ga{0}{2}^M\colon M_2 \to M_0\)
is the map on the quotient induced by~\(1-t\).
This makes sense if and only if
\begin{equation}
  \label{eq:a12_condition_M1}
  \im N(t) \subseteq M_1 \subseteq \ker(1-t).
\end{equation}
Let
\[
  H^0(M_0) \defeq \frac{\ker (1-t)}{\im N(t)}.
\]
This is naturally a module over \(\Z[t]/(1-t,N(t)) \cong \Z/p\).
A subgroup~\(M_1\)
as in~\eqref{eq:a12_condition_M1} is equivalent to a \(\Z/p\)-vector
subspace \(\dot{M}_1 \subseteq H^0(M_0)\).
Equation~\eqref{eq:alpha_exact} and exactness say that the remaining
generator \(\Ga{2}{1}^M\colon M_1 \to M_2\)
has kernel \(\im N(t) = \im \Ga{1}{0}^M \subseteq M_1\)
and image
\(\ker \Ga{0}{2}^M = \ker (1-t)/M_2 \subseteq M_0/M_1 = M_2\).
So~\(\Ga{2}{1}^M\) must be induced by a grading-reversing isomorphism
\[
\dot\alpha\colon \dot{M}_1 \congto \frac{H^0(M_0)}{\dot{M}_1}.
\]
The construction also shows that
\begin{equation}
  \label{eq:a12_vanishes_image_a21}
  \dot{M}_1 \cong \im(\Ga{2}{1}^M).
\end{equation}

\begin{theorem}
  \label{the:a21_vanishes}
  The category of exact \(\Kring\)\nb-modules
  with \(\Ga{1}{2}^M=0\)
  is equivalent to the category of triples
  \((M_0,\dot{M}_1,\dot\alpha)\)
  where
  \begin{enumerate}
  \item \(M_0\)
    is a \(\Z/2\)\nb-graded
    \(\Sring\)\nb-module with \(\ker N(t) = \im (1-t)\);
  \item \(\dot{M}_1 \subseteq H^0(M_0)\)
    is a \(\Z/2\)\nb-graded \(\Z/p\)\nb-vector subspace;
  \item \(\dot\alpha\)
    is a grading-reversing isomorphism
    \(\dot\alpha\colon \dot{M}_1 \to H^0(M_0)/\dot{M}_1\).
  \end{enumerate}
  Morphisms of triples
  \((M_0,\dot{M}_1,\dot\alpha) \to (L_0,\dot{L}_1,\dot\beta)\)
  are grading-preserving \(\Sring\)\nb-module
  homomorphisms \(f_0\colon M_0 \to L_0\)
  such that \(H^0(f_0)\colon H^0(M_0) \to H^0(L_0)\)
  maps~\(\dot{M}_1\)
  to~\(\dot{L}_1\)
  and intertwines the maps \(\dot\alpha\)
  and~\(\dot\beta\).
  The homomorphism of exact \(\Kring\)\nb-modules
  induced by a morphism of triples \(f_0\colon M_0 \to L_0\)
  is injective if and only if~\(f_0\)
  is injective and \(f_0^{-1}(L_1) = M_1\);
  it is surjective if and only if~\(f_0\)
  is surjective and \(f_0(M_1) = L_1\).
\end{theorem}

The proof of Theorem~\ref{the:a21_vanishes} is analogous to the proof of
Theorem~\ref{the:a12_vanishes}, and we omit further details.
Theorem~\ref{the:a21_vanishes} also specialises to the classification
for exact \(\Kring\)\nb-modules
with \(\Ga{1}{2}^M=0\)
and \(\Ga{2}{1}^M=0\) in Corollary~\ref{cor:a_2112_exact}.

\subsection{Other cases of a vanishing generator}
\label{sec:other_vanishing_generator}

The cases where one of the maps \(\Ga{2}{0}^M\)
or~\(\Ga{0}{2}^M\)
vanishes are analogous to the cases where \(\Ga{2}{1}^M\)
or~\(\Ga{1}{2}^M\)
vanishes.  We may, in fact, reduce these cases to the cases treated
above using the automorphism~\(\BS\)
of the ring~\(\Kring\),
which exchanges the role of the summands \(M_0\)
and~\(M_1\)
in a \(\Kring\)\nb-module.
The analogues of Theorems \ref{the:a12_vanishes}
and~\ref{the:a21_vanishes} for these situations are:

\begin{theorem}
  \label{the:a_20_exact}
  The category of exact \(\Kring\)\nb-modules
  with \(\Ga{2}{0}^M=0\)
  is equivalent to the category of triples
  \((M_1,\dot{M}_2,\dot\alpha)\)
  where~\(M_1\) is a \(\Z/2\)\nb-graded \(\Sring\)\nb-module that satisfies
  \begin{equation}
    \label{eq:a20_condition_exact}
    \ker (1-t) = \im N(t),
  \end{equation}
  \(\dot{M}_2 \subseteq H_0(M_1)\)
  is a \(\Z/2\)\nb-graded
  \(\Z/p\)\nb-vector
  subspace and~\(\dot\alpha\)
  is a grading-preserving isomorphism
  \(\dot\alpha\colon \dot{M}_2 \to H_0(M_1)/\dot{M}_2\).

  Here a morphism
  \((M_1,\dot{M}_2,\dot\alpha) \to (L_1,\dot{L}_2,\dot\beta)\)
  is a grading-preserving \(\Sring\)\nb-module
  homomorphism \(f_1\colon M_1 \to L_1\)
  such that \(H_0(f_1)\colon H_0(M_1) \to H_0(L_1)\)
  maps \(\dot{M}_2\)
  to~\(\dot{L}_2\)
  and intertwines the maps \(\dot\alpha\)
  and~\(\dot\beta\).
  The homomorphism of exact \(\Kring\)\nb-modules
  induced by a morphism of triples \(f_1\colon M_1 \to L_1\)
  is injective if and only if~\(f_1\)
  is injective and \(M_2 = f^{-1}(L_2)\);
  it is surjective if and only if \(f_1\colon M_1 \to L_1\)
  is surjective and \(L_2 = f(M_2)\).
\end{theorem}

\begin{theorem}
  \label{the:a_02_exact}
  The category of exact \(\Kring\)\nb-modules
  with \(\Ga{0}{2}^M=0\)
  is equivalent to the category of triples
  \((M_1,\dot{M}_0,\dot\alpha)\)
  where~\(M_1\) is a \(\Z/2\)\nb-graded \(\Sring\)\nb-module that satisfies
  \begin{equation}
    \label{eq:a02_condition_exact}
    \ker N(t) = \im (1-t),
  \end{equation}
  \(\dot{M}_0 \subseteq H^0(M_1)\)
  is a \(\Z/2\)\nb-graded
  \(\Z/p\)\nb-vector
  subspace and~\(\dot\alpha\)
  is a grading-preserving isomorphism
  \(\dot\alpha\colon \dot{M}_0 \to H^0(M_1)/\dot{M}_0\).

  Morphisms
  \((M_1,\dot{M}_0,\dot\alpha) \to (L_1,\dot{L}_0,\dot\beta)\)
  are grading-preserving \(\Sring\)\nb-module
  homomorphisms \(f_1\colon M_1 \to L_1\)
  such that \(H^0(f_1)\colon H^0(M_1) \to H^0(L_1)\)
  maps~\(\dot{M}_0\)
  to~\(\dot{L}_0\)
  and intertwines the maps \(\dot\alpha\)
  and~\(\dot\beta\).
  The homomorphism of exact \(\Kring\)\nb-modules
  induced by a morphism of triples \(f_1\colon M_1 \to L_1\)
  is injective if and only if~\(f_1\)
  is injective and \(f_1^{-1}(L_0) = M_0\);
  and it is surjective if and only if~\(f_1\)
  is surjective and \(L_0 = f(M_0)\).
\end{theorem}

Since \(\Ga{2}{1}\) and~\(\Ga{1}{2}\) are odd,
\[
M_2 \defeq \Sigma \setgiven{x\in \ker N(t)}{(x\bmod \im (1-t) \in \dot{M}_2}
\]
in the situation of Theorem~\ref{the:a_20_exact} and
\(M_2 \defeq \Sigma (M_1/M_0)\)
in the situation of Theorem~\ref{the:a_02_exact}.
That is, the grading on~\(M_2\)
is reversed in both cases.

There is also an analogue of these results for the cases where
\(\Ga{0}{1}^M=0\)
or \(\Ga{1}{0}^M=0\).
Since the automorphism~\(\BS\)
maps these two cases onto each other, it suffices to discuss one of
them.  The condition \(\Ga{0}{1}^M=0\)
holds if and only if~\(\Ga{1}{2}^M\)
is surjective, if and only if~\(\Ga{2}{0}^M\)
is injective, if and only if these maps fit into an exact sequence
\[
  \begin{tikzcd}
    M_0 \arrow[r, >->, "\Ga{2}{0}^M"]&
    M_2 \arrow[r, ->>, "\Ga{1}{2}^M"] &
    M_1.
  \end{tikzcd}
\]
As above, we may use this to identify~\(M_0\)
with a subgroup of~\(M_2\)
and~\(M_1\)
with \(\Sigma(M_2/M_0)\),
so that~\(\Ga{2}{0}^M\)
becomes the inclusion map and~\(\Ga{1}{2}^M\)
the quotient map, made grading-reversing.  Now~\(M_2\)
is a module over the ring~\(\Sring_2\)
defined in~\eqref{eq:S2}, and we denote its generators more briefly as
\(t^M = \Gt{2}^M\) and \(s^M = \Gs{2}^M\).

The map \(\Ga{0}{2}^M\colon M_2 \to M_0 \subseteq M_2\)
must be the corestriction of
\(\Ga{2}{0}^M \circ \Ga{0}{2}^M = 1 - t^M\)
by~\eqref{eq:alpha_202}, whereas \(\Ga{2}{1}^M\colon M_1 \to M_2\)
must be the map on the quotient induced by
\(\Ga{2}{1}^M \circ \Ga{1}{2}^M = 1-s^M\)
by~\eqref{eq:alpha_212}.  So we need
\(\im(1-t^M) \subseteq M_0 \subseteq \ker(1-s^M)\).
Let
\[
H(M_2) \defeq \ker(1-s^M)/\im(1-t^M).
\]
The maps on the subquotient~\(H(M_2)\)
induced by \(1-s^M\)
and~\(1-t^M\)
vanish, and then so does multiplication by~\(p\)
because \(p = (p-N(t)) +(p-N(s))\)
holds in~\(\Sring_2\).
So~\(H(M_2)\)
is a vector space over~\(\Z/p\).
A subgroup~\(M_0\)
as above is the preimage of a unique \(\Z/p\)\nb-linear
subspace~\(\dot{M}_0\)
in~\(H(M_2)\).
Exactness dictates that the missing map~\(\Ga{1}{0}^M\)
is a composite map
\[
M_0 \prto M_0/\im(1-t) = \dot{M}_0 \xrightarrow{\dot\alpha}
H(M_2)/\dot{M}_0 \cong \ker(1-s)/M_0 \subseteq M_2/M_0 = M_1
\]
for an isomorphism
\(\dot\alpha\colon \dot{M}_0 \congto H(M_2)/\dot{M}_0\)
and that \(\ker(1-t^M) = \im(1-s^M)\).
Conversely, any triple \((M_2,\dot{M}_0,\dot\alpha)\)
as above comes from an exact \(\Kring\)\nb-module
with \(\Ga{0}{1}^M=0\),
and the construction above is an equivalence of categories.

\begin{theorem}
  \label{the:a01_exact_module}
  The category of exact \(\Kring\)\nb-modules
  with \(\Ga{0}{1}^M=0\)
  is equivalent to the category of triples
  \((M_2,\dot{M}_0,\dot\alpha)\)
  where~\(M_2\)
  is a \(\Z/2\)\nb-graded \(\Sring_2\)\nb-module that satisfies
  \begin{equation}
    \label{eq:a01_condition_exact}
    \ker (1-t^M) = \im (1-s^M),
  \end{equation}
  \(\dot{M}_0 \subseteq H(M_2)\)
  is a \(\Z/2\)\nb-graded
  \(\Z/p\)\nb-vector
  subspace and~\(\dot\alpha\)
  is a grading-preserving isomorphism
  \(\dot\alpha\colon \dot{M}_0 \to H(M_2)/\dot{M}_0\).

  Here a morphism
  \((M_2,\dot{M}_0,\dot\alpha) \to (L_2,\dot{L}_0,\dot\beta)\)
  is a grading-preserving \(\Sring_2\)\nb-module
  homomorphism \(f_2\colon M_2 \to L_2\)
  such that \(H(f_2)\colon H(M_2) \to H(L_2)\)
  maps~\(\dot{M}_0\)
  to~\(\dot{L}_0\)
  and intertwines the maps \(\dot\alpha\)
  and~\(\dot\beta\).
  The homomorphism of exact \(\Kring\)\nb-modules
  induced by a morphism of triples \(f_2\colon M_2 \to L_2\)
  is injective if and only if~\(f_2\)
  is injective and \(M_0 = f^{-1}(L_0)\);
  it is surjective if and only if \(f_2\colon M_2 \to L_2\)
  is surjective and \(L_0 = f(M_0)\).
\end{theorem}

As above, the case where both \(\Ga{0}{1}^M=0\)
and \(\Ga{1}{0}^M=0\)
corresponds to cohomologically trivial \(\Sring_2\)\nb-modules,
and then \(\dot{M}_0\) and~\(\dot\alpha\) are redundant:

\begin{corollary}
  \label{cor:a_1001_exact}
  The category of exact \(\Kring\)\nb-modules
  with \(\Ga{0}{1}^M=0\)
  and \(\Ga{1}{0}^M=0\)
  is equivalent to the category of \(\Z/2\)\nb-graded
  cohomologically trivial \(\Sring_2\)\nb-modules.

  A homomorphism \(f_2\colon M_2\to L_2\)
  between two \(\Z/2\)\nb-graded
  cohomologically trivial \(\Sring_2\)\nb-modules
  induces an injective or surjective homomorphism of exact
  \(\Kring\)\nb-modules
  if and only if~\(f_2\) is injective or surjective, respectively.
\end{corollary}

\section{Examples of exact modules}
\label{sec:representative_examples}

We now describe some exact \(\Kring\)\nb-modules
where~\(M_1\)
is a cyclic group.  These correspond to actions of \(G\defeq \Z/p\)
on Cuntz algebras.  The examples in this section are representative in
the following sense:

\begin{theorem}
  \label{the:actions_on_Cuntz}
  Let~\(M\)
  be an exact \(\Kring\)\nb-module
  where~\(M_1\)
  is a cyclic group.  Then there is an extension
  \(M' \into M \prto M''\)
  of exact \(\Kring\)\nb-modules
  where~\(M'\)
  is one of the exact \(\Kring\)\nb-modules
  described in Examples
  \textup{\ref{exa:actions_on_Cuntz_1}},
  \textup{\ref{exa:actions_on_Cuntz_2}},
  \textup{\ref{exa:actions_on_Cuntz_3}},
  \textup{\ref{exa:actions_on_Cuntz_4}},
  \textup{\ref{exa:actions_on_Cuntz_6}},
  \textup{\ref{exa:actions_on_Cuntz_8}},
  \textup{\ref{exa:actions_on_Cuntz_5}},
  \textup{\ref{exa:actions_on_Cuntz_7}},
  \textup{\ref{exa:actions_on_Cuntz_9}}
  and~\textup{\ref{exa:actions_on_Cuntz_10}}
  and~\(M''\)
  is uniquely \(p\)\nb-divisible.
  So~\(M''\) is as in Example~\textup{\ref{exa:divisible}}.
  The extension above need not split.
\end{theorem}

Theorem~\ref{the:actions_on_Cuntz} will be proven in
Section~\ref{sec:prove_theorem} by a case-by-case study.

For each example, we first define \(\Z/2\)\nb-graded
Abelian groups~\(M_j\)
for \(j=0,1,2\).
So \(M = M_0 \oplus M_1 \oplus M_2\)
and~\(\Gu{j}\)
for \(j=0,1,2\)
is the projection onto the direct summand~\(M_j\).
Then we define maps \(\Ga{j}{m}^M\colon M_m \to M_j\)
for \(j,m\in\{0,1,2\}\)
with \(j\neq m\).
We compute the resulting maps \(\Gs{j}^M\colon M_j \to M_j\)
for \(j=1,2\)
and \(\Gt{j}^M\colon M_j \to M_j\)
for \(j=0,2\)
as defined in~\eqref{eq:def_st}.  Then we check that the relations in
Theorem~\ref{the:generators_and_relations} hold for these maps -- so
our data defines a \(\Kring\)\nb-module
-- and that the exact sequences in Definition~\ref{def:exact_R-module}
are exact -- so our \(\Kring\)\nb-module
is exact.  Here the relations \eqref{eq:1_vs_1}--\eqref{eq:alpha_vs_1}
in Theorem~\ref{the:generators_and_relations} hold by construction,
and~\eqref{eq:alpha_exact} is part of exactness.  The remaining
relations are equivalent to
\eqref{eq:alpha_010_N}--\eqref{eq:alpha_2_N}, which is what we check.

We always define~\(M_0\)
so that it has an obvious module structure over the ring
\(\Sring\defeq \Z[t]/(t^p-1)\)
and write \(f(t)\)
for the map of multiplication by~\(f(t)\)
in this module for \(f\in \Sring\).
Our definitions are always arranged so that \(\Gt{0}^M = t\)
in this notation.  In many examples, \((M_2,\Gt{2}^M)\)
is also a module over~\(\Sring\)
and the same notation applies.  Sometimes, we need the full
\(\Sring_2\)\nb-module
structure on~\(M_2\)
with~\(\Sring_2\)
as in~\eqref{eq:S2}.

The notation \((f)\)
or \((f_1,f_2)\)
in a context such as \(\Sring/(f)\)
or \(\Sring/(f_1,f_2)\)
denotes the ideal in~\(\Sring\)
generated by \(f\)
or by \(f_1,f_2\),
respectively.  Recall that \(\Z[\vartheta] \cong \Sring/(N(t))\)
is the ring generated by a primitive \(p\)th root of unity.

We begin with two examples of exact \(\Kring\)\nb-modules
with \(M_1=\Z\)
and \(\Gs{1}^M=\Gu{1}^M\).

\begin{example}
  \label{exa:actions_on_Cuntz_1}
  Let
  \begin{align*}
    M_0 &\defeq \Sring = \Z[t]/(t^p-1),\\
    M_1 &\defeq \Z,\\
    M_2 &\defeq \Sring/(N(t)) \cong \Z[\vartheta].
  \end{align*}
  All these groups are in even degree.  Let \(\Ga{2}{1}^M \defeq 0\)
  and \(\Ga{1}{2}^M \defeq 0\).
  Let \(\Ga{0}{1}^M\colon M_1 \to M_0\)
  be multiplication with~\(N(t)\).
  Let \(\Ga{1}{0}^M \defeq \mathrm{ev}_1\)
  be evaluation at~\(1\).
  Let~\(\Ga{2}{0}^M\)
  be the quotient map, and let~\(\Ga{0}{2}^M\)
  be induced by multiplication with~\((1-t)\);
  this is well defined and injective because
  \(1-t^p \mid (1-t) \cdot f\)
  for \(f\in \Z[t]\)
  is equivalent to \(N(t) \mid f\).
  We compute \(\Gt{0}^M = t\),
  \(\Gt{2}^M = t\),
  \(\Gs{1}^M = 1\),
  \(\Gs{2}^M = 1\).
  So \(N(\Gt{0}^M) = N(t)\),
  \(N(\Gt{2}^M) = 0\),
  \(N(\Gs{1}^M) = p\),
  \(N(\Gs{2}^M) = p\).
  Now the relations \eqref{eq:alpha_010_N}--\eqref{eq:alpha_2_N}
  become obvious.  We have already checked that~\(\Ga{0}{2}^M\)
  is injective.  And~\(\Ga{0}{1}^M\)
  is clearly injective and \(\Ga{1}{0}^M\)
  and~\(\Ga{2}{0}^M\)
  are surjective.  The kernel of~\(\Ga{2}{0}^M\)
  is the ideal in~\(\Sring\)
  generated by~\(N(t)\),
  which is isomorphic to~\(\Z\)
  because \((1-t)\cdot N(t) = 1-t^p = 0\)
  in~\(\Sring\).
  So the maps \(\Ga{0}{1}^M\)
  and~\(\Ga{2}{0}^M\)
  form a short exact sequence.  So do the maps \(\Ga{0}{2}^M\)
  and~\(\Ga{1}{0}^M\);
  the image of \(\Ga{0}{2}^M\)
  is the kernel of~\(\Ga{1}{0}^M\)
  because any \(f\in\Z[t]\)
  decomposes uniquely as \(f = f(1) + (t-1) f_2\)
  for some \(f_2 \in \Z[t]\).
  Together with \(\Ga{1}{2}^M = 0\)
  and \(\Ga{2}{1}^M = 0\),
  this says that the remaining relation~\eqref{eq:alpha_exact} holds
  and that the sequences in Definition~\ref{def:exact_R-module} are
  exact.  So the data above defines an exact \(\Kring\)\nb-module.
\end{example}

Let \(\mathrm{tv}\)
denote the trivial action of~\(G\)
on the Cuntz algebra~\(\mathcal{O}_\infty\).
Then \(F_*(\mathcal{O}_\infty,\mathrm{tv})\)
is easily seen to be isomorphic to the \(\Kring\)\nb-module
in Example~\ref{exa:actions_on_Cuntz_1}.  This gives another proof
that Example~\ref{exa:actions_on_Cuntz_1} defines an exact
\(\Kring\)\nb-module.

\begin{example}
  \label{exa:actions_on_Cuntz_not_split}
  We prove by an example that the sequence in
  Theorem~\ref{the:actions_on_Cuntz} need not split.  This example
  is a variant of Example~\ref{exa:actions_on_Cuntz_1}.  Let
  \[
    \Z_{(p)} \defeq
    \setgiven{a/b\in\Q}{(p,b)=1}.
  \]
  Let~\(M\) be the exact \(\Kring\)\nb-module in
  Example~\ref{exa:actions_on_Cuntz_1}.  Then \(\Z_{(p)}\otimes M\)
  is another exact \(\Kring\)\nb-module.  Define
  \begin{align*}
    M_0' &\defeq \setgiven{f\in \Z_{(p)}\otimes M_0}{f(1)\in\Z},\\
    M_1' &\defeq \Z \subseteq \Z_{(p)} \cong \Z_{(p)}\otimes M_1,\\
    M_2' &\defeq \Z_{(p)}\otimes M_2.
  \end{align*}
  Inspection shows that this is a \(\Kring\)\nb-submodule of
  \(\Z_{(p)}\otimes M\).  It contains \(M\cong \Z \otimes M\).  The
  quotient \(M'' \defeq M'/M\) is the uniquely \(p\)\nb-divisible
  exact \(\Kring\)\nb-module with \(M_1''=0\) and
  \(M_0'' \cong M_2'' \cong \Z_{(p)}[\vartheta]/\Z[\vartheta]\).
  Since \(M\) and~\(M''\) are exact, so is~\(M'\).  The extension
  \(M \into M' \prto M''\) does not split because
  \(M_2 \into M'_2 \prto M''_2\) does not split.
\end{example}

\begin{example}
  \label{exa:actions_on_Cuntz_2}
  Let~\(Q\) be the quotient \(\Z[\vartheta]\)\nb-module
  \begin{equation}
    \label{eq:def_Q}
    Q \defeq \frac{\Z[\vartheta,1/p]}{(1-\vartheta)\Z[\vartheta]}.
  \end{equation}
  Let
  \[
  M_0 \defeq \Z \oplus \Sigma Q,\qquad
  M_1 \defeq \Z,\qquad
  M_2 \defeq \Sigma Q.
  \]
  Here~\(\Sigma\)
  means that~\(Q\)
  is in degree~\(1\), whereas~\(\Z\) is in degree~\(0\).  Define
  \[
  \Ga{0}{1}^M \defeq
  \begin{pmatrix}
    1\\0
  \end{pmatrix},\quad
  \Ga{2}{0}^M \defeq
  \begin{pmatrix}
    0&1
  \end{pmatrix},\quad
  \Ga{1}{2}^M \defeq 0,\quad
  \Ga{1}{0}^M \defeq
  \begin{pmatrix}
    p&0
  \end{pmatrix},\quad
  \Ga{0}{2}^M \defeq
  \begin{pmatrix}
    0\\1-\vartheta
  \end{pmatrix}.
  \]
  Let \(\Ga{2}{1}^M\colon \Z \to Q\)
  be the composite of the unital ring homomorphism
  \(\Z\hookrightarrow \Z[\vartheta,1/p]\)
  and the quotient map to~\(Q\).
  Notice that \(\Ga{1}{2}^M\)
  and~\(\Ga{2}{1}^M\)
  are grading-reversing, whereas the other~\(\Ga{j}{k}\)
  are grading-preserving.  This is why we put~\(Q\)
  into odd degree.  We compute \(\Gt{0}^M = 1\oplus \vartheta\),
  \(\Gt{2}^M = \vartheta\),
  \(\Gs{1}^M = 1\),
  \(\Gs{2}^M = 1\).
  So \(N(\Gt{0}^M) = p\oplus 0\),
  \(N(\Gt{2}^M) = 0\),
  \(N(\Gs{1}^M) = p\),
  \(N(\Gs{2}^M) = p\).
  The relations \eqref{eq:alpha_010_N}--\eqref{eq:alpha_2_N} are now
  clear.  The maps \(\Ga{0}{1}^M\)
  and \(\Ga{2}{0}^M\)
  clearly form a short exact sequence, which matches \(\Ga{1}{2}^M =0\).
  The image of \(\Ga{1}{0}^M\)
  is \(p\cdot \Z\subseteq \Z\).
  This is equal to the kernel of \(\Ga{2}{1}^M\) because
  \[
  \Z[\vartheta]/(1-\vartheta) \cong \Z[t]/(1-t,N(t))  \cong \Z[t]/(1-t,p) \cong \Z/p.
  \]
  The kernel of \(\Ga{1}{0}^M\)
  is the direct summand~\(\Sigma Q\).
  This is also the image of \(\Ga{0}{2}^M\)
  because multiplication by \(1-\vartheta\)
  on~\(Q\)
  is surjective by Lemma~\ref{lem:S-module_surjective}.  This lemma
  also shows that the kernel of multiplication by \(1-\vartheta\)
  on~\(Q\)
  is the subgroup isomorphic to~\(\Z/p\)
  generated by the class of~\(1\).
  So \(\ker (\Ga{0}{2}^M) = \im (\Ga{2}{1}^M)\).
  Hence the remaining relation~\eqref{eq:alpha_exact} holds and the
  \(\Kring\)\nb-module defined above is exact.
\end{example}

The remaining examples are exact \(\Kring\)\nb-modules
with \(M_1 = \Z/p^k\)
for some \(k\in\N_{\ge1}\).
The following lemma will be used in several of these examples:

\begin{lemma}
  \label{lem:norm_if_unipotent}
  Let~\(X\)
  be an \(\Sring\)\nb-module
  with \((1-t^X)^2=0\).
  Then
  \[
  N(t^X) = p- \frac{p(p-1)}{2}\cdot (1-t^X).
  \]
  If \(p\neq2\) and \(p\cdot (1-t^X)=0\), then \(N(t^X) = p\).
\end{lemma}

\begin{proof}
  Any polynomial \(f\in\Z[t]\)
  may be written as
  \(f= f(1) + f'(1)\cdot(t-1) + f_2(t)\cdot (t-1)^2\).
  Then \(f(t^X) = f(1) + f'(1) (t^X-1)\)
  because \((1-t^X)^2=0\).
  Since \(N(1)=p\)
  and \(N'(1) = \sum_{j=1}^{p-1} j = p(p-1)/2\),
  this implies the first claim.  The second claim follows from the
  first one because \(2 \mid p-1\) if \(p\neq2\).
\end{proof}

We begin with the only example where the action of~\(\Z/p\)
on~\(M_1\) generated by~\(\Gs{1}^M\) is non-trivial:

\begin{example}
  \label{exa:actions_on_Cuntz_3}
  Assume \(k\ge2\)
  if \(p\neq2\)
  and assume \(k\ge3\)
  if \(p=2\).
  Let \(\tau\in \{1,\dotsc,p-1\}\).
  Let \(\nu \defeq p\)
  if \(p\neq2\)
  and \(\nu \defeq 2(1- 2^{k-2}\tau)\) if \(p=2\).  Define
  \[
  M_0 \defeq \Z/p^{k-1},\qquad M_1 \defeq \Z/p^k,\qquad M_2 \defeq \Sigma\Z/p,
  \]
  where~\(\Sigma\)
  denotes that~\(M_2\)
  is in degree~\(1\).
  For \(x\in\Z\),
  let \([x]\)
  denote its class in one of the quotients~\(M_j\),
  \(j=0,1,2\).  Define
  \begin{alignat*}{3}
    \Ga{0}{2}^M &\defeq 0,&\qquad
    \Ga{2}{1}^M[x] &\defeq [\tau\cdot x],&\qquad
    \Ga{1}{0}^M[x] &\defeq [\nu\cdot x],\\
    \Ga{2}{0}^M &\defeq 0,&\qquad
    \Ga{1}{2}^M[x] &\defeq [p^{k-1}\cdot x],&\qquad
    \Ga{0}{1}^M[x] &\defeq [x].
  \end{alignat*}
  Notice that \(\Ga{1}{2}^M\)
  and~\(\Ga{2}{1}^M\)
  are grading-reversing and the other generators grading-preserving.
  We compute \(\Gt{0}^M = 1\),
  \(\Gt{2}^M = 1\),
  \(\Gs{2}^M = 1\)
  because \(k\ge2\),
  \(\Gs{1}^M = 1-\tau p^{k-1}\).
  So \(N(\Gt{0}^M) = p\),
  \(N(\Gt{2}^M) = p\equiv 0\),
  \(N(\Gs{2}^M) = p\equiv 0\).
  We claim that \(N(\Gs{1}^M) = \nu\).
  If \(p=2\),
  then \(N(\Gs{1}^M) = \Gu{1}^M + \Gs{1}^M\)
  and the claim is trivial.  Let \(p\neq2\).
  We compute
  \begin{equation}
    \label{eq:s1_unipotent}
    (1-\Gs{1}^M)^2
    = \Ga{1}{2}^M \Ga{2}{1}^M \Ga{1}{2}^M \Ga{2}{1}^M
    = \Ga{1}{2}^M (1-\Gs{2}^M) \Ga{2}{1}^M
    = 0
  \end{equation}
  and \(p\cdot (1-\Gs{1}^M)=0\).
  Since \(p\neq2\),
  Lemma~\ref{lem:norm_if_unipotent} implies \(N(\Gs{1}^M) = p = \nu\).
  Hence~\eqref{eq:alpha_101_N} holds.  Equation~\eqref{eq:alpha_010_N}
  is trivial for \(p\neq2\)
  and holds for \(p=2\)
  because \(\nu\equiv 2 \bmod 2^{k-1}\).
  And~\eqref{eq:alpha_2_N} is trivial.  The map~\(\Ga{1}{0}^M\)
  is injective with image~\(p\Z/p^k\),
  and~\(\Ga{2}{1}^M\)
  is surjective with kernel~\(p\Z/p^k\).
  The map~\(\Ga{0}{1}^M\)
  is surjective with kernel~\(p^{k-1}\Z/p^k\),
  and~\(\Ga{1}{2}^M\)
  is injective with image~\(p^{k-1}\Z/p^k\).
  Hence~\eqref{eq:alpha_exact} holds and the sequences in
  Definition~\ref{def:exact_R-module} are exact.  So we have defined
  an exact \(\Kring\)\nb-module.
  % \begin{tikzcd}[row sep=large, column sep=large]
  %   &\Z/p^k \ar[dr, shift right, "\tau"'] \ar[dl, shift right, "1"']&\\
  %   \Z/p^{k-1} \ar[rr, shift right, "0"'] \ar[ur, shift right, "N(\Gs{1})"']&&
  %   \Sigma \Z/p \ar[ul, shift right, "p^{k-1}"'] \ar[ll, shift right, "0"']
  % \end{tikzcd}
\end{example}

\begin{example}
  \label{exa:actions_on_Cuntz_4}
  Assume \(p\neq2\)
  or \(k\ge2\).
  Let \(\ell \ge2\) and \(\tau\in \{1,\dotsc,p-1\}\).  Define
  \begin{align*}
    M_0 &\defeq \Z/p^{k-1} \oplus \Sigma\Z[\vartheta]/(1-\vartheta)^{\ell-1},\\
    M_1 &\defeq \Z/p^k,\\
    M_2 &\defeq \Sigma\Z[\vartheta]/(1-\vartheta)^\ell,
  \end{align*}
  \begin{alignat*}{3}
    \Ga{0}{1}^M([x])
    &= \begin{pmatrix} {}[x]\\0 \end{pmatrix},&\quad
    \Ga{2}{0}^M \begin{pmatrix} {}[x]\\{}[f] \end{pmatrix}
    &= {}[(1-\vartheta)f],&\quad
    \Ga{1}{2}^M [f]
    &= [p^{k-1} f(1)],\\
    \Ga{1}{0}^M \begin{pmatrix} {}[x]\\ {}[f] \end{pmatrix}
    &= [p\cdot x],&\quad
    \Ga{0}{2}^M \begin{pmatrix} {}[f] \end{pmatrix}
    &= \begin{pmatrix} 0\\ {}[f] \end{pmatrix},&\quad
    \Ga{2}{1}^M([x]) &= [\tau \cdot (1-\vartheta)^{\ell-1} x].
  \end{alignat*}
  Recall that the inclusion \(\Z \to \Z[\vartheta]\) induces an
  isomorphism \(\Z/p \cong \Z[\vartheta]/(1-\vartheta)\) because
  \((1-t,N(t)) = (1-t,p)\).  The map~\(\Ga{1}{2}^M\) is well defined
  as the composite of this isomorphism with the map
  \(\Z/p \to \Z/p^k\), \(x\mapsto p^{k-1} x\).  The
  map~\(\Ga{1}{2}^M\) is well defined as the composite of the
  inverse of this isomorphism with the map
  \(\Z[\vartheta]/(1-\vartheta) \to \Z[\vartheta]/(1-\vartheta)^l\),
  \(f\mapsto f\cdot (1-\vartheta)^{l-1}\).

  If \(k=1\),
  then the summand~\(\Z/p^{k-1}\)
  in~\(M_0\)
  disappears; this simplifies the formulas for the
  maps~\(\Ga{j}{m}^M\) in this case.  We compute
  \[
  \Gt{0}^M = 1 \oplus\vartheta,\qquad
  \Gs{1}^M = 1,\qquad
  \Gt{2}^M = \vartheta,\qquad
  \Gs{2}^M =
  \begin{cases}
    1&\text{if }k\ge 2,\\
    1-\tau\cdot(1-\vartheta)^{\ell-1}&\text{if }k=1.
  \end{cases}
  \]
  Here \(\Gs{1}^M= \Gu{1}^M\) follows because
  \(\im \Ga{2}{1}^M= (1-\vartheta)^{\ell-1}/(1-\vartheta)^\ell\) is
  contained in
  \(\ker \Ga{1}{2}^M= (1-\vartheta)/(1-\vartheta)^\ell\) because
  \(\ell\ge2\).  We claim that \(N(\Gs{2}^M)=p\).  This is clear if
  \(k\ge2\).  Let \(k=1\).  Since \(\Gs{1}^M=1\), a computation
  analogous to~\eqref{eq:s1_unipotent} shows that
  \((1-\Gs{2}^M)^2= 0\).  And \(p\cdot(1-\Gs{2}^M)= 0\) because
  \((1-\vartheta)^\ell\) divides \(p(1-\vartheta)^{\ell-1}\).  Since
  \(p\neq2\) if \(k=1\), Lemma~\ref{lem:norm_if_unipotent} implies
  \(N(\Gs{2}^M)=p\).  Since \(N(\Gt{2}^M)=N(\vartheta)=0\), the
  condition \(N(\Gs{2}^M) + N(\Gs{2}^M) = p\cdot \Gu{2}^M\)
  in~\eqref{eq:alpha_2_N} holds in all allowed cases.  Equations
  \eqref{eq:alpha_101_N} and~\eqref{eq:alpha_010_N} are easy to
  check.

  To check exactness, we will use that multiplication by
  \((1-\vartheta)^{\ell-1}\) and~\(1-\vartheta\) induce isomorphisms
  \[
  \frac{\Z[\vartheta]}{(1-\vartheta)\Z[\vartheta]}
  \cong \frac{(1-\vartheta)^{\ell-1}\Z[\vartheta]}{(1-\vartheta)^\ell\Z[\vartheta]},
  \qquad
  \frac{\Z[\vartheta]}{(1-\vartheta)^{\ell-1}\Z[\vartheta]}
  \cong \frac{(1-\vartheta)\Z[\vartheta]}{(1-\vartheta)^\ell\Z[\vartheta]},
  \]
  respectively.  Then we compute
  \begin{align*}
    \ker (\Ga{0}{1}^M)
    &= p^{k-1} \Z/p^k
      = \im (\Ga{1}{2}^M),\\
    \ker (\Ga{2}{1}^M)
    &= p \Z/p^k
      = \im (\Ga{1}{0}^M),\\
    \ker (\Ga{1}{2}^M)
    &= (1-\vartheta)/(1-\vartheta)^\ell
      = \im (\Ga{2}{0}^M),\\
    \ker (\Ga{0}{2}^M)
    &= (1-\vartheta)^{\ell-1}/(1-\vartheta)^\ell
      = \im (\Ga{2}{1}^M),\\
    \ker (\Ga{1}{0}^M)
    &= 0\oplus\Sigma\Z[\vartheta]/(1-\vartheta)^{\ell-1}
      = \im (\Ga{0}{2}^M),\\
    \ker (\Ga{2}{0}^M)
    &= \Z/p^{k-1} \oplus 0
      = \im (\Ga{0}{1}^M).
  \end{align*}
  This finishes the proof that~\(M\) is an exact
  \(\Kring\)\nb-module.

  The little invariant contained in~\(M\) does not depend on the
  parameter~\(\tau\).  The following lemma shows, however, that the
  isomorphism class of~\(M\) depends on~\(\tau\).  Since there are
  actions of~\(G\) that realise the exact \(\Kring\)\nb-modules in
  this example, we get counterexamples of objects
  in~\(\mathfrak{B}^G\) that are not \(\KK^G\)\nb-equivalent, but
  have isomorphic little invariants.
\end{example}

\begin{lemma}
  \label{lem:actions_on_Cuntz_4}
  The \(\Kring\)\nb-modules~\(M\) in
  Example~\textup{\ref{exa:actions_on_Cuntz_4}} for different values
  of~\(\tau\) are not isomorphic.
\end{lemma}

\begin{proof}
  Let \((x,y) \in M_1\oplus M_1\).  There is \(b\in M_2\) with
  \((\Gu{2}^M-\Gt{2}^M)^{\ell-1}(b) = \Ga{2}{1}^M(x)\) and
  \(\Ga{1}{2}^M(b) = y\) if and only if \(y = \tau p^{k-1}\cdot x\).
  Hence the map of multiplication by~\(\tau p^{k-1}\) on~\(\Z/p^k\)
  is an isomorphism invariant of~\(M\).  The information in this map
  is exactly the class of~\(\tau\) in~\(\Z/p\).
\end{proof}

Roughly speaking, Example~\ref{exa:actions_on_Cuntz_3} is the case
\(\ell=1\) of Example~\ref{exa:actions_on_Cuntz_4}, which we have
not allowed there.  If we took \(\ell=1\) in
Example~\ref{exa:actions_on_Cuntz_4}, we could simplify
\(\Z[\vartheta]/(1-\vartheta)^{\ell-1}= 0\) and
\(\Z[\vartheta]/(1-\vartheta)^\ell \cong \Z/p\).  Then the formulas
for~\(\Ga{j}{m}^M\) for \((j,m)\neq (0,1)\) are exactly those in
Example~\ref{exa:actions_on_Cuntz_3}.  The formulas
for~\(\Ga{1}{0}^M\) are the same as well if \(p\neq 2\), but not if
\(p=2\).  The subtle choice of~\(\Ga{1}{0}^M\) in
Example~\ref{exa:actions_on_Cuntz_3} ensures
that~\eqref{eq:alpha_101_N} also holds if \(p=2\).

\begin{example}
  \label{exa:actions_on_Cuntz_6}
  Define
  \(Q\defeq \Z[\vartheta,1/p] \bigm/ (1-\vartheta) \Z[\vartheta]\)
  as in~\eqref{eq:def_Q}.  Define
  \begin{align*}
    M_0 &\defeq \frac{\Z/p^k\oplus Q \oplus \Sigma Q}
          {\setgiven{([p^{k-1} x],[x],0)}{[x]\in\Z/p}},\\
    M_1 &\defeq \Z/p^k,\\
    M_2 &\defeq Q \oplus \Sigma Q.
  \end{align*}
  We describe elements of~\(M_0\)
  as equivalence classes of triples \([x,y,z]\)
  with \(x\in \Z/p^k\), \(y,z\in Q\).  Define
  \begin{alignat*}{3}
    \Ga{0}{1}^M([x])
    &\defeq \begin{bmatrix} x\\0\\0 \end{bmatrix},&\quad
    \Ga{2}{0}^M \begin{bmatrix} {}[x]\\ {}[f_1]\\ {}[f_2] \end{bmatrix}
    &\defeq \begin{pmatrix} {}[(1-\vartheta) \cdot f_1]\\ {}[f_2] \end{pmatrix},&\quad
    \Ga{1}{2}^M \begin{pmatrix} {}[f_1]\\ {}[f_2] \end{pmatrix}
    &\defeq 0,\\
    \Ga{1}{0}^M \begin{bmatrix} {}[x]\\{}[f_1]\\ {}[f_2] \end{bmatrix}
    &\defeq [p\cdot x],&\quad
    \Ga{0}{2}^M \begin{pmatrix} {}[f_1]\\ {}[f_2] \end{pmatrix}
    &\defeq \begin{bmatrix} 0\\ {}[f_1]\\ {}[(1-\vartheta)\cdot f_2] \end{bmatrix},&\quad
    \Ga{2}{1}^M([x]) &\defeq \begin{pmatrix} 0\\ {}[x] \end{pmatrix}.
  \end{alignat*}
  We have implicitly used the isomorphism
  \(\Z/p\cong\Z[\vartheta]/ (1-\vartheta)\) induced by the inclusion
  \(\Z \to \Z[\vartheta]\) and the inclusion
  \(\Z[\vartheta]/ (1-\vartheta) \subseteq Q\).  Notice that the
  maps \(\Ga{1}{0}^M\) and \(\Ga{2}{0}^M\) are well defined.  And
  the maps \(\Ga{1}{2}^M\) and~\(\Ga{2}{1}^M\) reverse the grading,
  while the other~\(\Ga{j}{m}^M\) preserve the grading.  We compute
  \[
  \Gt{0}^M = (1\oplus \vartheta \oplus\vartheta)_*,\qquad
  \Gs{1}^M = 1,\qquad
  \Gt{2}^M = \vartheta \oplus\vartheta,\qquad
  \Gs{2}^M = 1 \oplus 1.
  \]
  Now the relations \eqref{eq:alpha_101_N}, \eqref{eq:alpha_010_N}
  and~\eqref{eq:alpha_2_N} are clear.  And
  \begin{align*}
    \ker (\Ga{0}{1}^M)
    &= 0
      = \im (\Ga{1}{2}^M),\\
    \ker (\Ga{2}{1}^M)
    &= p \Z/p^k
      = \im (\Ga{1}{0}^M),\\
    \ker (\Ga{1}{2}^M)
    &= Q \oplus \Sigma Q
      = \im (\Ga{2}{0}^M),\\
    \ker (\Ga{0}{2}^M)
    &= 0 \oplus \Z[\vartheta]/(1-\vartheta)
      = \im (\Ga{2}{1}^M).
  \end{align*}
  Here we have used that multiplication by \(1-\vartheta\)
  on~\(Q\)
  is surjective and has the kernel \(\Z[\vartheta]/(1-\vartheta)\)
  (see Lemma~\ref{lem:S-module_surjective}).  The equivalence class of
  \(([x], [f_1], [f_2])\)
  belongs to the kernel of \(\Ga{1}{0}^M\)
  if and only if there is \(y\in \Z/p\) with \(x = p^{k-1} y\); then
  \[
  \begin{bmatrix} {}[x]\\ {}[f_1]\\ {}[f_2] \end{bmatrix}
  = \begin{bmatrix} 0\\ {}[f_1-y]\\ {}[f_2] \end{bmatrix}
  \in \im (\Ga{0}{2}^M).
  \]
  And the equivalence class of \(([x], [f_1], [f_2])\)
  belongs to the kernel of \(\Ga{2}{0}^M\)
  if and only if \([f_2]=0\)
  and \([f_1] \in \Z[\vartheta]/(1-\vartheta) \cong \Z/p\); then
  \[
  \begin{bmatrix} {}[x]\\ {}[f_1]\\ {}[f_2] \end{bmatrix}
  = \begin{bmatrix} {}[x-p^{k-1} f_1(1)]\\ 0\\ 0 \end{bmatrix}
  \in \im (\Ga{0}{1}^M).
  \]
  Thus~\(M\) is an exact \(\Kring\)\nb-module.
\end{example}

\begin{example}
  \label{exa:actions_on_Cuntz_8}
  Let \(k\ge1\),
  \(\ell\in\N_{\ge1}\) and \(c\in \{1,\dotsc,p-1\}\).  Define
  \begin{align*}
    M_0 &\defeq \Sring\bigm/ \bigl((1-t)^\ell - c p^{k-1} N(t), p^k N(t)\bigr),\\
    M_1 &\defeq \Z/p^k,\\
    M_2 &\defeq \Sring\bigm/ \bigl((1-t)^\ell, N(t)\bigr).
  \end{align*}
  Let~\(I\)
  be the ideal \(\bigl((1-t)^\ell - c p^{k-1} N(t), p^k N(t)\bigr)\)
  in~\(\Sring\).
  Let \(\Ga{2}{1}^M = 0\)
  and \(\Ga{1}{2}^M = 0\).
  Let \(\Ga{2}{0}^M \colon M_0 \to M_2\)
  be the quotient map; this exists because
  \(I \subseteq \bigl((1-t)^\ell, N(t)\bigr)\).
  Let \(\Ga{0}{2}^M \colon M_2 \to M_0\)
  be induced by multiplication with \(1-t\);
  this is well defined because
  \((1-t)\cdot \bigl((1-t)^\ell - c p^{k-1} N(t)\bigr) \equiv
  (1-t)^{\ell+1} \bmod t^p-1\),
  so that \((1-t)\cdot ((1-t)^\ell,N(t)) \subseteq I\)
  as ideals in~\(\Sring\).
  Let \(\Ga{1}{0}^M\colon M_0 \to M_1\)
  be induced by \(\mathrm{ev}_1\colon \Sring \to \Z\);
  this is well defined because \(N(1) = p\)
  and hence \(\mathrm{ev}_1(I) \subseteq (p^k)\).
  Let \(\Ga{0}{1}^M\colon M_1 \to M_0\)
  be induced by \(\Z \to \Sring\),
  \(j\mapsto N(t)\cdot j\);
  this is well defined because \(p^k N(t)\in I\).

  We compute that \(\Gt{0}^M = t\)
  and \(\Gt{2}^M = t\)
  are multiplication with~\(t\)
  on \(M_0\)
  and~\(M_2\),
  respectively, whereas \(\Gs{j}^M = 1\)
  for \(j=1,2\).
  The composite maps
  \(\Ga{0}{1}^M\circ \Ga{1}{0}^M\colon M_0 \to M_0\)
  and \(\Ga{1}{0}^M\circ \Ga{0}{1}^M\colon M_1 \to M_1\)
  are multiplication with \(N(t)\)
  and with~\(p\),
  respectively.  Hence the relations \eqref{eq:alpha_010_N},
  \eqref{eq:alpha_101_N} and~\eqref{eq:alpha_2_N} hold.

  The maps \(\Ga{2}{0}^M\)
  and~\(\Ga{1}{0}^M\)
  are surjective by construction.  And
  \(\im (\Ga{0}{1}^M) = \ker (\Ga{2}{0}^M)\)
  holds if and only if the map
  \(\coker (\Ga{0}{1}^M) \to M_2\)
  induced by~\(\Ga{2}{0}^M\)
  is injective.  Now
  \(M_2 = \Sring/((1-t)^\ell,N(t))\)
  and
  \(\coker (\Ga{0}{1}^M) = \Sring/(I + \Z\cdot N(t)) = \Sring/(N(t),
  (1-t)^\ell)\)
  as needed.  Similarly, we compute
  \[
  (1-t) + I
  = (1-t, c p^{k-1} N(t), p^k N(t))
  = (1-t, c p^k, p^{k+1})
  = (1-t,p^k)
  \]
  because \((c,p)=1\).  Hence the map from
  \(\coker (\Ga{0}{2}^M) = \Sring/((1-t) + I)\) to
  \(M_1 \cong \Sring/(1-t,p^k)\) induced by~\(\Ga{1}{0}^M\) is
  injective, so that \(\im (\Ga{0}{2}^M) = \ker (\Ga{1}{0}^M)\).  To
  check that~\(M\) is an exact \(\Kring\)\nb-module, it remains to
  show that \(\Ga{0}{2}^M\) and~\(\Ga{0}{1}^M\) are injective.
  These two claims translate to the following more elementary
  statements.  First, if \(f\in \Sring\) satisfies
  \((1-t) f \in I\), then \(f \in ((1-t)^\ell,N(t))\).  Secondly, if
  \(j\in \Z\) satisfies \(j N(t) \in I\), then \(p^k \mid j\).  We
  check these elementary claims.  Let \(f\in\Z[t]\).  First, if
  \((1-t) f \in I\), then there are \(f_1,f_2,f_3\in \Z[t]\) with
  \[
    (1-t) f = f_1\cdot (1-t^p) + f_2\cdot ((1-t)^\ell - c p^{k-1}
    N(t)) + f_3\cdot p^k N(t).
  \]
  It follows that \(p f_3(1) -c f_2(1)=0\).  Hence
  \(p \mid f_2(1)\).  Then \(f_2 \in (1-t,p)\) and \((1-t) f\)
  belongs to the ideal generated by
  \(1-t^p, (1-t)^{\ell+1}, p(1-t)^\ell, p^k N(t)\).  We may drop the
  generator \(p(1-t)^\ell\) because \(p \equiv N(t) \bmod (1-t)\)
  and \(1-t^p \mid N(t) \cdot (1-t)^\ell\); here we need
  \(\ell\ge1\).  So \(f \in ((1-t)^\ell,N(t))\) as asserted.  For
  the second claim, assume \(j N(t) \in I\).  Then there are
  \(f_1,f_2,f_3\in \Z[t]\) with
  \begin{equation}
    \label{eq:jN_equation}
    j N(t) = f_1 \cdot (1-t^p) + f_2\cdot ((1-t)^\ell - c p^{k-1}
    N(t)) + f_3\cdot p^k N(t).
  \end{equation}
  Hence \(N(t) \mid f_2\cdot (1-t)^\ell\).  Since multiplication
  by~\(1-t\) is injective on \(\Sring/(N(t)) \cong \Z[\vartheta]\),
  it follows that \(N \mid f_2\).  Hence \(p\mid f_2(1)\).
  Then~\eqref{eq:jN_equation} implies \(p^k \mid j\).
\end{example}

\begin{example}
  \label{exa:actions_on_Cuntz_5}
  Assume \(k\ge2\) or \(p\neq2\).  Define
  \begin{align*}
    M_0 &\defeq \Z/p^{k-1} \oplus \Sigma\Z[\vartheta] \oplus \Sigma Q,\\
    M_1 &\defeq \Z/p^k,\\
    M_2 &\defeq \Sigma\Z[\vartheta] \oplus \Sigma Q,
  \end{align*}
  \begin{alignat*}{3}
    \Ga{0}{1}^M([x])
    &= \begin{pmatrix} {}[x]\\0\\0 \end{pmatrix},&\quad
    \Ga{2}{0}^M \begin{pmatrix} {}[x]\\f_1\\ {}[f_2] \end{pmatrix}
    &= \begin{pmatrix} (1-\vartheta) f_1\\ {}[f_2] \end{pmatrix},&\quad
    \Ga{1}{2}^M \begin{pmatrix} f_1\\ {}[f_2] \end{pmatrix}
    &= [p^{k-1} f_1(1)],\\
    \Ga{1}{0}^M \begin{pmatrix} {}[x]\\f_1\\ {}[f_2] \end{pmatrix}
    &= [p\cdot x],&\quad
    \Ga{0}{2}^M \begin{pmatrix} f_1\\ {}[f_2] \end{pmatrix}
    &= \begin{pmatrix} 0\\ f_1\\ {}[(1-\vartheta)f_2] \end{pmatrix},&\quad
    \Ga{2}{1}^M([x]) &= \begin{pmatrix} 0\\ {}[x] \end{pmatrix}.
  \end{alignat*}
  The map~\(\Ga{1}{2}^M\)
  is well defined because
  \(\Z[t]/(N(t),t-1) = \Z[t]/(t-1,p) \cong \Z/p\)
  or, equivalently, \(\Z[\vartheta]/(1-\vartheta) \cong \Z/p\).
  (If \(k=1\),
  then the summand \(\Z/p^{k-1}=0\)
  disappears and the formulas for the maps~\(\Ga{j}{m}^M\)
  simplify accordingly.  In particular, \(\Ga{0}{1}^M\)
  and \(\Ga{1}{0}^M\) become~\(0\) in this case.)  We compute
  \[
  \Gt{0}^M = 1\oplus \vartheta \oplus\vartheta,\qquad
  \Gs{1}^M = 1,\qquad
  \Gt{2}^M = \vartheta \oplus\vartheta.
  \]
  The relations \eqref{eq:alpha_010_N} and~\eqref{eq:alpha_101_N} are
  clear, and~\eqref{eq:alpha_2_N} is equivalent to \(N(\Gs{2}^M)=p\)
  because \(N(\Gt{2}^M)=0\).
  If \(k\ge2\),
  then \(\Ga{2}{1}^M \circ \Ga{1}{2}^M = 0\).
  Hence \(\Gs{2}^M = 1 \oplus 1\)
  and \(N(\Gs{2}^M)=p\).
  If \(k=1\),
  then \(\Gu{2}^M - \Gs{2}^M = \Ga{2}{1}^M \circ \Ga{1}{2}^M\)
  maps \((f_1,[f_2]) \mapsto (0,[f_1(1)])\),
  where we use the canonical inclusion
  \(\Z/p \cong \Z[\vartheta]/(1-\vartheta) \subseteq Q\).
  Hence \((\Gu{2}^M - \Gs{2}^M)^2=0\).
  Since \(p\cdot (\Gu{2}^M - \Gs{2}^M)=0\)
  as well and \(p\neq2\),
  Lemma~\ref{lem:norm_if_unipotent} implies \(N(\Gs{2}^M) =p\)
  as needed.  The definitions imply
  \[
  \ker (\Ga{0}{1}^M) = p^{k-1} \Z/p^k = \im (\Ga{1}{2}^M),\qquad
  \ker (\Ga{2}{1}^M) = p \Z/p^k = \im (\Ga{1}{0}^M).
  \]
  Multiplication by~\(1-\vartheta\)
  is injective on \(\Sigma\Z[\vartheta]\)
  with cokernel~\(\Z/p\)
  and surjective on \(\Sigma Q\)
  with kernel~\(\Z/p\)
  by Lemmas \ref{lem:S-module_injective}
  and~\ref{lem:S-module_surjective}.  This implies
  \begin{align*}
    \ker (\Ga{1}{2}^M)
    &= (1-\vartheta) \Z[\vartheta] \oplus \Sigma Q
      = \im (\Ga{2}{0}^M),\\
    \ker (\Ga{0}{2}^M)
    &= 0\oplus \Z[\vartheta]/(1-\vartheta) \Z[\vartheta]
      = \im (\Ga{2}{1}^M),\\
    \ker (\Ga{1}{0}^M)
    &= 0 \oplus \Sigma\Z[\vartheta] \oplus \Sigma Q
      = \im (\Ga{0}{2}^M),\\
    \ker (\Ga{2}{0}^M)
    &= \Z/p^{k-1} \oplus 0 \oplus 0
      = \im (\Ga{0}{1}^M).
  \end{align*}
  It follows that~\(M\) is an exact \(\Kring\)\nb-module.
\end{example}

\begin{example}
  \label{exa:actions_on_Cuntz_7}
  Define
  \begin{align*}
    M_0 &\defeq \Sring/(p^{k-1} N(t)) \oplus \Sigma \Sring/(N(t)),\\
    M_1 &\defeq \Z/p^k,\\
    M_2 &\defeq \Sring/(N(t)) \oplus \Sigma \Sring/(N(t))
          = \Z[\vartheta] \oplus \Sigma\Z[\vartheta],
  \end{align*}
  \begin{alignat*}{3}
    \Ga{1}{0}^M \begin{pmatrix} {}[f_1]\\ {}[f_2] \end{pmatrix}
    &\defeq [f_1(1)],&\qquad
    \Ga{0}{2}^M \begin{pmatrix} {}[f_1]\\{}[f_2] \end{pmatrix}
    &\defeq \begin{pmatrix} {}[(1-t)f_1]\\{}[f_2]  \end{pmatrix},\\
    \Ga{0}{1}^M([x])
    &\defeq \begin{pmatrix} {}[x\cdot N(t)]\\0 \end{pmatrix},&\qquad
    \Ga{2}{0}^M \begin{pmatrix} {}[f_1]\\{}[f_2] \end{pmatrix}
    &\defeq \begin{pmatrix} {}[f_1]\\{}[(1-t)f_2] \end{pmatrix},\\
    \Ga{1}{2}^M \begin{pmatrix} {}[f_1]\\{}[f_2] \end{pmatrix}
    &\defeq [p^{k-1} f_2(1)],&\qquad
    \Ga{2}{1}^M([x]) &\defeq0.
  \end{alignat*}
  The map~\(\Ga{1}{0}^M\)
  is well defined because \(p^{k-1} N(1) = p^k\),
  and~\(\Ga{1}{2}^M\)
  is well defined because
  \(\Sring/(N(t),1-t) \cong \Z[\vartheta]/(1-\vartheta) \cong \Z/p\).
  We compute
  \[
  \Gt{0}^M = t \oplus t,\qquad
  \Gt{2}^M = t \oplus t,\qquad
  \Gs{1}^M = 1,\qquad
  \Gs{2}^M = 1 \oplus 1.
  \]
  So \(N(\Gt{0}^M) = N(t) \oplus N(t) = N(t) \oplus 0\)
  and \(N(\Gt{2}^M) = 0\).
  Now \(\Ga{0}{1}^M\circ \Ga{1}{0}^M = N(\Gt{0}^M)\)
  follows from \(N(t) \cdot f = N(t) \cdot f(1)\)
  for all \(f\in \Sring\),
  giving~\eqref{eq:alpha_010_N}.  The relations
  \eqref{eq:alpha_101_N} and~\eqref{eq:alpha_2_N} are clear.  The
  map~\(\Ga{1}{0}^M\)
  is clearly surjective and~\(\Ga{0}{2}^M\)
  is injective (see Lemma~\ref{lem:S-module_injective}).  Since the
  kernel of \(\Sring\to\Z\),
  \(f\mapsto f(1)\),
  is the ideal generated by~\(1-t\),
  it follows that \(\ker(\Ga{1}{0}^M) = \im(\Ga{0}{2}^M)\).
  Together with \(\Ga{2}{1}^M=0\),
  this is one of the exact sequences in
  Definition~\ref{def:exact_R-module}.  The isomorphism
  \(\Z[\vartheta]/(1-\vartheta) \cong \Z/p\) implies that
  \[
  \ker (\Ga{1}{2}^M)
  = \Z[\vartheta] \oplus \Sigma (1-\vartheta) \Z[\vartheta]
  = \im (\Ga{2}{0}^M).
  \]
  And \(\ker (\Ga{2}{0}^M) = (N(t))/(p^{k-1} N(t)) \oplus 0\)
  because multiplication by~\(1-t\)
  on \(\Sring/(N(t))\)
  is injective by Lemma~\ref{lem:S-module_injective}.  So
  \(\ker (\Ga{2}{0}^M) = \im (\Ga{0}{1}^M)\).
  And \(\ker (\Ga{0}{1}^M) = p^{k-1} \Z/p^k = \im (\Ga{1}{2}^M)\).
  So~\(M\) is an exact \(\Kring\)\nb-module.
\end{example}

\begin{example}
  \label{exa:actions_on_Cuntz_9}
  Let \(k\ge1\).
  Let \(\Z[\vartheta] \defeq \Sring/(N(t))\)
  and \(Q\defeq \Z[\vartheta,1/p] \bigm/ (1-\vartheta) \Z[\vartheta]\)
  as in~\eqref{eq:def_Q}.  Define
  \begin{align*}
    M_0 &\defeq \frac{\Sring/(p^k N(t)) \oplus Q}{\setgiven{(p^{k-1} N(t) j,j)}{j\in \Z/p}},\\
    M_1 &\defeq \Z/p^k \cong \Sring/(1-t,p^k) = \Sring/(1-t,p^{k-1} N(t)),\\
    M_2 &\defeq \Sring/(N(t)) \oplus Q = \Z[\vartheta] \oplus Q.
  \end{align*}
  Define \(\Ga{1}{2}^M \defeq 0\), \(\Ga{2}{1}^M \defeq 0\), and
  \begin{alignat*}{2}
    \Ga{0}{1}^M([x])
    &= \begin{bmatrix} {}[x\cdot N(t)]\\0 \end{bmatrix},&\qquad
    \Ga{2}{0}^M \begin{bmatrix} {}[f_1]\\{}[f_2] \end{bmatrix}
    &= \begin{pmatrix} {}[f_1]\\ {}[(1-t)\cdot f_2] \end{pmatrix},\\
    \Ga{1}{0}^M \begin{bmatrix} {}[f_1]\\ {}[f_2] \end{bmatrix}
    &= [f_1(1)],&\qquad
    \Ga{0}{2}^M \begin{pmatrix} {}[f_1]\\ {}[f_2] \end{pmatrix}
    &= \begin{bmatrix} {}[(1-t)\cdot f_1]\\ {}[f_2] \end{bmatrix}.
  \end{alignat*}
  The maps \(\Ga{1}{0}^M\)
  and~\(\Ga{2}{0}^M\)
  are well defined -- that is, they vanish on \((p^k N(t),0)\)
  and \((p^{k-1} N(t) j,j)\)
  for \(j\in\Z/p\)
  -- because \(p^{k-1} N(1) = p^k \equiv 0\)
  in~\(M_1\)
  and \((1-t) j = 0\)
  in~\(Q\)
  for all \(j\in\Z/p\).
  And the map~\(\Ga{0}{2}^M\)
  is well defined -- that is, it vanishes on \((N(t),0)\)
  -- because \((1-t) N(t) = 0\) in~\(\Sring\).  We compute
  \[
  \Gt{0}^M = (t \oplus\vartheta)_*,\qquad
  \Gs{1}^M = 1,\qquad
  \Gt{2}^M = t\oplus\vartheta,\qquad
  \Gs{2}^M = 1\oplus 1.
  \]
  This implies \(N(\Gt{2}^M)=N(\vartheta)=0\)
  and \(N(\Gs{2}^M) = p\),
  giving~\eqref{eq:alpha_2_N}.  And
  \(\Ga{1}{0}^M\circ \Ga{0}{1}^M[x] = [p\cdot x]\)
  and \(\Ga{0}{1}^M\circ \Ga{1}{0}^M[f_1,f_2] = [N(t)\cdot f_1,0]\);
  so \eqref{eq:alpha_101_N} and~\eqref{eq:alpha_010_N} also hold.

  The map~\(\Ga{0}{1}^M\)
  is injective.  Its cokernel is the quotient of \(\Sring\oplus Q\)
  by \((N(t)) \oplus \Z/p\)
  because elements of the form \((x\cdot N(t),0)\)
  are killed in it, and \([0,j] = [-p^{k-1} j N(t),0]\)
  holds in~\(M_0\).
  The map~\(\Ga{2}{0}^M\)
  induces an isomorphism from this to \(\Z[\vartheta]\oplus Q\)
  because of Lemmas \ref{lem:S-module_injective}
  and~\ref{lem:S-module_surjective}.  So the maps \(\Ga{0}{1}^M\)
  and~\(\Ga{2}{0}^M\)
  form a short exact sequence.  The map~\(\Ga{0}{2}^M\)
  is injective because \((1-t) \cdot \Sring \cap N(t) \cdot \Sring = 0\).
  Its cokernel is the quotient of \(\Sring\oplus Q\)
  by the subgroup generated by \((1-t) \Sring \oplus Q\),
  \((p^k N(t)) \oplus 0\)
  and \(\setgiven{(p^{k-1} N(t) j,j)}{j\in \Z/p}\).
  All of~\(Q\)
  is killed in this cokernel, and evaluation at~\(1\)
  is an isomorphism \(\Sring/(1-t) \cong \Z\),
  which maps \(N(t)\mapsto p\).
  So the cokernel of~\(\Ga{0}{2}^M\)
  is isomorphic to~\(\Z/p^k\)
  through the map~\(\Ga{1}{0}^M\).
  Thus both sequences in Definition~\ref{def:exact_R-module} are
  exact.  This verifies the relation~\eqref{eq:alpha_exact} and shows
  that the data above defines an exact \(\Kring\)\nb-module.
\end{example}

\begin{example}
  \label{exa:actions_on_Cuntz_10}
  Let \(p=2\).
  Then the ring~\(\Kring\)
  has an automorphism that maps
  \(\Ga{j}{m}\mapsto \Ga{\sigma(j)}{\sigma(m)}\)
  for the permutation \(\sigma = (02)\).
  It also switches the \(\Z/2\)\nb-grading
  on~\(M_2\).
  And it maps \(\Gs{1}\)
  to~\(-\Gs{1}\).
  Hence each example above produces another example where the sign
  of~\(\Gs{1}\)
  is switched, by permuting the summands~\(M_j\)
  and the maps \(\Ga{j}{m}^M\)
  according to the permutation~\(\sigma\).
  We denote it by~\(M^\sigma\).

  We claim that~\(M^\sigma\)
  is truly a new example except in Examples
  \ref{exa:actions_on_Cuntz_6} and~\ref{exa:actions_on_Cuntz_7} in the
  special case \(k=1\).
  The only example above where \(\Gs{1}^M\neq\Gu{1}^M\)
  is Example~\ref{exa:actions_on_Cuntz_3}, where
  \(\Gs{1}^M = 1-\tau p^{k-1}\).
  Since \(k\ge3\)
  is assumed if \(p=2\),
  it follows that \(1-\tau p^{k-1} \equiv 1 \bmod 4\)
  and hence \(-(1-\tau p^{k-1}) \equiv 3 \bmod 4\).
  So the automorphism applied to Example~\ref{exa:actions_on_Cuntz_3}
  gives new examples.  If \(\Gs{1}^M=\Gu{1}^M\),
  then \(-\Gs{1}^M=\Gs{1}^M\)
  if and only if multiplication by~\(2\)
  is the identity map on~\(M_1\).
  This only happens if \(M_1=\Z/2\).
  This is not possible in Examples \ref{exa:actions_on_Cuntz_1},
  \ref{exa:actions_on_Cuntz_2}, \ref{exa:actions_on_Cuntz_4},
  \ref{exa:actions_on_Cuntz_5}, so they always give new examples.  The
  other examples Examples \ref{exa:actions_on_Cuntz_6},
  \ref{exa:actions_on_Cuntz_8}, \ref{exa:actions_on_Cuntz_7}
  and~\ref{exa:actions_on_Cuntz_9} also give new examples if \(p=2\)
  and \(k\ge2\).
  We examine what happens in the latter examples if \(p=2\)
  and \(k=1\).
  It can be seen that \(M \cong M^\sigma\)
  in the situation of Examples \ref{exa:actions_on_Cuntz_6}
  and~\ref{exa:actions_on_Cuntz_7} if \(k=1\)
  and \(p=2\).
  In contrast, we claim that Examples \ref{exa:actions_on_Cuntz_8}
  and~\ref{exa:actions_on_Cuntz_9} give new examples also for \(k=1\)
  and \(p=2\).
  Since \(\Ga{2}{1}^M = 0\)
  and \(\Ga{1}{2}^M = 0\)
  in these two examples, \(\Ga{0}{1}^{M^\sigma} = 0\)
  and \(\Ga{1}{0}^{M^\sigma} = 0\).
  No other example with \(k=1\)
  and \(p=2\)
  has this property.  As a result, \(M^\sigma\)
  is indeed a new example except for Examples
  \ref{exa:actions_on_Cuntz_6} and~\ref{exa:actions_on_Cuntz_7} with
  \(k=1\).
\end{example}

\section{Proof of Theorem~\ref{the:actions_on_Cuntz}}
\label{sec:prove_theorem}

Let~\(M\)
be an arbitrary exact \(\Kring\)\nb-module
where~\(M_1\)
is a cyclic group.  First, we study the possibilities for the action
of~\(\Z/p\) on~\(M_1\).

We begin with the case \(M_1=\Z\).
The automorphism group of~\(M_1\)
has two elements, and the non-trivial one acts by \(x\mapsto -x\).
This automorphism of order~\(2\)
only occurs if \(p=2\).
In the case \(p=2\),
there is an automorphism of the ring~\(\Kring\)
that maps~\(\Gs{1}\)
to~\(-\Gs{1}\).
Since Example~\ref{exa:actions_on_Cuntz_10} takes this automorphism
into account, Theorem~\ref{the:actions_on_Cuntz} holds in the case
\(M_1=\Z\)
once it holds in the case where both \(M_1=\Z\)
and \(\Gs{1}^M = \Gu{1}^M\).
Then \(N(\Gs{1}^M) = p\).  So
\begin{equation}
  \label{eq:ker_im_Ns1_infinite}
  \ker N(\Gs{1}^\Z) = 0,\qquad
  \im N(\Gs{1}^\Z) = p\cdot \Z,\qquad
  \coker N(\Gs{1}^\Z) \cong \Z/p.
\end{equation}

Next, let~\(M_1\)
be a finite cyclic group.  Then \(M_1 \cong \Z/p^k \oplus M_1'\)
with a cyclic group~\(M_1'\)
of order coprime to~\(p\).
So~\(M_1'\)
is uniquely \(p\)\nb-divisible.
Theorem~\ref{the:remove_divisible} extends~\(M'_1\)
to an exact \(\Kring\)\nb-submodule~\(M'\)
of~\(M\),
which has the form of Example~\ref{exa:divisible} with \(Y=0\).
The quotient \(M'' \defeq M/M'\)
is an exact \(\Kring\)\nb-module with \(M_1'' = \Z/p^k\).

\begin{lemma}
  \label{lem:non-p-part_summand}
  The submodule~\(M'\)
  is a direct summand, so that \(M = M'' \oplus M'\).
\end{lemma}

\begin{proof}
  The projection~\(e_1\)
  onto~\(M_1'\)
  that kills~\(\Z/p^k\)
  commutes with~\(\Gs{1}^M\).
  Since~\(M'_1\)
  is uniquely \(p\)\nb-divisible,
  it splits as
  \(M_1' = N(\Gs{1}^M) (M_1') \oplus (1-\Gs{1}^M)(M_1')\),
  and \(1-\Gs{1}^M\)
  is invertible on \((1-\Gs{1}^M)(M_1')\)
  (see Proposition~\ref{pro:divisible_S-module}).  Let \(e_{1N}\)
  and~\(e_{1-}\)
  be the projections onto \(N(\Gs{1}^M) (M_1')\)
  and \((1-\Gs{1}^M)(M_1')\), respectively.  Now the maps
  \begin{alignat*}{2}
    e_0 &\colon M_0 \to M_0,&\qquad
    e_0 &\defeq \Ga{0}{1}^M \circ p^{-1} e_{1 N}\circ \Ga{1}{0}^M,\\
    e_2 &\colon M_2 \to M_2,&\qquad
    e_2 &\defeq \Ga{2}{1}^M \circ (1-t)^{-1} e_{1-} \circ \Ga{1}{2}^M
  \end{alignat*}
  are well defined.  Equations \eqref{eq:alpha_101_N} and
  \eqref{eq:alpha_121} imply \(e_0^2 = e_0\),
  \(e_2^2 = e_2\),
  and their images are easily seen to be \(M'_0\)
  and~\(M'_2\),
  respectively.  The idempotent maps~\(e_j\)
  for \(j=0,1,2\)
  form a \(\Kring\)\nb-module
  map.  So they give the asserted direct sum decomposition.
\end{proof}

As a result, if Theorem~\ref{the:actions_on_Cuntz} holds in the case
where \(M_1 = \Z/p^k\)
for some \(k\in\N_{\ge1}\),
then it holds whenever \(M_1 = \Z/n\)
with \(n\in\N_{\ge1}\).
So in the case of finite~\(M_1\),
we may assume without loss of generality that \(M_1 = \Z/p^k\).

Now we consider the possible representations of~\(\Z/p\)
on \(M_1 = \Z/p^k\).
Any automorphism of~\(\Z/p^k\)
is of the form \(x\mapsto x\cdot m\)
for a unit~\(m\)
in the ring~\(\Z/p^k\).
To get an action of~\(\Z/p\),
we need \(m^p \equiv 1 \bmod p\).
This forces \(m\equiv 1\bmod p\).
So only the trivial automorphism is possible if \(k=1\).
Assume now that \(k\ge2\).
If \(p\neq 2\),
then the multiplicative group of elements in \(1+ p\Z \bmod p^k\)
is isomorphic to the additive group \(p\Z/p^k\cong \Z/p^{k-1}\)
because the power series defining the exponential function and the
logarithm converge.  Hence \(m^p=1 \bmod p\)
holds if and only if \(m = 1 - p^{k-1}\tau\)
for some \(\tau\in\Z/p\).
If \(p=2\),
then the argument above only works if \(m\equiv 1\bmod 4\).
It shows that the possible solutions of \(m^p=1\bmod 2^k\)
with \(m\equiv 1\bmod 4\)
are \(m = 1 - 2^{k-1}\tau\)
for \(\tau\in\Z/2\),
and \(\tau\neq0\)
only if \(k\ge3\).
So all solutions in~\(\Z/2^k\)
are \(\pm (1 - 2^{k-1}\tau)\)
for the two signs and \(\tau\in\Z/2\),
and \(\tau\neq0\)
is only allowed if \(k\ge3\).
An automorphism of~\(\Kring\)
that switches the sign of~\(\Gs{1}\)
is described in Section~\ref{sec:prime_2}.  Hence the cases
\(m = - (1 - 2^{k-1}\tau)\)
are covered by Example~\ref{exa:actions_on_Cuntz_10}.
Therefore, if \(M_1=\Z/p^k\),
then we may assume without loss of generality that~\(\Gs{1}^M\)
is multiplication with \(m = 1 - p^{k-1} \tau\)
for some \(\tau\in\Z/p\),
where \(\tau=0\)
if \(k=1\)
or if \(k=2\)
and \(p=2\).
Then we compute \((\Gs{1}^M)^j(x) = m^j\cdot x = (1-p^{k-1}j\tau)x\)
for \(j\in\N\).
If \(p\neq 2\),
then \(N(\Gs{1}^M)(x) = p\cdot x\)
because \(p \mid \sum_{j=1}^{p-1} j\).
If \(p=2\),
then \(N(\Gs{1}^M)(x) = 2\cdot (1- 2^{k-2}\tau)\cdot x\).
Since \(1- 2^{k-2}\tau\)
for \(k\ge 3\) is a unit in~\(\Z/2^k\), this implies
\begin{equation}
  \label{eq:ker_im_Ns1_finite}
  \ker N(\Gs{1}^{\Z/p^k}) = p^{k-1}\Z/p^k \cong \Z/p,\qquad
  \im N(\Gs{1}^{\Z/p^k}) = p\Z/p^k.
\end{equation}
Hence \(\coker N(\Gs{1}^{\Z/p^k}) \cong \Z/p\).
Lemma~\ref{lem:short_exact_sequences_consequences} implies
\begin{equation}
  \label{eq:im_ker_alpha_10_N1s1}
  \ker \Ga{0}{1}^M \subseteq \ker N(\Gs{1}^M),\quad
  \im (\Ga{1}{0}^M) \supseteq \im N(\Gs{1}^M),\quad
  \coker (\Ga{1}{0}^M) \subseteq \coker N(\Gs{1}^M);
\end{equation}
the third claim in~\eqref{eq:im_ker_alpha_10_N1s1} is equivalent to
the second.  Since~\(M\)
is exact, the map~\(\Ga{2}{1}^M\)
induces an isomorphism \(\coker \Ga{1}{0}^M \cong \im \Ga{2}{1}^M\).
Equations \eqref{eq:ker_im_Ns1_infinite}
and~\eqref{eq:ker_im_Ns1_finite} show that \(\ker N(\Gs{1}^M)\)
and \(\coker N(\Gs{1}^M)\)
are both~\(\Z/p\).
Since the only subgroups of~\(\Z/p\)
are \(\{0\}\)
and~\(\Z/p\), our classification problem splits into four cases:
\begin{enumerate}
\item \(\ker \Ga{0}{1}^M = \{0\}\)
  and \(\coker \Ga{1}{0}^M \cong \im \Ga{2}{1}^M= \{0\}\);
\item \(\ker \Ga{0}{1}^M = \{0\}\)
  and \(\coker \Ga{1}{0}^M \cong \im \Ga{2}{1}^M \cong \Z/p\);
\item \(\ker \Ga{0}{1}^M \cong \Z/p\)
  and \(\coker \Ga{1}{0}^M \cong \im \Ga{2}{1}^M = \{0\}\);
\item \(\ker \Ga{0}{1}^M \cong \Z/p\)
  and \(\coker \Ga{1}{0}^M \cong \im \Ga{2}{1}^M \cong \Z/p\);
\end{enumerate}

We split the proof of Theorem~\ref{the:actions_on_Cuntz} into several
propositions.  The first one handles exact \(\Kring\)\nb-modules
with \(M_1=\Z\)
(Proposition~\ref{pro:O-infty}).  Then we treat exact
\(\Kring\)\nb-modules
with \(M_1=\Z/p^k\)
in the first three cases above in Propositions \ref{pro:case1},
\ref{pro:case2} and~\ref{pro:case3}, respectively.  The fourth case is
split further into Propositions \ref{pro:case4a}
and~\ref{pro:case4b}.  These propositions together will prove
Theorem~\ref{the:actions_on_Cuntz}.

\begin{proposition}
  \label{pro:O-infty}
  Let~\(M\)
  be an exact \(\Kring\)\nb-module
  with \(M_1 = \Z\)
  and \(\Gs{1}^M = \Gu{1}^M\).
  Then \(\coker \Ga{1}{0}^M = \{0\}\)
  or \(\coker \Ga{1}{0}^M = \Z/p\).
  In the first case, let~\(M'\)
  be the exact \(\Kring\)\nb-module
  in Example~\textup{\ref{exa:actions_on_Cuntz_1}}.  In the second
  case, let~\(M'\)
  be the exact \(\Kring\)\nb-module
  in Example~\textup{\ref{exa:actions_on_Cuntz_2}}.  There is a
  \(\Kring\)\nb-module
  extension \(M' \into M \prto M''\)
  where~\(M''\)
  is a uniquely \(p\)\nb-divisible
  exact \(\Kring\)\nb-module
  as in Example~\textup{\ref{exa:divisible}} with \(X = 0\)
  and \(Z=0\).
\end{proposition}

\begin{proof}
  The assumption \(\Gs{1}^M=\Gu{1}^M\)
  implies \(N(\Gs{1}^M) = p\).
  So \(\ker N(\Gs{1}^M) = 0\)
  and \(\coker N(\Gs{1}^M) \cong \Z/p\).
  Hence \(\ker \Ga{0}{1}^M = \{0\}\)
  and either \(\coker \Ga{1}{0}^M = \{0\}\)
  or \(\coker \Ga{1}{0}^M = \Z/p\).
  Thus~\(\Ga{0}{1}^M\)
  is injective, \(\Ga{1}{2}^M=0\)
  and~\(\Ga{2}{0}^M\)
  is surjective because~\(M\)
  is exact.  And the image of~\(\Ga{2}{1}^M\)
  is isomorphic to either \(0\)
  or~\(\Sigma\Z/p\);
  the suspension comes in because~\(\Ga{2}{1}^M\)
  reverses the grading.

  Exact \(\Kring\)\nb-modules
  with \(\Ga{1}{2}^M=0\)
  are studied in Section~\ref{sec:a21_zero}.
  Theorem~\ref{the:a21_vanishes} shows that their isomorphism classes
  are in bijection with isomorphism classes of triples
  \((M_0,\dot{M}_1,\dot\alpha)\),
  where~\(M_0\)
  is an \(\Sring\)\nb-module
  with \(\ker N(t) = \im (1-t)\),
  \(\dot{M}_1 \subseteq H^0(M_0) \defeq \ker (1-t)/\im N(t)\)
  is a \(\Z/p\)\nb-vector
  subspace and~\(\dot\alpha\)
  is a grading-preserving isomorphism
  \(\dot\alpha\colon \dot{M}_1 \congto H^0(M_0)/\dot{M}_1\).
  Equation~\eqref{eq:a12_vanishes_image_a21} shows that
  \(\dot{M}_1 \cong \im(\Ga{2}{1}^M)\).
  The two cases for \(\coker \Ga{1}{0}^M\)
  correspond to triples with
  \begin{enumerate}
  \item \label{en:M1_Z_case_1}%
    \(\dot{M}_1=0\);
  \item \label{en:M1_Z_case_2}%
    \(\dot{M}_1\cong\Sigma\Z/p\).
  \end{enumerate}
  The existence of a grading-reversing isomorphism
  \(\dot\alpha\colon \dot{M}_1 \congto H^0(M_0)/\dot{M}_1\)
  implies \(\ker (1-t) = \im N(t)\)
  in case~\ref{en:M1_Z_case_1} and
  \(\ker (1-t)/ \im N(t) \cong \Z/p \oplus \Sigma \Z/p\)
  in case~\ref{en:M1_Z_case_2}.  By construction,
  \(M_1 \subseteq M_0\)
  is an extension of \(\im N(t)\)
  by~\(\dot{M}_1\).  So \(\im N(t) \cong \Z\) holds in both cases.

  Let~\(M'\)
  be the exact \(\Kring\)\nb-module
  in Example~\ref{exa:actions_on_Cuntz_1} if \(\dot{M}_1=0\)
  and the exact \(\Kring\)\nb-module
  in Example~\ref{exa:actions_on_Cuntz_2} if
  \(\dot{M}_1\cong\Sigma\Z/p\).
  Notice that \(M'_1=\Z\),
  \(\Gs{1}^{M'}=\Gu{1}^{M'}\)
  in both examples, so that~\(M'\)
  also corresponds to a triple \((M_0',\dot{M}_1', \dot\alpha')\).
  And \(\dot{M}_1' \cong \dot{M}_1\)
  holds in both cases.  We are going to define a morphism of triples
  \(\psi\colon M_0' \to M_0\)
  that restricts to an isomorphism \(M_1' \to M_1\)
  and is injective as a map \(M_0' \to M_0\).
  The criterion in Theorem~\ref{the:a21_vanishes} shows that this
  map~\(\psi\)
  induces an injective \(\Kring\)\nb-module
  homomorphism \(\psi_*\colon M' \hookrightarrow M\).
  Let \(M''\defeq M/M'\).
  This is an exact \(\Kring\)\nb-module
  as well because both \(M'\)
  and~\(M\)
  are exact.  And \(M''_1=0\).
  Hence Theorem~\ref{the:divisible_exact_module} shows that~\(M''\)
  is uniquely \(p\)\nb-divisible
  and has the form of Example~\ref{exa:divisible} with \(X=0\)
  and \(Z=0\).
  It remains to build the homomorphism of triples~\(\psi\) above.

  We first do this in the case \(\dot{M}_1=0\).
  Then \(\ker (1-t) = \im N(t)\)
  and \(\ker N(t) = \im (1-t)\).
  That is, \(M_0\)
  is cohomologically trivial, and we are in the situation of
  Corollary~\ref{cor:a_2112_exact}.  And
  \begin{align*}
    M_1 &= \ker (1-t) = \im N(t) \cong \Z,\\
    M_2 &= M_0/M_1 = M_0/\ker(1-t) \cong \im (1-t) = \ker N(t),
  \end{align*}
  where the isomorphism \(M_0/\ker(1-t) \cong \im (1-t)\)
  is induced by~\(1-t\).

  \begin{lemma}
    \label{lem:O-infty_coh_trivial_M2}
    Let \(a\in \im N(t) \cong \Z\)
    be a generator and choose \(y\in M_0\)
    with \(N(t)y = a\).
    Then \((1-t)y\)
    is a generator of\/ \(\im(1-t)/\im (1-t)^2\),
    which is isomorphic to \(\Z/p\).
    And \(\ker(1-t) \cap \im(1-t)=0\).
  \end{lemma}

  \begin{proof}
    The map~\(N(t)\)
    acts by multiplication by~\(p\)
    on \(\im N(t)\)
    by Lemma~\ref{lem:S-module_inclusions}.  This is injective because
    \(\im N(t) \cong \Z\).
    So \(\ker N(t) \cap \im N(t) = 0\).
    This is equivalent to \(\ker (1-t) \cap \im (1-t) = 0\)
    by cohomological triviality.  Equivalently, \(1-t\)
    restricts to an injective map on \(\im(1-t)\).
    Let \(a\in \im N(t) \cong \Z\)
    be a generator.  Let \(x\in \im (1-t)\).
    This means that there is \(z\in M_0\)
    with \(x=(1-t)z\).
    This element is unique up to adding an element of
    \(\ker (1-t) = \im N(t) \cong \Z\).
    We have \(x\in \im(1-t)^2\)
    if and only if~\(z\)
    may be chosen in \(\im (1-t) = \ker N(t)\).
    Equivalently, there is \(m\in\Z\)
    so that \(N(t)(z+m\cdot a)=0\).
    Now \(N(t)(m\cdot a) = p\cdot m\cdot a\)
    because \(t\cdot a=a\).
    So there is \(m\in\Z\)
    with \(N(t)(z+m\cdot a)=0\)
    if and only if \(N(t)(z)\)
    belongs to the subgroup generated by \(p\cdot a\).
    Hence there is a well defined injective map from
    \(\im (1-t)/\im (1-t)^2\)
    to~\(\Z/p\)
    that maps the class of~\(x\)
    to the class of~\(N(t)z\)
    in \(a\Z/pa\Z\).
    Since \(a\in \im N(t)\),
    there is \(y\in M_0\)
    with \(N(t)y = a\).
    Then the map \(\im (1-t)/\im (1-t)^2\to\Z/p\)
    maps \((1-t)y\in \im(1-t)\)
    to a non-trivial element.  Hence this map is also surjective.
  \end{proof}

  \begin{lemma}
    \label{lem:O-infty_coh_trivial_final}
    The map \(\psi\colon \Sring \to M_0\),
    \(f\mapsto f(t)y\),
    is an injective \(\Sring\)\nb-module
    homomorphism.  Its image contains~\(M_1\)
    and its cokernel is uniquely \(p\)\nb-divisible.
  \end{lemma}

  \begin{proof}
    The image of~\(\psi\)
    contains \(\Z\cdot a\),
    which is~\(M_1\)
    by the choice of~\(a\).
    So~\(\psi\)
    has the same cokernel as the map \(\Z[\vartheta] \to \im(1-t)\),
    \(f\mapsto f(t) (1-t)y\).
    This map is injective with a uniquely \(p\)\nb-divisible
    cokernel by Lemmas \ref{lem:S-module_injective}
    and~\ref{lem:O-infty_coh_trivial_M2}.  It remains to show
    that~\(\psi\)
    is injective as well.  The principal ideal in
    \(\Sring\defeq \Z[t]/(t^p-1)\)
    generated by~\(t-1\)
    is isomorphic to \(\Z[\vartheta]\)
    because \((t^p-1) \mid f\cdot (t-1)\)
    if and only if \(N(t) \mid f\).
    The map~\(\psi\)
    restricts to an injective map from \((t-1)\cdot \Sring\)
    to \(\im (1-t) \subseteq M_0\)
    by Lemma~\ref{lem:S-module_injective}.  Let \(\psi(f)=0\).
    Then \(\psi((t-1)\cdot f)=0\)
    as well.  This implies \(N(t) \mid f(t)\)
    because~\(\psi\)
    is injective on \((t-1)\cdot \Sring \cong \Sring/(N(t))\).
    Then \(f\equiv m\cdot N(t) \bmod (t^p-1)\)
    for some \(m\in\Z\).
    Now \(\psi(f) = m\cdot N(t) y = m\cdot a\),
    and this is only~\(0\)
    if \(m=0\).
    Then \(t^p-1 \mid f\),
    that is, \(f=0\) in~\(\Sring\).  So~\(\psi\) is injective.
  \end{proof}

  The exact \(\Kring\)\nb-module~\(M'\)
  in Example~\ref{exa:actions_on_Cuntz_1} also satisfies
  \(\Ga{2}{1}^{M'} = 0\)
  and \(\Ga{1}{2}^{M'} = 0\).
  So Corollary~\ref{cor:a_2112_exact} shows that it is associated to a
  cohomologically trivial \(\Sring\)\nb-module,
  namely, \(M_0' = \Sring\).
  So the map~\(\psi\)
  above induces a \(\Kring\)\nb-module
  homomorphism \(M' \to M\).
  It is injective by the criterion in
  Corollary~\ref{cor:a_2112_exact}.  This finishes the proof in the
  case of exact \(\Kring\)\nb-modules
  with \(M_1=\Z\), \(\Ga{2}{1}^{M'} = 0\) and \(\Ga{1}{2}^{M'} = 0\).

  Now we treat the case where \(\dot{M}_1 \cong \Sigma\Z/p\).
  Since \(N(t)|_{M_1}\)
  is multiplication by~\(p\),
  the maps \(\Ga{0}{1}^M\)
  and~\(N(\Gs{1}^M)\)
  are both injective and have the same image, namely, the subgroup
  \(p\cdot \Z\subseteq \Z\).
  Hence the statement in
  Lemma~\ref{lem:decomposable_case_criterion_0}.%
  \ref{lem:decomposable_case_criterion2} holds.  So the equivalent
  statements \ref{lem:decomposable_case_criterion1}
  and~\ref{lem:decomposable_case_criterion3} in
  Lemma~\ref{lem:decomposable_case_criterion_0} hold as well.  Thus
  \begin{align*}
    \ker (\Gu{2}^M - \Gt{2}^M)
    &= \ker (\Ga{0}{2}^M)
      = \im (\Ga{2}{1}^M)
      \cong \Z/p,\\
    \im (\Gu{2}^M - \Gt{2}^M)
    &= \im (\Ga{2}{0}^M)
      = \ker (\Ga{1}{2}^M)
      = M_2,\\
    \shortintertext{and}
    M_0
    &= M_1 \oplus \im (\Ga{0}{2}^M)
      = M_1 \oplus \im (1-t)
  \end{align*}
  as an \(\Sring\)\nb-module;
  here we have used~\eqref{eq:alpha_020} and that~\(\Ga{2}{0}^M\)
  is surjective.  Lemma~\ref{lem:S-module_surjective} applies and
  gives an injective \(\Sring\)\nb-module
  homomorphism \(M_2' \hookrightarrow M_2\)
  with \(M_2'=\Sigma Q\)
  as in Example~\ref{exa:actions_on_Cuntz_2}, such that \(M_2/M_2'\)
  is uniquely \(p\)\nb-divisible.
  The map \(\Ga{0}{2}^M\colon \Sigma Q \to M_0\)
  has kernel~\(\Sigma\Z/p\)
  and \(Q/(\Z/p) \cong Q\)
  by multiplication with \(1-t\).
  Hence \(M_0 \cong M_1 \oplus \Sigma Q\)
  as an \(\Sring\)\nb-module,
  and~\(\dot{M}_1\)
  is the image of~\(M_1\)
  in \(H^0(M_0)\).
  We may choose the generator \(e\in \Z/p\)
  in the proof of Lemma~\ref{lem:S-module_surjective} to
  be~\(\dot\alpha([1])\).
  Then the injective homomorphism
  \(M_0' \defeq \Z\oplus \Sigma Q\hookrightarrow M_0\)
  is a morphism of triples as needed for
  Theorem~\ref{the:a21_vanishes}.  It is injective on~\(M_0'\)
  and induces an isomorphism \(M_1' \congto M_1\).
  Hence the induced homomorphism of exact \(\Kring\)\nb-modules
  is injective by Theorem~\ref{the:a21_vanishes}.  Its
  cokernel~\(M''\)
  is an exact \(\Kring\)\nb-module
  with \(M_1'' = 0\).
  This finishes the proof in the case of exact \(\Kring\)\nb-modules
  with \(M_1=\Z\),
  \(\Ga{2}{1}^{M'} = 0\)
  and \(\im (\Ga{2}{1}^M) \cong \Z/p\).
  So Proposition~\ref{pro:O-infty} is proven.
\end{proof}

We have proven Theorem~\ref{the:actions_on_Cuntz} in the case
\(M_1=\Z\).
So from now on, we may assume that \(M_1 = \Z/p^k\)
for some \(k\ge1\).
If \(\Gs{1}^M \neq 1\),
then \(\Ga{1}{2}^M\neq0\)
and \(\Ga{2}{1}^M\neq0\)
by~\eqref{eq:def_st}.  We shall first treat the three cases where at
least one of the maps \(\Ga{1}{2}^M\)
and \(\Ga{2}{1}^M\)
is zero.  So we shall only meet exact \(\Kring\)\nb-modules
with \(\Gs{1}^M = 1\) for a while.

The first case in our list has \(\ker \Ga{0}{1}^M = \{0\}\)
and \(\coker \Ga{1}{0}^M = \{0\}\).
The study of this case is parallel to the case \(\dot{M}_1=0\)
in the proof of Proposition~\ref{pro:O-infty}.  Again,
\(\Ga{1}{2}^M=0\)
and \(\Ga{2}{1}^M=0\)
hold, so that we are in the situation of
Corollary~\ref{cor:a_2112_exact}.  Exact \(\Kring\)\nb-modules
in this case are equivalent to cohomologically trivial
\(\Sring\)\nb-modules with
\begin{align*}
  M_1 &= \ker(1-t) = \im N(t) = \Z/p^k,\\
  M_2 &= M_0/M_1 = M_0/\ker(1-t) \cong \im (1-t) = \ker N(t).
\end{align*}

\begin{proposition}
  \label{pro:case1}
  Let~\(M\)
  be an exact \(\Kring\)\nb-module
  with \(M_1=\Z/p^k\)
  and \(\ker \Ga{0}{1}^M = \{0\}\)
  and \(\coker \Ga{1}{0}^M = \{0\}\).
  Then there is a \(\Kring\)\nb-module
  extension \(M' \into M \prto M''\),
  where~\(M'\)
  is one of the exact \(\Kring\)\nb-modules
  in Example \textup{\ref{exa:actions_on_Cuntz_8}}
  or~\textup{\ref{exa:actions_on_Cuntz_9}} and~\(M''\)
  is a uniquely \(p\)\nb-divisible
  exact \(\Kring\)\nb-module
  as in Example~\textup{\ref{exa:divisible}} with \(X = 0\)
  and \(Z=0\).
\end{proposition}

\begin{proof}
  We begin by studying~\(M_2\),
  which we identify with \(\ker N(t) = \im (1-t)\)
  as above.  We compute the kernel and cokernel of \(1-\vartheta\)
  acting on~\(M_2\).
  We use the Snake Lemma for the endomorphism \(1-t\)
  of the extension \(M_2 \into M_0 \prto \im N(t)\),
  where the quotient map is~\(N(t)\).
  Using that~\(M_0\)
  is cohomologically trivial, this gives a long exact sequence
  \begin{multline*}
    0\to \ker (1-t)|_{M_2} \to \im N(t) \xrightarrow{N(t)} \im N(t)
    \\\to \coker (1-t)|_{M_2} \xrightarrow{N(t)} \im N(t) \xrightarrow{\id} \im N(t) \to 0.
  \end{multline*}
  Since \(N(t)\colon \im N(t) \to \im N(t)\)
  is multiplication by~\(p\) on~\(\Z/p^k\), this implies
  \begin{equation}
    \label{eq:case1_ker_coker_theta}
    \ker (1-t)|_{M_2}\cong \Z/p,\qquad
    \coker (1-t)|_{M_2}\cong \Z/p.
  \end{equation}
  More precisely, \(\ker (1-t)|_{M_2} \cong \Z/p\)
  is generated by the class of~\(p^{k-1}\)
  in \(\Z/p^k = M_1 \cong \im N(t)\).
  The boundary map in the Snake Lemma maps the generator~\(1\)
  of~\(\Z/p^k\)
  to a generator of \(\coker (1-t)|_{M_2}\).
  Choose \(b\in M_0\)
  with \(N(t)b = [1]\)
  in~\(\Z/p^k\).
  Then \((1-t)b \in \im (1-t) = M_2\)
  represents the generator of \(\coker (1-t)|_{M_2}\).
  Next we classify the resulting modules over~\(\Z[\vartheta]\):

  \begin{lemma}
    \label{lem:Ztheta_p_ker-coker_1}
    Let~\(M_2\)
    be a \(\Z[\vartheta]\)-module
    with~\eqref{eq:case1_ker_coker_theta}.  Then there is a submodule
    \(M_2'\subseteq M_2\)
    such that \(M_2/M_2'\)
    is uniquely \(p\)\nb-divisible,
    and which has the following form.  Either
    \(M_2' = \Z[\vartheta]/(1-\vartheta)^\ell\)
    for some \(\ell\in\N_{\ge1}\)
    or
    \(M_2' = \Z[\vartheta] \oplus Q\).
  \end{lemma}

  \begin{proof}
    Let \(a_0\in M_2\)
    generate \(\ker (1-\vartheta)\).
    We recursively construct \(a_n\in M_2\)
    with \((1-\vartheta)a_n = a_{n-1}\)
    as long as \(a_{n-1} \in \im (1-\vartheta)\).
    Assume first that this process stops at some \(\ell\in\N\),
    that is, \(a_0,\dotsc,a_{\ell-1}\)
    are defined and \(a_{\ell-1} \notin \im (1-\vartheta)\).
    Define \(\psi\colon \Z[\vartheta] \to M_2\),
    \(f\mapsto f \cdot a_{\ell-1}\).
    So \(\psi\bigl((1-\vartheta)^j\bigr) = a_{\ell-j-1}\)
    for \(j=0,\dotsc,\ell-1\)
    and \(\psi\bigl((1-\vartheta)^\ell\bigr) = 0\).
    An argument as in the proof of Lemma~\ref{lem:S-module_surjective}
    shows that the map \(\Z[\vartheta]/(1-\vartheta)^\ell \to M_2\)
    induced by~\(\psi\)
    is injective.  And~\(\psi\)
    induces an isomorphism between the kernel and cokernel of
    multiplication by~\(1-\vartheta\)
    on \(\Z[\vartheta]/(1-\vartheta)^\ell\)
    and~\(M_2\).
    By the Snake Lemma, multiplication by~\(1-\vartheta\)
    is invertible on the cokernel of~\(\psi\).
    Then the cokernel of~\(\psi\)
    is uniquely \(p\)\nb-divisible
    by Proposition~\ref{pro:divisible_S-module}.  So~\(M_2\)
    has the asserted form with
    \(M_2' = \Z[\vartheta]/(1-\vartheta)^\ell\).

    Now assume that~\(a_\ell\)
    is defined for all \(\ell\in\N\).
    Let \(X_0\subseteq M_2\)
    be the subgroup generated by the~\(a_\ell\).
    This is a submodule because \((1-\vartheta)a_\ell = a_{\ell-1}\)
    for \(\ell\ge1\)
    and \((1-\vartheta)a_0 = 0\).
    By construction, multiplication by~\(1-\vartheta\)
    on~\(X_0\)
    is surjective and its kernel is \((\Z/p)\cdot a_0 \cong \Z/p\).
    So the map \(\psi_0\colon Q \to X_0\)
    defined by \((1-\vartheta)^{-n} f \mapsto f(\vartheta) a_n\)
    for all \(n\in\N\),
    \(f\in\Z[\vartheta]\)
    is an isomorphism by Lemma~\ref{lem:S-module_surjective}.  The
    Snake Lemma implies that multiplication by~\(1-\vartheta\)
    is injective on~\(M_2/X_0\)
    and still has the same cokernel~\(\Z/p\).
    So Lemma~\ref{lem:S-module_injective} gives an injective
    \(\Z[\vartheta]\)-module
    homomorphism \(\dot\psi_1\colon \Z[\vartheta] \to M_2/X_0\)
    with uniquely \(p\)\nb-divisible
    cokernel.  We may lift~\(\dot\psi_1\)
    to a \(\Z[\vartheta]\)-module
    homomorphism \(\psi_1\colon \Z[\vartheta] \to M_2\)
    because~\(\Z[\vartheta]\)
    is free.  Since~\(\dot\psi_1\) is injective, the resulting map
    \[
    (\psi_0,\psi_1)\colon Q \oplus \Z[\vartheta] \to M_2
    \]
    is still injective.  Its cokernel is isomorphic to the cokernel
    of~\(\dot\psi_1\),
    which is uniquely \(p\)\nb-divisible.
    So~\(M_2'\) has the asserted form.
  \end{proof}

  Let \(M_2'\subseteq M_2\)
  be the submodule given by Lemma~\ref{lem:Ztheta_p_ker-coker_1}.  We
  study the two possibilities for~\(M_2'\) separately.

  \begin{lemma}
    \label{lem:case_11}
    Let \(M_2' = \Z[\vartheta]/(1-\vartheta)^\ell\)
    in Lemma~\textup{\ref{lem:Ztheta_p_ker-coker_1}}.  There are
    \(c\in \{1,\dotsc,p-1\}\)
    and a submodule \(M_0'\subseteq M_0\)
    that contains \(M_1= \im N(t)\), such that
    \[
    M_0' \cong \Z[t]\bigm/ \bigl(t^p-1,(1-t)^\ell - c p^{k-1} N(t)\bigr)
    \]
    and \(M_0/M_0' \cong M_2/M_2'\)
    is uniquely \(p\)\nb-divisible.
  \end{lemma}

  \begin{proof}
    Above Lemma~\ref{lem:Ztheta_p_ker-coker_1}, we have chosen
    \(b\in M_0\)
    with \(N(t)b = [1]\)
    in~\(\Z/p^k\).
    We have seen that~\((1-t)b\)
    generates \(\coker (1-t)|_{M_2}\).
    We may add any element of \(\ker N(t) = \im (1-t) = M_2\)
    to~\(b\)
    without changing \(N(t)b = [1]\).
    This way, we may add an arbitrary element of \((1-t)M_2\)
    to~\((1-t)b\)
    and arrange that \((1-t)b \in M_2'\).
    The group~\(M_2'\)
    is also a ring, and~\((1-t)b\)
    must be invertible in its ring structure because it generates
    \(\coker (1-t)|_{M_2'} \cong \Z/p\).
    There is a group automorphism of~\(M_2'\)
    that maps~\((1-t)b\)
    to the unit element in~\(M_2'\).
    Then we may assume without loss of generality that \((1-t)b = 1\).
    Then \((1-t)^{\ell}b = (1-\vartheta)^{\ell-1} \in M_2'\)
    is a generator of \(\ker (1-t)|_{M_2}\).
    Above Lemma~\ref{lem:Ztheta_p_ker-coker_1}, we have seen that
    \([p^{k-1}]\in M_1 \cap M_2\)
    is also a generator for \(\ker (1-t)|_{M_2}\).
    So there is \(c\in \{1,\dotsc,p-1\}\) such that
    \begin{equation}
      \label{eq:annihilator_b_non-trivial}
      (1-t)^\ell b = c[p^{k-1}] = c p^{k-1} N(t) b.
    \end{equation}
    That is, \((1-t)^\ell - c p^{k-1} N(t)\)
    belongs to the annihilator of~\(b\).
    So we get a well defined map
    \[
    \psi\colon \Sring\bigm/ \bigl((1-t)^\ell - c p^{k-1} N(t), p^k N(t)\bigr) \to M_0.
    \]
    We have built an exact \(\Kring\)\nb-module~\(M'\)
    with
    \(M_0' = \Sring\bigm/ \bigl((1-t)^\ell - c p^{k-1} N(t), p^k
    N(t)\bigr)\)
    in Example~\ref{exa:actions_on_Cuntz_8}.  This example has
    \(\Ga{2}{1}^M = 0\)
    and \(\Ga{1}{2}^M = 0\),
    and so it falls under the classification in
    Corollary~\ref{cor:a_2112_exact}.  So~\(M_0'\)
    is cohomologically trivial and~\(M'\)
    is the exact \(\Kring\)\nb-module
    associated to~\(M_0'\)
    by the construction in Corollary~\ref{cor:a_2112_exact}.  By
    Corollary~\ref{cor:a_2112_exact}, the \(\Sring\)\nb-module
    homomorphism~\(\psi\)
    above induces a homomorphism of exact \(\Kring\)\nb-modules
    \(\psi_*\colon M'\to M\),
    which is injective if and only if~\(\psi\)
    is injective.  Actually, \(\psi\)
    clearly induces an isomorphism \(M_1' \congto M_1\).
    So injectivity on~\(M_0'\)
    follows if the restriction of~\(\psi\)
    to \(M_2' = \ker N(t^{M'}) = \im N(t^{M'}) \subseteq M_0'\)
    is injective.  And this follows from
    Lemma~\ref{lem:Ztheta_p_ker-coker_1} and the description
    of~\(M_2'\) in Example~\ref{exa:actions_on_Cuntz_8}.
  \end{proof}

  Next we consider the case with \(M_2' = \Z[\vartheta] \oplus Q\)
  in Lemma~\ref{lem:Ztheta_p_ker-coker_1}.  As above, choose
  \(b\in M_0\)
  with \(N(t) b = [1] \in \Z/p^k\).
  Then \((1-t)b\)
  generates \(\coker (1-t)|_{M_2}\).
  We use this element to generate the direct summand
  \(\Z[\vartheta]\)
  in~\(M_2'\).
  Recall that \(M_2 = \ker N(t) = \im (1-t)\)
  and that
  \(\ker (1-t) \cap \im (1-t) = \ker N(t) \cap \im N(t) \cong \Z/p\)
  corresponds to the subgroup \(p^{k-1} \Z/p^k\)
  in \(M_1 = \Z/p^k\).
  So \(j \in \Z/p^k\)
  belongs to~\(M_2\)
  if and only if \(p^{k-1} \mid j\).
  And \([p^{k-1}] = p^{k-1} N(t) b\)
  generates \(\ker (1-t)|_{M_2} \cong \Z/p\).
  We use this generator to build the embedding
  \(\psi_Q\colon Q \hookrightarrow M_2\)
  in the proof of Lemma~\ref{lem:Ztheta_p_ker-coker_1}.  Then we
  define a map \(\psi\colon \Sring\oplus Q \to M_0\)
  by \(\psi(f) \defeq f(t)\cdot b\)
  for \(f\in \Sring\) and \(\psi(x) \defeq \psi_Q(x)\) for \(x\in Q\).

  \begin{lemma}
    \label{lem:ann_b_case_12}
    The kernel of the map \(\Sring\oplus Q \to M_0\)
    above is the subgroup generated by
    \((p^{k-1} N(t) ,1) \in \Sring\oplus Q\).
  \end{lemma}

  \begin{proof}
    We have built the embedding~\(\psi_Q\)
    so that \(\psi(p^{k-1} N(t) ,1)=0\).
    Now let \((f_1,f_2)\in \Sring\oplus Q\)
    satisfy \(\psi(f_1,f_2)=0\).
    We must show that \((f_1,f_2)\)
    is a multiple of \((p^{k-1} N(t) ,1)\).
    Since the image of \(\psi(f_1,f_2)\)
    in~\(M_0/M_2\)
    vanishes, \(p^k\)
    divides \(f_1(1)\).
    So we may write \(f_1 = j\cdot p^{k-1} N(t) + (1-t) f_3\)
    for some \(j\in\Z\),
    \(f_3\in \Sring\).
    Then \((f_1,f_2) = j\cdot (p^{k-1} N(t) ,1) + ((1-t)f_3,f_2-j)\).
    Lemma~\ref{lem:Ztheta_p_ker-coker_1} shows that~\(\psi\)
    restricts to an injective map on
    \((1-t) \Sring\oplus Q \cong \Z[\vartheta] \oplus Q\).
    Hence \(((1-t)f_3,f_2-j)=0\).
  \end{proof}

  Let~\(M'\)
  be the exact \(\Kring\)\nb-module
  in Example~\ref{exa:actions_on_Cuntz_9}.
  Lemma~\ref{lem:ann_b_case_12} shows that the map
  \(\psi\colon \Sring\oplus Q \to M_0\)
  descends to an injective map \(M_0' \to M_0\).
  Now the same argument as above shows that the map \(M' \to M\)
  associated to~\(\psi\)
  is injective and has uniquely \(p\)\nb-divisible
  cokernel.  This finishes the proof of Proposition~\ref{pro:case1}.
\end{proof}

The following proposition deals with the case
\(\ker \Ga{0}{1}^M = \{0\}\)
and \(\coker \Ga{1}{0}^M \cong \im \Ga{2}{1}^M \cong \Z/p\).

\begin{proposition}
  \label{pro:case2}
  Let~\(M\)
  be an exact \(\Kring\)\nb-module.
  Assume that \(M_1 = \Z/p^k\)
  as an Abelian group, that \(\ker \Ga{0}{1}^M = \{0\}\)
  and \(\im \Ga{2}{1}^M \cong \Z/p\).
  Then there is a \(\Kring\)\nb-module
  extension \(M' \into M \prto M''\),
  where~\(M'\)
  is the exact \(\Kring\)\nb-module
  in Example \textup{\ref{exa:actions_on_Cuntz_6}} and~\(M''\)
  is a uniquely \(p\)\nb-divisible
  exact \(\Kring\)\nb-module
  as in Example~\textup{\ref{exa:divisible}} with \(X = 0\)
  and \(Z=0\).
\end{proposition}

\begin{proof}
  The assumption \(\ker \Ga{0}{1}^M = \{0\}\)
  is equivalent to \(\Ga{1}{2}^M=0\)
  by exactness.  So we are in the situation of
  Theorem~\ref{the:a21_vanishes}.  Thus exact \(\Kring\)\nb-modules
  of this type are equivalent to certain triples
  \((M_0,\dot{M}_1,\dot\alpha)\).
  Here \(M_0\)
  is an \(\Sring\)\nb-module
  with \(\ker N(t) = \im (1-t)\);
  \(\dot{M}_1\)
  is a \(\Z/p\)\nb-subvector
  space of \(H^0(M_0) \defeq \ker (1-t) \bigm/ \im N(t)\)
  and~\(\dot\alpha\)
  is an isomorphism
  \(\dot\alpha\colon \dot{M}_1 \congto H^0(M_0)/\dot{M}_1\).
  The assumption \(\im \Ga{2}{1}^M\cong\Z/p\)
  implies \(\im\dot\alpha\cong\Z/p\)
  or, equivalently, \(\dot{M}_1\cong\Z/p\).
  Hence \(\ker (1-t)/M_1 \cong H^0(M_0)/\dot{M}_1 \cong \Sigma\Z/p\).
  And \(\im (1-t) \cong \Z/p^{k-1}\).

  \begin{lemma}
    \label{lem:case2_surjective_t-1}
    There is an injective \(\Sring\)\nb-module
    homomorphism
    \(\psi'\colon Q \oplus \Sigma Q\hookrightarrow\ker N(t)\)
    such that \(\ker N(t)/\im(\psi')\) is uniquely \(p\)\nb-divisible.
  \end{lemma}

  \begin{proof}
    Since \(p M_1 = N(t) (M_1)\)
    and \(\im N(t) = N(t) (M_0)\)
    are two subgroups in \(M_1 \cong \Z/p^k\)
    of index~\(p\),
    they are equal.  Therefore, any \(x\in M_0\)
    may be written as \(x=x_1 + x_2\)
    with \(x_1\in M_1\)
    and \(x_2\in \ker N(t)\):
    choose \(x_1\in M_1\)
    with \(N(t) x_1 = N(t)x\)
    and put \(x_2 \defeq x-x_1\).
    Since \((1-t)|_{M_1}=0\),
    it follows that \((1-t)(M_0) = (1-t)(\ker N(t))\).
    Since \(\ker N(t) = \im (1-t)\),
    the restriction \(1-t\colon \ker N(t) \to \ker N(t)\)
    is surjective.

    Next we claim that
    \(\ker (1-t) \cap \ker N(t) \cong \Z/p \oplus \Sigma \Z/p\).
    Recall that \(\ker (1-t)/M_1 \cong \Sigma \Z/p\)
    in the case under study.  Let \(x\in \ker (1-t)\)
    represent the (odd) generator of \(\ker (1-t)/M_1\).
    We have seen above that there is \(x_1\in M_1\)
    with \(N(t)x = N(t)x_1\).
    Then \(x-x_1\)
    is another representative of the generator of \(\ker (1-t)/M_1\)
    that also belongs to \(\ker N(t)\).
    So any element of \(\ker (1-t)/M_1 \cong \Sigma\Z/p\)
    lifts to an element of \(\ker (1-t) \cap \ker N(t)\).
    This lifting is unique modulo \(\ker N(t) \cap M_1\).
    And \(M_1=\Z/p^k\)
    implies that \(\ker N(t) \cap M_1 \cong \Z/p\).
    This finishes the proof that
    \(\ker (1-t) \cap \ker N(t) \cong \Z/p \oplus \Sigma \Z/p\).
    Since the restriction of~\(1-t\)
    to \(\ker(1-t)\)
    is surjective, Lemma~\ref{lem:S-module_surjective} gives an
    embedding \(\psi'\colon Q \oplus \Sigma Q \to \ker (1-t)\)
    with uniquely \(p\)\nb-divisible cokernel.
  \end{proof}

  Let~\(M'\)
  be the exact \(\Kring\)\nb-module
  built in Example~\ref{exa:actions_on_Cuntz_6}.  Since
  \(\Ga{1}{2}^{M'} = 0\),
  it is also associated to a suitable triple as in
  Theorem~\ref{the:a21_vanishes}.  Let
  \(X\defeq \im(\psi') \subseteq \ker(1-t) \subseteq M_0\).
  Since~\(X\)
  contains \(\ker (1-t) \cap \ker N(t)\),
  it follows from the proof of Lemma~\ref{lem:case2_surjective_t-1}
  that \(M_1 \cap X = M_1 \cap \ker N(t) = \Z/p\).
  We may choose the map~\(\psi'\)
  so that the canonical generator \([p^{k-1}]\in M_1=\Z/p^k\)
  of \(M_1 \cap X\)
  goes to~\((1,0)\).
  Then \(M_1 + X \subseteq M_0\)
  is canonically isomorphic to~\(M_0'\)
  as an \(\Sring\)\nb-module.
  We choose the generator of~\(\Sigma Q\)
  in~\(X\)
  to be \(\Ga{2}{1}^M([1])\).
  Then the injective \(\Sring\)\nb-module
  homomorphism \(\psi\colon M_0' \to M_0\)
  is a homomorphism of triples as in Theorem~\ref{the:a21_vanishes}.
  The criterion in Theorem~\ref{the:a21_vanishes} shows that the
  induced homomorphism of exact \(\Kring\)\nb-modules
  \(M' \to M\)
  is injective.  As above, it follows that the quotient
  \(M'' \defeq M/M'\)
  is an exact \(\Kring\)\nb-module
  with \(M''_1=0\).
  Then it is isomorphic to an exact \(\Kring\)\nb-module
  as in Example~\ref{exa:divisible} with \(X=0\)
  and \(Z=0\).
  This finishes the proof of Proposition~\ref{pro:case2}.
\end{proof}

Now we turn to the case \(\ker \Ga{0}{1}^M = p^{k-1}\Z/p^k\)
and \(\coker \Ga{1}{0}^M = \{0\}\).

\begin{proposition}
  \label{pro:case3}
  Let~\(M\)
  be an exact \(\Kring\)\nb-module.
  Assume that \(M_1 = \Z/p^k\)
  as an Abelian group, that \(\ker \Ga{0}{1}^M = p^{k-1}\Z/p^k\)
  and \(\coker \Ga{1}{0}^M = \{0\}\).
  Then there is a \(\Kring\)\nb-module
  extension \(M' \into M \prto M''\),
  where~\(M'\)
  is the exact \(\Kring\)\nb-module
  in Example \textup{\ref{exa:actions_on_Cuntz_7}} and~\(M''\)
  is a uniquely \(p\)\nb-divisible
  exact \(\Kring\)\nb-module
  as in Example~\textup{\ref{exa:divisible}} with \(X = 0\)
  and \(Z=0\).
\end{proposition}

\begin{proof}
  The assumptions \(\ker \Ga{0}{1}^M = p^{k-1}\Z/p^k\)
  and \(\coker \Ga{1}{0}^M = \{0\}\)
  are equivalent to \(\im (\Ga{1}{2}^M) \cong \Z/p\)
  and \(\Ga{2}{1}^M=0\)
  by exactness.  So we are in the situation of
  Theorem~\ref{the:a12_vanishes}.  Exact \(\Kring\)\nb-modules
  with \(\Ga{2}{1}^M=0\)
  are equivalent to certain triples \((M_0,\dot{M}_2,\dot\alpha)\),
  where~\(M_0\)
  is an \(\Sring\)\nb-module
  with \(\ker (1-t) = \im N(t)\),
  \(\dot{M}_2\)
  is a \(\Z/p\)\nb-subvector
  space of \(H_0(M_0) \defeq \frac{\ker N(t)}{\im (1-t)}\),
  and~\(\dot\alpha\)
  is a grading-reversing isomorphism
  \(\dot\alpha\colon \dot{M}_2 \congto H_0(M_0)/\dot{M}_2\).
  We have \(\dot{M}_2\cong \im(\Ga{1}{2}^M) \cong \Z/p\).
  We also need
  \[
  M_1 = M_0/ M_2 \cong M_0\bigm/ (\im (1-t) + \dot{M}_2)
  \]
  to be isomorphic to~\(\Z/p^k\).
  And \(\Ga{2}{1}^M=0\)
  implies \(\Gu{1}^M - \Gs{1}^M=0\).
  So Lemma~\ref{lem:short_exact_sequences_consequences} implies
  \(\ker (\Gu{0}^M - \Gt{0}^M) = \ker (\Ga{2}{0}^M)\).
  This is equal to \(\im(\Ga{0}{1}^M)\)
  because~\(M\)
  is exact, and this is isomorphic to \(\Z/p^{k-1}\)
  because \(\ker(\Ga{0}{1}^M) = \im(\Ga{1}{2}^M) = p^{k-1}\Z/p^k\).
  In terms of our triples, this means that
  \[
  \ker (1-t) = N(t)(M_0) \cong \Z/p^{k-1}.
  \]
  One short exact sequence in Table~\ref{tab:short_exact_sequences}
  implies \(\ker (\Ga{1}{0}^M) \cap \im(\Ga{0}{1}^M)=0\)
  because \(\Ga{0}{1}^M\)
  and \(N(\Gs{1}^M)\)
  have the same kernel, namely, \(\Z/p\).
  Hence \(\im(\Ga{0}{2}^M) \cap \ker (\Gu{0}^M - \Gt{0}^M)= 0\).
  In terms of our triples above, this means that
  \(\ker(1-t)\cap M_2=0\),
  that is, \(1-t\colon M_2 \to M_2\)
  is injective.  And \(\coker (\Ga{1}{0}^M)=0\)
  and \(\coker N(\Gs{1}^M)\cong \Z/p\) imply
  \[
  M_0/(\im (\Ga{0}{1}^M) + \ker (\Ga{1}{0}^M)) \cong \Z/p.
  \]

  The map~\(\Ga{1}{2}^M\)
  is a grading-reversing isomorphism from
  \(\coker (\Ga{2}{0}^M)\) onto \(\im (\Ga{1}{2}^M) \cong \Z/p\).
  So \(\coker (\Ga{2}{0}^M) \cong \Sigma\Z/p\).
  Now one of the short exact sequences in
  Table~\ref{tab:short_exact_sequences} gives an extension
  \(\Z/p \into \coker (\Gu{2}^M- \Gt{2}^M) \prto \Sigma \Z/p\).  So
  \[
  M_2/(1-t)M_2 \cong \Z/p \oplus \Sigma \Z/p.
  \]
  The proof also allows us to choose representatives
  \(e_0,e_1 \in M_2\)
  for the generators of \(M_2/(1-t)M_2\).
  First, let \(a\in M_0\)
  be an even-degree element that lifts the generator of
  \(M_0/M_2 = M_1 = \Z/p^k\).
  The image of~\(a\)
  in~\(\Z/p^k\)
  does not belong to \(\im N(t) = \ker (1-t)\).
  Thus the element \(e_0\defeq (1-t)a\)
  of \(\im (1-t) \subseteq M_2\)
  cannot be of the form \((1-t)x\)
  with \(x\in M_2\)
  because \(a-x_2\)
  is another representative of the generator of~\(\Z/p^k\)
  and so \(N(t) (a-x_2)\neq0\).
  Secondly, there is an odd-degree element \(e_1\in M_2\)
  with \(\Ga{1}{2}^M(e_1) = [p^{k-1}] \in \Z/p^k\).
  Then \(e_1 \notin (1-t)M_2\).

  Now Lemma~\ref{lem:S-module_injective} shows that the
  grading-preserving map
  \[
  \varphi\colon \Z[\vartheta] \oplus \Sigma \Z[\vartheta] \to M_2,\qquad
  (f_0,f_1)\mapsto f_0\cdot e_0+ f_1\cdot e_1,
  \]
  is injective with uniquely \(p\)\nb-divisible
  cokernel.  To include~\(M_1\), we use the grading-preserving map
  \[
  \psi\colon \Sring \oplus \Sigma \Z[\vartheta] \to M_0,\qquad
  (f_0,f_1)\mapsto f_0\cdot a+ f_1\cdot e_1.
  \]
  Its image surjects onto~\(M_1\)
  because~\(a\) lifts a generator of~\(M_1\).  We claim that
  \[
  \ker(\psi) = (p^{k-1} N(t)) \oplus 0.
  \]
  First, \(\psi(f_0,f_1) \in M_2\)
  for \(f_0\in \Sring\),
  \(f_1\in \Sigma\Z[\vartheta]\)
  is equivalent to \(f_0 \in (t-1,p^k) = (t-1,p^{k-1} N(t))\)
  because \(M_1 \cong \Z/p^k \cong \Sring/(t-1,p^k)\)
  as an \(\Sring\)\nb-module.
  The element \(p^{k-1} N(t) a\)
  becomes~\(0\)
  in \(M_1=M_0/M_2\),
  so that it belongs to~\(M_2\).
  And it also belongs to \(\ker(1-t)\).
  Since \(M_2 \cap \ker(1-t)=0\),
  it follows that \(p^{k-1} N(t)\cdot a=0\).
  That is, \((p^{k-1} N(t)) \oplus 0\)
  is contained in the kernel of~\(\psi\).
  Conversely, let \(\psi(f_0,f_1)=0\).
  Then \(\psi(f_0,f_1)\in M_2\),
  so that \(f_0 \in (t-1,p^{k-1} N(t))\).
  That is, we may write \(f_0 = (1-t) f_2 + p^{k-1} N(t) f_3\).
  The second summand is contained in the kernel of~\(\psi\).
  So \(\psi((1-t) f_2,f_1) = f_2\cdot e_0 + f_1 \cdot e_1 =0\).
  The injectivity of~\(\varphi\)
  above shows that \(f_2,f_1\in (N(t))\).
  Hence \((1-t)f_2=0\)
  in~\(\Sring\)
  and \(f_1=0\)
  in \(\Z[\vartheta]\).
  So~\(p^{k-1} N(t)\)
  generates the kernel of~\(\psi\)
  as asserted.  We have found an injective map from the
  \(\Sring\)\nb-module
  at~\(0\)
  in Example~\ref{exa:actions_on_Cuntz_7} to~\(M_0\).
  This induces a homomorphism of exact \(\Kring\)\nb-modules
  by Theorem~\ref{the:a12_vanishes}.  Since it restricts to an
  isomorphism \(M'_1 \congto M_1\),
  the criterion in Theorem~\ref{the:a12_vanishes} shows that the
  map \(M_2' \to M_2\)
  is injective as well.  The quotient \(M'' \defeq M/M'\)
  is an exact \(\Kring\)\nb-module
  with \(M''_1=0\).
  Thus it is as in Example~\ref{exa:divisible} with \(X=0\)
  and \(Z=0\).
  This finishes the proof of Proposition~\ref{pro:case3}.
\end{proof}

Now we turn to the fourth case, \(\ker (\Ga{0}{1}^M) \cong \Z/p\)
and \(\coker (\Ga{1}{0}^M) \cong \Z/p\).
This is the only case where exact \(\Kring\)\nb-modules
with a non-trivial \(\Z/p\)\nb-action
on~\(M_1\)
may occur.  We have already shown at the beginning of
Section~\ref{sec:prove_theorem} that, without loss of generality,
\(\Gs{1}\)
is multiplication with some element of~\(\Z/p^k\)
of the form \(1-p^{k-1}\tau\)
for some \(\tau\in \Z/p\),
where \(\tau=0\)
if \(k=1\)
or if \(k=2\)
and \(p=2\).
This leaves out about half of the cases for \(p=2\),
which are covered by Example~\ref{exa:actions_on_Cuntz_10} because it
allows to change the sign of~\(\Gs{1}^M\)
if \(p=2\)
and \(k>1\).
The following proposition describes the exact \(\Kring\)\nb-modules
with \(M_1 = \Z/p^k\)
and \(\Gs{1}^M\neq 1\).
It gives Theorem~\ref{the:actions_on_Cuntz} for all exact
\(\Kring\)\nb-modules
with \(M_1=\Z/p^k\) and a non-trivial \(\Z/p\)\nb-action on~\(M_1\).

\begin{proposition}
  \label{pro:case4a}
  Let \(\tau\in\{1,\dotsc,p-1\}\).
  Assume \(p\neq 2\)
  and \(k\ge2\),
  or \(k\ge 3\).
  Let~\(M\)
  be an exact \(\Kring\)\nb-module
  with \(M_1 = \Z/p^k\)
  and \(\Gs{1}^M(x) = (1-p^{k-1}\tau)\cdot x\).
  Then there is an injective homomorphism from the exact
  \(\Kring\)\nb-module~\(M'\)
  in Example~\textup{\ref{exa:actions_on_Cuntz_3}} into~\(M\)
  whose cokernel is a uniquely \(p\)\nb-divisible
  exact \(\Kring\)\nb-module
  as in Example~\textup{\ref{exa:divisible}} with \(X=0\) and \(Z=0\).
\end{proposition}

The proof of Proposition~\ref{pro:case4a} is parallel to the proof of
the next proposition:

\begin{proposition}
  \label{pro:case4b}
  Let~\(M\)
  be an exact \(\Kring\)\nb-module
  with \(M_1 = \Z/p^k\),
  \(\Gs{1}^M = \Gu{1}^M\),
  \(\ker (\Ga{0}{1}^M) \cong \Z/p\),
  and \(\coker (\Ga{1}{0}^M) \cong \Z/p\).
  Assume \(k\ge2\)
  or \(p\neq2\).
  Then there is an extension of exact \(\Kring\)\nb-modules
  \(M' \into M \prto M''\),
  where~\(M'\)
  is isomorphic to one of the exact \(\Kring\)\nb-modules
  described in Examples \textup{\ref{exa:actions_on_Cuntz_4}}
  and~\textup{\ref{exa:actions_on_Cuntz_5}} and~\(M''\)
  is uniquely \(p\)\nb-divisible
  and has the form described in Example~\textup{\ref{exa:divisible}}
  with \(X=0\) and \(Z=0\).
\end{proposition}

\begin{proof}[Proof of Propositions \textup{\ref{pro:case4a}}
  and~\textup{\ref{pro:case4b}}]
  In the situation of Proposition~\ref{pro:case4a},
  \(\Gs{1}^M \neq \Gu{1}^M\),
  forcing \(\Ga{1}{2}^M\neq0\)
  and \(\Ga{2}{1}^M\neq 0\).
  Then~\(\Ga{0}{1}^M\)
  cannot be injective and~\(\Ga{1}{0}^M\)
  cannot be surjective.  So we must be in the case where
  \(\ker (\Ga{0}{1}^M) \cong \Z/p\),
  and \(\coker (\Ga{1}{0}^M) \cong \Z/p\).
  This is already assumed in Proposition~\ref{pro:case4b}, so it
  holds in the situation of both propositions.  It implies
  \begin{align}
    \label{eq:case_pp1_ker_21}
    \im (\Ga{1}{2}^M) &=
    \ker (\Ga{0}{1}^M) = p^{k-1}\Z/p^k = \ker N(\Gs{1}),\\
    \label{eq:case_pp1_im_12}
    \ker (\Ga{2}{1}^M) &=
    \im (\Ga{1}{0}^M) = p\Z/p^k = \im N(\Gs{1}).
  \end{align}
  Equations \eqref{eq:case_pp1_ker_21} and~\eqref{eq:case_pp1_im_12}
  give statement~\ref{lem:decomposable_case_criterion2} in
  Lemma~\ref{lem:decomposable_case_criterion_0}.  That lemma shows that
  \begin{equation}
    \label{eq:case_pp_sum}
    M_0 \cong \im (\Ga{0}{1}^M) \oplus \im(\Ga{0}{2}^M).
  \end{equation}
  The discussion after Lemma~\ref{lem:decomposable_case_criterion_0}
  culminates in a long exact
  sequence~\eqref{eq:decomposable_case_long_exact}
  involving the maps
  \(N(\Gs{1}^M)\),
  \(\Ga{2}{1}^M\),
  \(\Gu{2}^M - \Gt{2}^M\),
  \(\Ga{1}{2}^M\).
  Our next goal is to understand~\(M_2\)
  with the maps \(\Gs{2}^M\)
  and~\(\Gt{2}^M\).
  We claim that \(N(\Gs{2}^M) = p\)
  and hence \(N(\Gt{2}^M) = 0\)
  by~\eqref{eq:alpha_2_N};
  so \((M_2,\Gt{2}^M)\)
  is a module over \(\Z[\vartheta] = \Z[t]/(N(t))\).
  If \(k\ge2\),
  then \eqref{eq:case_pp1_ker_21} and~\eqref{eq:case_pp1_im_12} imply
  \(\im (\Ga{1}{2}^M) \subseteq \ker (\Ga{2}{1}^M)\),
  so that
  \[
  \Gu{2}^M - \Gs{2}^M
  = \Ga{2}{1}^M \circ \Ga{1}{2}^M
  = 0.
  \]
  This implies \(N(\Gs{2}^M) = p\).
  Now let \(k=1\)
  and \(p\neq2\).
  Then \(\Gs{1}^M=1\),
  and this implies \((1-\Gs{2}^M)^2=0\)
  by a computation as in~\eqref{eq:s1_unipotent}.
  The assumption
  \(\im (\Ga{2}{1}^M) \cong \coker (\Ga{1}{0}^M) \cong \Z/p\)
  implies \(p\cdot\Ga{2}{1}^M=0\)
  and hence \(p\cdot (1-\Gs{2}^M)=0\).
  Now Lemma~\ref{lem:norm_if_unipotent} shows that
  \(N(\Gs{2}^M) = p\).
  So this holds in the situation of both propositions.
  
  Equation~\eqref{eq:case_pp1_ker_21} implies
  \(\im (\Ga{1}{2}^M) \cong \Z/p\).  The map~\(\Ga{2}{1}^M\) induces
  a grading-reversing isomorphism from \(M_1/ \ker (\Ga{2}{1}^M)\)
  to \(\im (\Ga{2}{1}^M)\).  So~\eqref{eq:case_pp1_im_12} implies
  \(\im (\Ga{2}{1}^M) \cong \Sigma\Z/p\).  Then the long exact
  sequence in~\eqref{eq:decomposable_case_long_exact} shows that
  \[
  \ker (\Gu{2}^M - \Gt{2}^M) = \im(\Ga{2}{1}^M) \cong \Sigma\Z/p,\qquad
  \coker (\Gu{2}^M - \Gt{2}^M) \cong \Sigma\Z/p.
  \]
  We have classified \(\Z[\vartheta]\)-modules
  with this property in Lemma~\ref{lem:Ztheta_p_ker-coker_1} and found
  that there are two basic examples, such that all others are
  extensions of them by uniquely \(p\)\nb-divisible
  \(\Z[\vartheta]\)\nb-modules.
  So the proof now splits into two subcases.  Namely, \(M_2\)
  contains either \(M_2' = \Sigma\Z[\vartheta] \oplus \Sigma Q\)
  or \(M_2' = \Sigma\Z[\vartheta]/(1-\vartheta)^\ell\)
  for some \(\ell\in\N_{\ge1}\),
  such that the quotient \(M_2/M_2'\)
  is uniquely \(p\)\nb-divisible.
  Here the~\(\Sigma\)
  reminds us that these \(\Z[\vartheta]\)-submodules
  occur in odd parity.

  Assume first that~\(M_2\)
  contains \(M_2'\defeq \Sigma\Z[\vartheta] \oplus \Sigma Q\).
  By our description of the image and kernel of~\(\Ga{2}{1}^M\),
  it must factor through the quotient map \(\Z/p^k \prto \Z/p\),
  an isomorphism
  \(\Gd{2}{1}^{M'}\colon \Z/p \congto \ker (\Gu{2}^M - \Gt{2}^M)\),
  and the inclusion
  \(\ker (\Gu{2}^M - \Gt{2}^M) \hookrightarrow M_2' \subseteq M_2\).
  Similarly, \(\Ga{1}{2}^{M'}\)
  factors through the quotient map
  \(M_2 \prto \coker (\Gu{2}^M - \Gt{2}^M) = \coker (\Gu{2}^{M_2'} -
  \Gt{2}^{M_2'})\),
  an isomorphism
  \(\Gd{1}{2}^{M'}\colon \coker (\Gu{2}^{M_2'} - \Gt{2}^{M_2'})
  \congto \Z/p\),
  and the injective map \(\Z/p \to\Z/p^k\),
  \(c\mapsto c\cdot p^{k-1}\).
  The proof of Lemma~\ref{lem:Ztheta_p_ker-coker_1} starts with
  arbitrary generators of \(\ker (\Gu{2}^M - \Gt{2}^M)\)
  and \(\coker (\Gu{2}^M - \Gt{2}^M)\).
  We may choose them so that~\(\Gd{2}{1}^{M'}[1]\)
  is the image of~\(1\)
  in
  \(\Sigma Q =
  \Sigma\Z[\vartheta,1/p]\bigm/(1-\vartheta)\Z[\vartheta]\)
  and \(\Gd{1}{2}^{M'}\)
  maps the unit element in \(\Sigma\Z[\vartheta]\)
  to the unit in~\(\Z/p\).
  Let \(M_1' = M_1\)
  and
  \(M_0' = \Ga{0}{1}^{M'}(M_1) \oplus \Ga{0}{2}^{M'}(M_2') \subseteq
  M_0\);
  the decomposition of~\(M_0'\)
  is direct by~\eqref{eq:case_pp_sum}.  We may identify
  \begin{align*}
    \Ga{0}{1}^{M'}(M_1)
    &\cong M_1/\ker (\Ga{0}{1}^{M'}) \cong \Z/p^{k-1},\\
    \Ga{0}{2}^{M'}(M_2')
    &\cong M_2'/\ker (\Ga{0}{2}^{M'}\cap M_2')
      \cong M_2'/\ker (\Gu{2}^{M_2'} - \Gt{2}^{M_2'})
      \cong \Sigma\Z[\vartheta,1/p]/\Z[\vartheta]
      \cong \Sigma Q,
  \end{align*}
  where the last isomorphism is multiplication by~\(1-\vartheta\).
  Now inspection shows that~\(M'\)
  is isomorphic to the exact \(\Kring\)\nb-module
  in Example~\ref{exa:actions_on_Cuntz_5}.  The quotient
  \(\Kring\)\nb-module
  \(M'' \defeq M/M'\)
  is again exact and has \(M_1''=0\).
  So Theorem~\ref{the:divisible_exact_module} implies that~\(M''\)
  has the form of Example~\ref{exa:divisible} with \(X=0\)
  and \(Z=0\).
  This deals with the subcase where~\(M_2\)
  contains
  \(\Sigma\Z[\vartheta] \oplus \Sigma Q\).

  Next assume that~\(M_2\)
  contains \(M_2' = \Sigma\Z[\vartheta]/(1-\vartheta)^\ell\)
  for some \(\ell\ge1\).
  As above, the map~\(\Ga{2}{1}^{M'}\)
  must factor through the quotient map \(\Z/p^k \prto \Z/p\),
  an isomorphism
  \(\Gd{2}{1}^{M'}\colon \Z/p \congto \ker (\Gu{2}^M - \Gt{2}^M)\),
  and the inclusion
  \(\ker (\Gu{2}^M - \Gt{2}^M) \hookrightarrow M_2'\subseteq M_2\);
  and \(\Ga{1}{2}^{M'}\)
  factors through the quotient map
  \(M_2 \prto \coker (\Gu{2}^M - \Gt{2}^M)\),
  an isomorphism
  \(\Gd{1}{2}^{M'}\colon \coker (\Gu{2}^M - \Gt{2}^M) \congto \Z/p\),
  and the injective map \(\Z/p \to\Z/p^k\),
  \(c\mapsto c\cdot p^{k-1}\).
  Here the unit element of \(\Z[\vartheta]/(1-\vartheta)^\ell\)
  represents a generator of
  \(\coker (\Gu{2}^M - \Gt{2}^M) \cong \Z/p\),
  and \((1-\vartheta)^{\ell-1}\)
  generates \(\ker (\Gu{2}^M - \Gt{2}^M)\).
  By an automorphism of \(\Z[\vartheta]/(1-\vartheta)^\ell\),
  we may arrange that~\(\Gd{1}{2}^{M'}\)
  is the canonical map.  So~\(\Ga{1}{2}^{M'}\)
  restricts to \(p^{k-1} \mathrm{ev}_1\)
  as in Examples \ref{exa:actions_on_Cuntz_3}
  and~\ref{exa:actions_on_Cuntz_4}.  The other map~\(\Gd{2}{1}^{M'}\)
  involves an isomorphism of~\(\Z/p\),
  that is, multiplication with some \(\tau\in \{1,\dotsc,p-1\}\).
  As above, we put \(M_1' = M_1\)
  and \(M_0' = \Ga{0}{1}^{M'}(M_1) + \Ga{0}{2}^{M'}(M_2')\)
  and then show that
  \(M_0' \cong \Z/p^{k-1} \oplus
  \Sigma\Z[\vartheta]/(1-\vartheta)^{\ell-1}\).
  If \(\ell=1\),
  then the second summand goes away and the generator of
  \(\ker (1-\vartheta)\)
  in~\(M_2\)
  also represents a generator of \(\coker (1-\vartheta)\).
  Therefore, \(M'\)
  is isomorphic to the exact \(\Kring\)\nb-module
  in Example~\ref{exa:actions_on_Cuntz_3} with the given~\(\tau\).
  If \(\ell>1\),
  then we get the exact \(\Kring\)\nb-module
  in Example~\ref{exa:actions_on_Cuntz_4} with the given~\(\tau\).
  As in the case above, the quotient \(M'' \defeq M/M'\)
  is an exact \(\Kring\)\nb-module
  with \(M_1''=0\).
  Hence it is described by Example~\ref{exa:divisible}.
\end{proof}

Finally, we are left with the case where
\(\ker (\Ga{0}{1}^M) = \Z/p\),
\(\coker (\Ga{1}{0}^M) = \Z/p\),
\(k=1\)
and \(p=2\).
Here the equation \(N(\Gs{2}^M)=p\)
breaks down, so that the structure of~\(M_2\)
may be rather different.  We may treat this case using the
automorphisms of Köhler's ring~\(\Kring\)
for \(p=2\)
in Section~\ref{sec:prime_2}.  Since \(k=1\),
our assumptions say that \(\Ga{0}{1}^M = 0\)
and \(\Ga{1}{0}^M = 0\).
The involutive automorphism of~\(\Kring\)
that exchanges \(0\)
and~\(2\)
replaces~\(M\)
by an exact \(\Kring\)\nb-module~\(M^\sigma\)
with \(\Ga{2}{1}^M = 0\)
and \(\Ga{1}{2}^M = 0\)
or, equivalently, \(\ker \Ga{0}{1}^M = \{0\}\)
and \(\coker \Ga{1}{0}^M = \{0\}\).
So we are in the situation of Proposition~\ref{pro:case1}.  It gives
an extension \(M' \into M^\sigma \prto M''\),
where~\(M'\)
is as in Example \ref{exa:actions_on_Cuntz_8}
or~\ref{exa:actions_on_Cuntz_9} and~\(M''\)
is a uniquely \(p\)\nb-divisible
exact \(\Kring\)\nb-module
as in Example~\ref{exa:divisible} with \(X = 0\)
and \(Z=0\).
So~\(M\)
itself contains~\((M')^\sigma\)
with a uniquely \(p\)\nb-divisible
quotient.  And this is covered by
Example~\ref{exa:actions_on_Cuntz_10}.  We have treated all the
possibilities for exact \(\Kring\)\nb-modules
with cyclic~\(M_1\).
So the proof of Theorem~\ref{the:actions_on_Cuntz} is finished.

\section{Equivalence at the prime~\texorpdfstring{$p$}{p}}
\label{sec:p-equivalence}

The examples of exact \(\Kring\)\nb-modules
in Section~\ref{sec:representative_examples} only give all examples up
to uniquely \(p\)\nb-divisible
quotients.  We now study what this gives on the level of~\(\KK^G\).
We introduce a weaker notion of equivalence in~\(\KK^G\)
which, roughly speaking, looks only at the \(p\)\nb-torsion
information in an object.  When we classify up to this notion of
equivalence, then the uniquely \(p\)\nb-divisible
quotients become irrelevant, and also some of the examples in
Section~\ref{sec:representative_examples} are identified.

Let
\begin{equation}
  \label{eq:R-module_extension}
  \begin{tikzcd}
    M' \arrow[r, >->, "i"] &
    M \arrow[r, ->>, "q"] &
    M''
  \end{tikzcd}
\end{equation}
be an extension of countable, exact \(\Kring\)\nb-modules.
By Theorems \ref{the:Koehler_invariant_UCT_classifies}
and~\ref{the:range_F}, there are objects \(A',A\)
of~\(\mathfrak{B}^G\)
with \(F_*(A') \cong M'\)
and \(F_*(A) \cong M\),
and there is \(j\in \KK^G_0(A',A)\)
with \(F_*(j) = i\)
(up to the chosen isomorphisms \(F_*(A') \cong M'\)
and \(F_*(A) \cong M\)).
Since~\(\mathfrak{B}^G\)
is a triangulated category, the map~\(j\) is part of an exact triangle
\[
\begin{tikzcd}
  A' \arrow[r, "j"] &
  A \arrow[r, "r"] &
  A'' \arrow[r, "\delta"] &
  \Sigma A'.
\end{tikzcd}
\]
Here~\(\Sigma A''\)
is a ``cone'' of~\(j\).
Since~\(F_*\)
is a homological functor, it maps this exact triangle to a long exact
sequence of \(\Kring\)\nb-modules.
Since \(F_*(j)\)
is injective, \(F_*(\delta)=0\)
and \(F_*(r)\)
is a cokernel for \(F_*(j) = i\).
So there is an isomorphism \(F_*(A'') \cong M''\)
such that \(F_*(r) = q\)
(up to the chosen isomorphisms \(F_*(A) \cong M\)
and \(F_*(A'') \cong M''\)).
In other words, we have lifted the \(\Kring\)\nb-module
extension~\eqref{eq:R-module_extension} to an \(F_*\)\nb-exact,
exact triangle in~\(\mathfrak{B}^G\).
Theorem~\ref{the:actions_on_Cuntz} gives such
extensions~\eqref{eq:R-module_extension} where \(M''_1=0\).
So~\(M''\)
describes an action on an object that is non-equivariantly
\(\KK\)-equivalent
to~\(0\).
By Theorem~\ref{the:action_O2}, this corresponds to an action
on~\(\mathcal{O}_2\).
So the exact triangle above involves the objects of~\(\mathfrak{B}^G\)
that realise the exact \(\Kring\)\nb-modules
\(M'\)
and~\(M\)
and an action of~\(G\)
on~\(\mathcal{O}_2\).
In this sense, Theorem~\ref{the:actions_on_Cuntz} must allow
extensions as in~\eqref{eq:R-module_extension} merely to accommodate
exact triangles involving actions on~\(\mathcal{O}_2\).

In the situation of Theorem~\ref{the:actions_on_Cuntz}, we also know
that~\(M''\)
is uniquely \(p\)\nb-divisible.
So~\(j\) is a \(p\)\nb-equivalence in the following sense:

\begin{definition}
  \label{def:p-equivalence}
  Let \(A,B\)
  be objects of~\(\mathfrak{B}^G\).
  We call \(f\in \KK^G_0(A,B)\)
  a \emph{\(p\)\nb-equivalence}
  if \(F_*(\cone(f))\)
  is uniquely \(p\)\nb-divisible.
\end{definition}

We define \(p\)\nb-equivalence
for objects of~\(\mathfrak{B}^G\)
to be the equivalence relation that is generated by \(A\sim B\)
if there is a \(p\)\nb-equivalence
in \(\KK^G_0(A,B)\).
We make this more transparent using localisation techniques.  Let
\(\mathrm{UHF}(p^\infty)\)
be the UHF-algebra of type~\(p^\infty\).
Its \(\K\)\nb-theory
is \(\Z[1/p]\).
Let \(u\colon \C \to \mathrm{UHF}(p^\infty)\)
be the unit map.  On \(\K\)\nb-theory,
this is the inclusion \(\Z \hookrightarrow \Z[1/p]\).
Let \(L_p\defeq \cone(u)\).
Then \(\K_0(L_p)=0\)
and \(\K_1(L_p) = \Z[1/p]/\Z\).
We equip \(\C\),
\(\mathrm{UHF}(p^\infty)\)
and~\(L_p\)
with the trivial action of~\(G\).
Given any \(A\) in~\(\mathfrak{B}^G\), there is an exact triangle
\[
A \otimes  L_p \to
A \otimes \C \to A\otimes \mathrm{UHF}(p^\infty) \to A \otimes \Sigma L_p,
\]
and we may identify \(A\otimes \C\cong A\).
Since \(\K\)\nb-theory
commutes with inductive limits, so does our invariant~\(F_*\).  So
\[
F_*(A\otimes \mathrm{UHF}(p^\infty)) = F_*(A) \otimes_\Z \Z[1/p].
\]
This is uniquely \(p\)\nb-divisible.
So the boundary map \(A\otimes L_p \to A\)
in the exact triangle above is a \(p\)\nb-equivalence.

\begin{lemma}
  \label{lem:p-equivalent}
  Two objects \(A,B\)
  of~\(\mathfrak{B}^G\)
  are \(p\)\nb-equivalent
  if and only if \(A\otimes L_p\)
  and \(B\otimes L_p\) are \(\KK^G\)-equivalent.
\end{lemma}

\begin{proof}
  Since the boundary maps \(A\otimes L_p \to A\)
  and \(B\otimes L_p \to B\)
  are \(p\)\nb-equivalences,
  a \(\KK^G\)-equivalence
  between \(A\otimes L_p\)
  and \(B\otimes L_p\)
  implies that \(A\)
  and~\(B\)
  are \(p\)\nb-equivalent.
  Let \(f\in \KK^G_0(A,B)\)
  be a \(p\)\nb-equivalence.
  We claim that \(f\otimes 1\in \KK^G(A\otimes L_p, B\otimes L_p)\)
  is a \(\KK^G\)-equivalence.
  It follows from this claim that \(A\otimes L_p\)
  and \(B\otimes L_p\)
  are \(\KK^G\)-equivalent
  if \(A\)
  and~\(B\)
  are \(p\)\nb-equivalent.
  Let \(C\)
  be the cone of~\(f\).
  Then \(C\otimes L_p\)
  is a cone of \(f\otimes \id_{L_p}\).
  By assumption, \(F_*(C)\)
  is uniquely \(p\)\nb-divisible.
  Hence the canonical map \(F_*(C) \to F_*(C) \otimes_\Z \Z[1/p]\)
  is an isomorphism.  This is the map induced by the inclusion map
  \(C \to C\otimes \mathrm{UHF}(p^\infty)\).
  So \(F_*(C \otimes L_p)=0\).
  Since \(C\otimes L_p\)
  belongs to~\(\mathfrak{B}^G\),
  it follows that \(C\otimes L_p\)
  is \(\KK^G\)-equivalent
  to~\(0\).
  And this is equivalent to \(f\otimes \id_{L_p}\)
  being a \(\KK^G\)-equivalence.
\end{proof}

So the objects of~\(\mathfrak{B}^G\)
of the form~\(A\otimes L_p\)
are representatives for all \(p\)\nb-equivalence
classes in~\(\mathfrak{B}^G\).
The special property of these representatives is that
\[
F_*(A\otimes L_p) \otimes_\Z \Z[1/p] = 0.
\]
So \(p\)\nb-equivalence
classes of objects of~\(\mathfrak{B}^G\)
are in bijection with isomorphism classes of countable exact
\(\Kring\)\nb-modules~\(M\)
with \(M \otimes_\Z \Z[1/p] = 0\)
by Lemma~\ref{lem:p-equivalent}.  Inspecting our examples, we see that
the \(\Kring\)\nb-modules
in Examples \ref{exa:actions_on_Cuntz_3},
\ref{exa:actions_on_Cuntz_4}, \ref{exa:actions_on_Cuntz_6} have the
property that \(M \otimes_\Z \Z[1/p] = 0\).
So have the examples for \(p=2\)
built as in Example~\ref{exa:actions_on_Cuntz_10} from these examples.
It is also easy to see that the actions on~\(\mathcal{O}_\infty\)
for Examples \ref{exa:actions_on_Cuntz_1}
and~\ref{exa:actions_on_Cuntz_2} are \(p\)\nb-equivalent.
The criterion above does not directly apply here because \(M_1 = \Z\)
forces \(M\otimes_\Z \Z[1/p]\) to be non-trivial.

So when we classify actions of~\(\Z/p\)
only up to \(p\)\nb-equivalence,
then the number of cases to consider is reduced considerably.

\end{document}